\documentclass[smallextended,referee,envcountsect]{svjour3}
\smartqed
\usepackage{graphicx}
\usepackage{xcolor,color}

\usepackage{amsmath}
\usepackage{amssymb}
\usepackage{mathtools}
\usepackage{hyperref}
\usepackage{paralist}
\usepackage{color}
\usepackage{multicol} 
\usepackage{booktabs} 

\newcommand{\Id}{\mathbb{I}}
\newcommand{\R}{\mathbb{R}}
\newcommand{\Rext}{\R\cup\{+\infty\}}

\newcommand{\set}[1]{\left\{#1\right\}}
\newcommand{\sets}[1]{\{#1\}}
\newcommand{\norm}[1]{\left\Vert#1\right\Vert}
\newcommand{\norms}[1]{\Vert#1\Vert}
\newcommand{\iprod}[1]{\left\langle #1\right\rangle}
\newcommand{\iprods}[1]{\langle #1\rangle}
\newcommand{\Eproof}{\hfill $\square$}
\newcommand{\prox}{\mathrm{prox}}
\newcommand{\proj}{\mathrm{proj}}
\newcommand{\diag}[1]{\mathrm{diag}(#1)}

\newcommand{\dom}[1]{\mathrm{dom}(#1)}
\newcommand{\gra}[1]{\mathrm{gra}(#1)}
\newcommand{\ran}[1]{\mathrm{ran}(#1)}
\newcommand{\zer}[1]{\mathrm{zer}(#1)}
\newcommand{\mbf}[1]{\mathbf{#1}}
\newcommand{\mbb}[1]{\mathbb{#1}}
\newcommand{\mcal}[1]{\mathcal{#1}}

\newcommand{\Xc}{\mathcal{X}}
\newcommand{\Gc}{\mathcal{G}}
\newcommand{\Hc}{\mathcal{H}}
\newcommand{\Lc}{\mathcal{L}}

\newcommand{\Tc}{\mathcal{T}}

\newcommand{\Uc}{\mathcal{U}}
\newcommand{\Cc}{\mathcal{C}}
\newcommand{\Nc}{\mathcal{N}}
\newcommand{\Pc}{\mathcal{P}}
\newcommand{\Vc}{\mathcal{V}}

\newcommand{\intx}[1]{\mathrm{int}\left(#1\right)}

\newcommand{\BigO}[1]{\mathcal{O}\left(#1\right)}
\newcommand{\BigOs}[1]{\mathcal{O}\big(#1\big)}

\newcommand{\ul}[1]{\underline #1}

\newcommand{\mytbi}[1]{\textbf{\textit{#1}}}
\newcommand{\mytb}[1]{\textbf{#1}}
\newcommand{\myti}[1]{\textit{#1}}

\newcommand{\beforesec}{\vspace{-3ex}}
\newcommand{\aftersec}{\vspace{-2ex}}
\newcommand{\beforesubsec}{\vspace{-4ex}}
\newcommand{\aftersubsec}{\vspace{-2ex}}
\newcommand{\beforesubsubsec}{\vspace{-2.5ex}}
\newcommand{\aftersubsubsec}{\vspace{-2.5ex}}

\begin{document}

\title{Revisiting Extragradient-Type Methods -- Part 1: Generalizations and Sublinear Convergence Rates}
\titlerunning{Revisiting Extragradient-Type Methods -- Part 1: Non-accelerated Methods}


\author{Quoc Tran-Dinh\vspace{0.25ex}\\
\newline {Department of Statistics and Operations Research}\\
\newline The University of North Carolina at Chapel Hill\\
318 Hanes Hall, UNC-Chapel Hill, NC 27599-3260.\\
\newline \textit{Email:} \url{quoctd@email.unc.edu}.}

\author{Quoc Tran-Dinh \and Nghia Nguyen-Trung}

\institute{Quoc Tran-Dinh \and Nghia Nguyen-Trung \at
             Department of Statistics and Operations Research\\ 
             The University of North Carolina at Chapel Hill\\
             318 Hanes Hall, UNC-Chapel Hill, NC 27599-3260\\
             \textit{Corresponding author:} quoctd@email.unc.edu
}

\date{Received: date / Accepted: date}

\maketitle

\begin{abstract}
\small
This paper presents a comprehensive analysis of the well-known extragradient (EG) method for solving both equations and inclusions.
First, we unify and generalize EG for [non]linear equations to a wider class of algorithms, encompassing various existing schemes and potentially new variants.
Next, we analyze both sublinear ``best-iterate'' and ``last-iterate'' convergence rates for the entire class of algorithms, and derive  new convergence results for two well-known instances.
Second, we extend our EG  framework above to ``monotone'' inclusions, introducing a new class of algorithms and its corresponding convergence results.
Third, we also unify and generalize Tseng's forward-backward-forward splitting (FBFS) method to a broader class of algorithms to solve [non]linear inclusions when a weak-Minty solution exists, and establish its ``best-iterate'' convergence rate.
Fourth, to complete our picture, we also investigate sublinear rates of two other common variants of EG  using our EG analysis framework developed here: the reflected forward-backward splitting and the golden ratio methods. 
Finally, we conduct an extensive numerical experiment to validate our theoretical findings. 
Our results demonstrate that several new variants of our proposed algorithms outperform existing schemes in the majority of examples.
\end{abstract}


\section{Introduction}\label{sec:intro}
\aftersec
The \myti{generalized equation} (also called the \myti{[non]linear inclusion}) provides a unified template to model various problems in computational mathematics and related fields such as optimization problems (both unconstrained and constrained), minimax optimization, variational inequality, complementarity, game theory,  and fixed-point problems, see, e.g., \cite{Bauschke2011,reginaset2008,Facchinei2003,phelps2009convex,Rockafellar2004,Rockafellar1976b,ryu2016primer}.
This mathematical model has also found various direct applications  in operations research, economics, uncertainty quantification, and transportations, see, e.g., \cite{Ben-Tal2009,giannessi1995variational,harker1990finite,Facchinei2003,Konnov2001}.
Moreover, recent advancements in machine learning and robust optimization, particularly in generative adversarial networks (GANs) and adversarial training, have led to a surge of interest in minimax problems, a special case of generalized equations \cite{arjovsky2017wasserstein,Ben-Tal2009,goodfellow2014generative,levy2020large,madry2018towards,rahimian2019distributionally}. 
These models have also found applications in online learning and reinforcement learning, see, e.g.,   \cite{arjovsky2017wasserstein,azar2017minimax,bhatia2020online,goodfellow2014generative,jabbar2021survey,levy2020large,lin2022distributionally,madry2018towards,rahimian2019distributionally,wei2021last}.
Such prominent applications have motivated a new research trend for [non]linear inclusions in the last many years.

\vspace{0.5ex}
\noindent\textbf{Problem statement.}
In this paper, we consider the following \myti{generalized equation} (also known as a [composite] \myti{[non]linear inclusion}):
\begin{equation}\label{eq:NI}
\textrm{Find $x^{\star}\in\dom{\Phi }$ such that:} \quad 0 \in \Phi x^{\star} \equiv Fx^{\star} + Tx^{\star},
\tag{NI}
\end{equation}
where $F : \R^p\to\R^p$ is a single-valued mapping, $T : \R^p\rightrightarrows 2^{\R^p}$ is a set-valued (or multivalued) mapping from $\R^p$ to $2^{\R^p}$ (the set of all subsets of $\R^p$), $\Phi := F + T$, and $\dom{\Phi } := \dom{F}\cap\dom{T}$ is the domain of $\Phi$.  
Here, we will focus on the finite-dimensional Euclidean spaces $\R^p$.
However, it is worth noting that most results presented in this paper can be extended to Hilbert spaces as demonstrated in the existing literature.

\vspace{0.5ex}
\noindent\textbf{Special cases.}
We are also interested in various special cases of \eqref{eq:NI} as follows.
If $T = 0$, then \eqref{eq:NI} reduces to a \myti{[non]linear equation}:
\begin{equation}\label{eq:NE}
\textrm{Find $x^{\star}\in\dom{F}$ such that:} \quad Fx^{\star} = 0.
\tag{NE}
\end{equation}
If $T := \partial{g}$, the subdifferential of a proper, closed, and convex function $g$, then \eqref{eq:NI} reduces a \myti{mixed variational inequality problem} (MVIP) of the form:
\begin{equation}\label{eq:MVIP}
\textrm{Find $x^{\star}\in\R^p$ such that:} \ \iprods{Fx^{\star}, x - x^{\star}} + g(x) - g(x^{\star}) \geq 0, \ \forall x \in \R^p.
\tag{MVIP}
\end{equation}
In particular, if $T = \Nc_{\Xc}$, the normal cone of a nonempty, closed, and convex set $\Xc$ in $\R^p$ (i.e. $g = \delta_{\Xc}$, the indicator of $\Xc$), then \eqref{eq:MVIP} reduces the classical (Stampacchia's) \myti{variational inequality problem} (VIP):
\begin{equation}\label{eq:VIP}
\textrm{Find $x^{\star}\in\Xc$ such that:} \quad \iprods{Fx^{\star}, x - x^{\star}} \geq 0, \ \textrm{for all} \ x \in \Xc.
\tag{VIP}
\end{equation}
Another important special case of \eqref{eq:NI}  is the following minimax problem:
\begin{equation}\label{eq:minimax_prob}
\min_{u \in\R^m} \max_{v\in \R^n} \Big\{ \Lc(u, v) := \varphi(u) + \Hc(u, v) - \psi(v) \Big\},
\end{equation}
where $\varphi : \R^m\to\Rext$ and $\psi : \R^n\to\Rext$ are often proper, closed, and convex functions, and $\Hc : \R^m\times\R^n\to\R$ is a bifunction, often assumed to be differentiable, but not necessarily convex-concave.
If we denote $x := [u, v]$ as the concatenation of $u$ and $v$, and define $T := [\partial{\varphi}, \partial{\psi}]$ and $F := [\nabla_u{\Hc}(u, v), -\nabla_v{\Hc}(u, v)]$, then the optimality condition of \eqref{eq:minimax_prob} is exactly covered by \eqref{eq:NI} as a special case.

\vspace{1ex}
\noindent\textbf{Related work.}
Theory and solution methods for \eqref{eq:NI} and its special cases have been extensively studied for decades. 
Numerous monographs and articles have explored the existence of solutions, their properties, and various numerical methods.
Researchers have investigated \eqref{eq:NI} under a range of monotonicity assumptions, including monotone, quasi-monotone, pseudo-monotone, hypomonotone, and weakly monotone conditions. 
This has led to a rich body of literature, with key references including \cite{Bauschke2011,bauschke2020generalized,Bonnans2000a,Bonnans1994a,combettes2004proximal,Dontchev1996,Facchinei2003,Konnov2001,mordukhovich2006variational,Robinson1980,Rockafellar1997,Zeidler1984}.

Contemporary solution methods  for \eqref{eq:NI} and its special cases often rely on a fundamental assumption: \textit{maximal monotonicity} of $F$ and $T$, or of $\Phi$ to guarantee global convergences.
These methods essentially generalize  existing optimization algorithms such as gradient, proximal-point, Newton, and interior-point schemes to \eqref{eq:NI} and its special cases \cite{combettes2004proximal,Facchinei2003,Fukushima1992,Martinet1970,Peng1999,pennanen2002local,Rockafellar1976b,TranDinh2016c,vuong2018weak}, while leveraging the splitting structure of \eqref{eq:NI} to use individual operators defined on $F$ and $T$.
This approach leads to a class of operator splitting algorithms for solving \eqref{eq:NI} such as forward-backward splitting (FBS) and Douglas-Rachford (DRS) splitting schemes, as seen in \cite{Bauschke2011,Combettes2005,Davis2015,eckstein1992douglas,lin2020near,Lions1979}.
Alternatively, other approaches for \eqref{eq:NI} and its special cases rely on primal-dual, dual averaging, and mirror descent techniques, with notable works including \cite{Chambolle2011,Nemirovskii2004,Nesterov2007a}, and many recent works such as \cite{Chen2013a,chen2017accelerated,Cong2012,Davis2014,Esser2010a,He2012b,Nesterov2006d,TranDinh2015b,tran2019non,ZhuLiuTran2020}.

Unlike optimization and minimax problems, convergence analysis of gradient or forward-based methods for solving \eqref{eq:NI} faces a  fundamental challenge as \eqref{eq:NI} does not have an objective function, which plays a central role in constructing a Lyapunov or potential function to analyze convergence.
This drawback is even more critical when analyzing algorithms for solving nonmonotone instances of \eqref{eq:NI}. 
Additionally, unlike convex functions where strong properties such as coerciveness and cyclic monotonicity hold for their [sub]gradients beyond monotonicity, this is not the case for general monotone and Lipschitz continuous operators. 
This lack of a strong property results in gradient/forward-based methods being non-convergent, which limits their applicability in practice, see, e.g.,  \cite{Facchinei2003} for further details. 

To address this issue, the extragradient (EG) method was introduced by G. M. Korpelevich in 1976  \cite{korpelevich1976extragradient} and also by A. S. Antipin in \cite{antipin1976}. 
This method performs two sequential gradient/forward steps at each iteration, making it twice as expensive as the standard gradient method, but convergent under only the monotonicity  and the Lipschitz continuity of $F$. 
Since then, this method has been extended and modified in different directions to reduce its per-iteration complexity, including in certain nonmonotone settings, see \cite{alacaoglu2022beyond,censor2011subgradient,censor2012extensions,Iusem1997,khobotov1987modification,malitsky2015projected,malitsky2020forward,malitsky2014extragradient,Monteiro2010a,Monteiro2011,popov1980modification,solodov1999hybrid,solodov1999new,solodov1996modified,tseng1990further,tseng2000modified} for various representative examples.
Among these variants of EG, the past-extragradient scheme in  \cite{popov1980modification} and Teng’s forward-backward-forward splitting method in \cite{tseng2000modified} are the most notable ones. 
However, the results discussed in those are only applicable to the monotone setting of \eqref{eq:NI} and its special cases. 
Additionally, most of the convergence results discussed in the literature are asymptotic, leading to sublinear “best-iterate” convergence rates of the residual term associated with \eqref{eq:NI}. 
Under stronger assumptions such as ``strong monotonicity'' or error bound conditions, linear convergence rates can be achieved. 
Such types of convergence guarantees have been widely studied in the literature and are beyond the scope of this paper, see, e.g., \cite{Bauschke2011,Facchinei2003,Konnov2001} for concrete examples.
A non-exhaustive  summary of classical and recent results can be found in Table \ref{tbl:survey_results}. 

From an application perspective, motivated by recent machine learning and robust optimization applications, such as Generative Adversarial Networks (GANs), adversarial training, distributionally robust optimization, reinforcement learning, and online learning, efficient methods for solving minimax problems have become critically important and attractive. 
This is particularly true in nonconvex-nonconcave, large-scale, and stochastic settings, as evidenced in works such as \cite{arjovsky2017wasserstein,azar2017minimax,Ben-Tal2009,bhatia2020online,goodfellow2014generative,jabbar2021survey,levy2020large,lin2022distributionally,madry2018towards,rahimian2019distributionally}.
Several researchers have proposed and revisited EG and its variants, including \cite{bohm2022solving,daskalakis2018training,diakonikolas2021efficient,pethick2022escaping}. 
A notable work is due to \cite{diakonikolas2021efficient}, where the authors proposed an EG-plus (EG+) variant of EG, capable of handling nonmonotone instances of \eqref{eq:NE}, known as weak-Minty solutions. 
In \cite{pethick2022escaping}, this method was further extended to \eqref{eq:NI}, while \cite{bohm2022solving,luo2022last}  modified EG+ for Popov's methods, as well as optimistic gradient variants.

\begin{table}[ht!]
\vspace{-4ex}
\newcommand{\cell}[1]{{\!\!}{#1}{\!\!}}
\begin{center}
\caption{Summary of existing and our results in this paper and the most related references}\label{tbl:survey_results}
\vspace{-2ex}
\begin{small}
\resizebox{\textwidth}{!}{  
\begin{tabular}{|c|c|c|c|c|} \hline
\cell{Methods} & \cell{Assumptions} & \cell{Add. Assumptions} & \cell{Convergence Rates} & \cell{References} \\ \hline
\multicolumn{5}{|c|}{ For solving \eqref{eq:NE}} \\ \hline
\mytb{EG}/\mytb{EG+}/\mytb{FBFS} & $F$ is \mytb{wMs}  & None & $\BigOs{1/\sqrt{k}}$ best-iterate   & \cite{diakonikolas2021efficient,pethick2022escaping} \\ \hline
                                                                                                 & $F$ is \mytb{mono}  & None & $\BigOs{1/\sqrt{k}}$ last-iterate   & \cite{golowich2020last,gorbunov2022extragradient} \\ \hline
\mytb{PEG}/\mytb{OG}/\mytb{FRBS}/\mytb{RFBS}/\mytb{GR} & $F$ is \mytb{wMs} & $F$ is \mytb{chm}  & $\BigOs{1/\sqrt{k}}$ best-iterate & \cite{bohm2022solving,tran2022accelerated} \\ \hline
                                                                                                 & $F$ is \mytb{mono}  & $F$ is ``\mytb{coherent}'' & $\BigOs{1/\sqrt{k}}$ last-iterate   & \cite{mertikopoulos2019optimistic,mokhtari2020convergence} \\ \hline
{\color{blue}\mytb{GEG/GPEG}} & {\color{blue}$F$ is \mytb{wMs}} &  {\color{blue}$F$ is \mytb{chm}} & {\color{blue}$\BigOs{1/\sqrt{k}}$ best and last} & {\color{red}Ours$^{*}$} \\ \hline
\multicolumn{5}{|c|}{ For solving \eqref{eq:NI}, \eqref{eq:MVIP}, and \eqref{eq:VIP}} \\ \hline
\mytb{EG}/\mytb{EG+} & $\Phi$ is \mytb{mono} & $F$ is \mytb{mono}, $T = \Nc_{\Xc}$ & $\BigOs{1/\sqrt{k}}$ best and last &   \cite{cai2022tight} \\ \hline
{\color{blue}\mytb{GEG/GPEG}} & {\color{blue}$\Phi$ is \mytb{mono}} & {\color{blue}$F$ is \mytb{mono}, $T$ is \mytb{3-cm}} & {\color{blue}$\BigOs{1/\sqrt{k}}$ best and last} & {\color{red}Ours} \\ \hline
\mytb{FBFS}  & $\Phi$ is \mytb{wMs} & None & $\BigOs{1/\sqrt{k}}$ best-iterate & \cite{pethick2022escaping} \\ \hline
\mytb{OG}/\mytb{FRBS} & $\Phi$ is \mytb{mono} & None & $\BigOs{1/\sqrt{k}}$ best-iterate &   \cite{malitsky2020forward} \\ \hline
{\color{blue}\mytb{GFBFS/GFRBS}} & {\color{blue}$\Phi$ is \mytb{wMs}} & {\color{blue}None} & {\color{blue}$\BigOs{1/\sqrt{k}}$ best-iterate} & {\color{red}Ours$^{*}$} \\ \hline
\mytb{RFBS} & $F$ is \mytb{mono} & $T = \Nc_{\Xc}$ or $T$ is \mytb{mono} &   $\BigOs{1/\sqrt{k}}$ best-iterate &  \cite{cevher2021reflected,malitsky2015projected} \\ \hline
                      & $F$ is \mytb{mono} & $T = \Nc_{\Xc}$ &   $\BigOs{1/\sqrt{k}}$ last-iterate &  \cite{cai2022baccelerated} \\ \hline
{\color{blue}\mytb{RFBS}} & {\color{blue}$F$ is \mytb{mono}} & {\color{blue}$T$ is \mytb{mono}} &   {\color{blue}$\BigOs{1/\sqrt{k}}$ best and last} &  {\color{red}Ours} \\ \hline
\mytb{GR} & $F$ is \mytb{mono} & $T= \partial{g}$ &   $\BigOs{1/\sqrt{k}}$ best-iterate &  \cite{malitsky2019golden} \\ \hline
{\color{blue}\mytb{GR+}} & {\color{blue}$F$ is \mytb{mono}} & {\color{blue}$T$ is \mytb{$3$-cm}} &   {\color{blue}$\BigOs{1/\sqrt{k}}$ best-iterate} &  {\color{red}Ours} \\ \hline
\end{tabular}}
\end{small}
\\\vspace{1ex}
{\footnotesize
\textbf{Abbreviations:} 
\mytb{EG} $=$ extragradient; 
\mytb{PEG} $=$ past extragradient; 
\mytb{FBFS} $=$ forward-backward-forward splitting;
\mytb{OG} $=$ optimistic gradient;
\mytb{FRBS} $=$ forward-reflected-backward splitting;
\mytb{RFBS} $=$ reflected-forward-backward splitting;
\mytb{GR} $=$ golden ratio.
In addition, \mytb{wMs} $=$ weak-Minty solution;
\mytb{mono} $=$ monotone;
\mytb{chm} $=$ co-hypomonotone;
and \mytb{3-cm} $=$ $3$-cyclically monotone.
}
\end{center}
\vspace{-2ex}
{ {\color{red}$^{*}$} The last-iterate rates of \mytb{EG}/\mytb{EG+}/\mytb{FBFS}/\mytb{PEG}/\mytb{OG}/\mytb{FRBS}/\mytb{RFBS}/\mytb{GR} for \eqref{eq:NE} under the co-hypomonotonicity of $F$ were first proven in our unpublished report \cite{luo2022last}. }
\vspace{-4ex}
\end{table}

\vspace{0.5ex}
\noindent\textbf{Our approach and contribution.}
The EG method \cite{korpelevich1976extragradient} and its variants for solving \eqref{eq:NI} can be written as $x^{k+1} := \hat{J}_{\eta T}(x^k - \eta d^k)$, where $\eta > 0$ is a given stepsize, $d^k$ is a search direction, and $\hat{J}_{\eta T}$ is an approximation of the resolvent $J_{\eta T}$ of $\eta T$.
Hitherto, existing works primarily focus on:
\begin{compactitem}
\item[$\mathrm{(i)}$] Instantiating $d^k$ to obtain a specific variant and then studying convergence of such a scheme.
For example, if we choose $d^k := F(J_{\eta T}(x^k - \eta Fx^k))$, then we get the classical EG algorithm in \cite{korpelevich1976extragradient}, while if we set $d^k := F(y^k)$ with $y^k := J_{\eta T}(x^k - \eta Fy^{k-1})$, then we get Popov's past-EG scheme in \cite{popov1980modification}.
\item[$\mathrm{(ii)}$] Approximating $J_{\eta T}$ by a simpler operator $\hat{J}_{\eta T}$.
For instance, if $T = \delta_{\Xc}$, then we can replace $J_{\eta T}$ by the projection onto an appropriate half-space.  
\end{compactitem}
While $\mathrm{(ii)}$ has been intensively studied, e.g., in \cite{censor2011subgradient,Facchinei2003,Konnov2001,malitsky2014extragradient}, especially when $T = \Nc_{\Xc}$, and it is not our main focus in this paper, $\mathrm{(i)}$ covers a few non-unified schemes such as \cite{korpelevich1976extragradient,malitsky2015projected,malitsky2019golden,malitsky2020forward,popov1980modification,tseng2000modified} and is far from being comprehensive.
The latter raises the following question:
\begin{center}
\myti{Can we unify and generalize EG to cover broader classes of algorithms?}
\end{center}
In this paper, we attempt to tackle this question by unifying and generalizing EG to three broader classes of algorithms to solve both \eqref{eq:NE} and \eqref{eq:NI} as shown in Table~\ref{tbl:survey_results}.
Specifically, our contribution can be summarized as follows.
\begin{compactitem}
\item[$\mathrm{(a)}$]
First, we unify and generalize EG to a wider class of algorithms for solving \eqref{eq:NE}.
We prove both $\BigOs{1/\sqrt{k}}$-best-iterate and last-iterate convergence rates of this generalized scheme.
Then, we specify our results to cover various existing variants from the literature, where our results apply.

\item[$\mathrm{(b)}$]
Second, we unify and generalize EG to a broader class of schemes for solving \eqref{eq:NI}, and establish its $\BigOs{1/\sqrt{k}}$-best-iterate and last-iterate convergence rates under a ``monotonicity'' assumption.
Again, we  specify our method to cover: the standard EG scheme and Popov's past-EG algorithm.

\item[$\mathrm{(c)}$]
Third, we develop a generalization of Tseng's FBFS method, and provide a new unified $\BigOs{1/\sqrt{k}}$-best-iterate convergence rate for this scheme.
Our result covers both Tseng's classical FBFS algorithm and the FRBS (also equivalent to the optimistic gradient) method as its instances.

\item[$\mathrm{(d)}$]  
Fourth, we present a new best-iterate and last-iterate convergence analysis for the reflected forward-backward splitting (RFBS) methods to solve \eqref{eq:NI} under a ``monotonicity'' assumption.
Alternatively, we also analyze the best-iterate convergence rate of the golden ratio (GR) method in \cite{malitsky2019golden} by extending $T$ to a $3$-cyclically monotone operator and the range of the GR parameter $\tau$ from a fixed value  $\tau := \frac{1+\sqrt{5}}{2}$ to $1 < \tau < 1 + \sqrt{3}$.

\item[$\mathrm{(e)}$]
Finally, we implement our algorithms and their competitors  and perform an extensive test on different problems.
Our results show some promising improvement of the new methods over existing ones on various examples.
\end{compactitem}

\textit{Comparison.}
Since EG and its common variants in Table~\ref{tbl:survey_results} are classical, their convergence analysis is known and can be found in the literature.
Here, we only compare our results with the most recent and related work.

First, our generalized EG (called GEG) scheme for \eqref{eq:NE} covers both the classical EG \cite{korpelevich1976extragradient}  and Popov's past-EG \cite{popov1980modification} methods as instances, and allows to derive new variants.
As we will show in Subsection~\ref{subsec:EG4NE_scheme}, these two existing instances already cover almost known methods in the literature.
Nevertheless, our GEG has a flexibility to choose a search direction, called $u^k$, to potentially improve its performance over EG or past-EG (see Variant 3 in Subsection~\ref{subsec:EG4NE_scheme}).
Our convergence and convergence rate guarantees are new and covers many known results, including \cite{bohm2022solving,diakonikolas2021efficient,golowich2020last,gorbunov2022extragradient,gorbunov2022convergence,luo2022last,pethick2022escaping}.
Note that  the best-iterate rate has recently been obtained when a weak-Minty solution of \eqref{eq:NE} exists, see \cite{diakonikolas2021efficient,pethick2022escaping} for EG and  \cite{bohm2022solving,luo2022last}  for past-EG.
The last-iterate convergence rates for EG and past-EG have been recently proven in \cite{golowich2020last,gorbunov2022extragradient} for the monotone equation \eqref{eq:NE}, and in our unpublished report \cite{luo2022last} for the co-hypomonotone case (see also \cite{gorbunov2022convergence} for a refined analysis). 
Our analysis in this paper is rather elementary, unified, and thus different from those.

Second, existing convergence analysis for the EG method often focuses on \eqref{eq:VIP} and \eqref{eq:MVIP}.
In this paper, we extend the analysis for a class of \eqref{eq:NI}, where $T$ is $3$-cyclically monotone.
This class of operator is theoretically broader than $T = \partial{g}$.
Moreover, our generalization of EG in this setting covers different existing schemes, while also allows us to derive new variants. 
Our last-iterate convergence rate analysis is new and much broader than a special case in  \cite{cai2022tight}.

Third, again, our FBFS scheme is also broader and covers both Tseng's method \cite{tseng2000modified} and the FRBS method in \cite{malitsky2020forward} as specific instances.
Our best-iterate convergence rate analysis is for a weak-Minty solution, which is broader than those in \cite{malitsky2020forward,tseng2000modified}.
It also covers the results in  \cite{luo2022last,pethick2022escaping}.

Fourth, RFBS was proposed in \cite{malitsky2015projected} to solve \eqref{eq:VIP} and was extended to solve \eqref{eq:NI} in  \cite{cevher2021reflected}. 
The best-iterate rates were proven in these works, and the last-iterate rate of RFBS for solving \eqref{eq:VIP} has recently been proven in \cite{cai2022baccelerated}. 
Our result here is more general and covers these works as special cases.
Our analysis in both cases is also different from \cite{cevher2021reflected} and \cite{cai2022baccelerated}. 

Finally, we note that after our unpublished report \cite{luo2022last} and the first draft of this paper \cite{tran2023sublinear} were online, we find some concurrent and following-up  papers such as \cite{bohm2022solving,gorbunov2022convergence,pethick2023solving,pethick2024stable} that are related to our results presented in this paper, but we do not including them as existing results in Table~\ref{tbl:survey_results}.

\vspace{0.5ex}
\noindent\textbf{Paper outline.}
The rest of this paper is organized as follows. 
Section \ref{sec:background}  reviews basic concepts and related results used in this paper. 
Section \ref{sec:EG4NE}  generalizes and investigates the convergence rates of EG and its variants for solving \eqref{eq:NE}. 
Section \ref{sec:EG4NI} focuses on the EG method and its variants for solving \eqref{eq:NI}. 
Section \ref{sec:FBFS4NI} studies the FBFS method and its variants for solving \eqref{eq:NI}. 
Section \ref{sec:other_methods} provides a new convergence rate analysis of FBFS and OG for solving \eqref{eq:NI}. 
Section \ref{sec:numerical_experiments} presents several numerical examples to validate our theoretical results.

\beforesec
\section{Background and Preliminary Results}\label{sec:background}
\aftersec
We recall several concepts which will be used in this paper.
These concepts and properties are well-known and can be found, e.g., in \cite{Bauschke2011,Facchinei2003,Rockafellar2004,Rockafellar1970,ryu2016primer}.

\beforesubsec
\subsection{\bf Basic Concepts, Monotonicity, and Lipschitz Continuity}\label{subsec:basic_concepts}
\aftersubsec
We work with finite dimensional Euclidean spaces $\R^p$ and $\R^n$ equipped with standard inner product $\iprods{\cdot, \cdot}$ and Euclidean norm $\norms{\cdot}$.
For a multivalued mapping $T : \R^p \rightrightarrows 2^{\R^p}$, $\dom{T} = \set{x \in\R^p : Tx \not= \emptyset}$ denotes its domain, $\ran{T} := \bigcup_{x \in \dom{T}}Tx$ is its range, and $\gra{T} = \set{(x, y) \in \R^p\times \R^p : y \in Tx}$ stands for its graph, where $2^{\R^p}$ is the set of all subsets of $\R^p$.
The inverse mapping of $T$ is defined as $T^{-1}y := \sets{x \in \R^p : y \in Tx}$.
For a proper, closed, and convex function $f : \R^p\to\Rext$, $\dom{f} := \sets{x \in \R^p : f(x) < +\infty}$ denotes the domain of $f$, $\partial{f}$ denotes the subdifferential of $f$, and $\nabla{f}$ stands for the [sub]gradient of $f$.
For a symmetric matrix $\mbf{Q}$, $\lambda_{\min}(\mbf{Q})$ and $\lambda_{\max}(\mbf{Q})$ denote the smallest and largest eigenvalues of $\mbf{Q}$, respectively. 

\vspace{0.5ex}
\noindent\textbf{$\mathrm{(a)}$~Monotonicity.}
For a multivalued mapping $T : \R^p \rightrightarrows 2^{\R^p}$ and $\mu \in \R$, we say that $T$ is $\mu$-monotone if 
\begin{equation*}
\begin{array}{ll}
\iprods{u - v, x - y} \geq \mu\norms{x - y}^2, \quad \forall (x, u), (y, v)  \in \gra{T}.
\end{array}
\end{equation*}
If $T$ is single-valued, then this condition reduces to $\iprods{Tx - Ty, x - y} \geq \mu\norms{x - y}^2$ for all $x, y\in\dom{T}$.
If $\mu = 0$, then we say that $T$ is monotone.
If $\mu > 0$, then $T$ is  $\mu$-strongly monotone (or sometimes called coercive).  
If $\mu < 0$, then we say that $T$ is weakly monotone.
It is also called $\vert\mu\vert$-hypomonotone, see \cite{bauschke2020generalized}.
If $T = \partial{g}$, the subdifferential of a proper and convex function, then $T$ is also monotone.
If $g$ is $\mu$-strongly convex, then $T = \partial{g}$ is also $\mu$-strongly monotone.

We say that $T$ is $\rho$-comonotone if there exists $\rho \in \R$ such that 
\begin{equation*}
\begin{array}{ll}
\iprods{u - v, x - y} \geq \rho\norms{u  - v}^2, \quad \forall (x, u), (y, v)  \in \gra{T}.
\end{array}
\end{equation*}
If $\rho = 0$, then this condition reduces to the monotonicity of $T$.
If $\rho > 0$, then $T$ is called $\rho$-co-coercive. 
In particular, if $\rho = 1$, then $T$ is firmly nonexpansive.
If $\rho < 0$, then $T$ is called $\vert\rho\vert$-co-hypomonotone, see, e.g., \cite{bauschke2020generalized,combettes2004proximal}.
Note that a co-hypomonotone operator can also be nonmonotone.
As a concrete example, consider a linear mapping $Fx = Qx + q$ for a symmetric matrix $Q$ and a vector $q$.
If $Q$ is invertible then it is obvious to check that $\iprods{Fx - Fy, x - y} \geq \rho\norms{Fx - Fy}^2$ for all $x, y \in \R^p$, where $\rho := \lambda_{\min}(Q^{-1})$.
If $\rho < 0$, then $F$ is $\vert \rho\vert$-co-hypomonotone, but $F$ is not monotone since $Q$ is not positive semidefinite.

We say that $T$ is maximally $\mu$-monotone if $\gra{T}$ is not properly contained in the graph of any other $\mu$-monotone operator.
If $\mu = 0$, then we say that $T$ is maximally monotone.
These definitions are also extended to $\vert\rho\vert$-cohypomonotone operators.
For a proper, closed, and convex function $f : \R^p\to\Rext$, the subdifferential $\partial{f}$ of $f$ is maximally monotone. 

For a given mapping $T$ such that $\zer{T} := \set{x \in \dom{T} : 0 \in Tx} \neq\emptyset$,  we say that $T$ is star-monotone (respectively, $\mu$-star-monotone or $\rho$-star-comonotone (see \cite{loizou2021stochastic})) if for some $x^{\star} \in\zer{T}$, we have $\iprods{u, x - x^{\star}} \geq 0$ (respectively, $\iprods{u, x - x^{\star}} \geq \mu\norms{x - x^{\star}}^2$ or $\iprods{u, x - x^{\star}} \geq \rho\norms{u}^2$) for all $(x, u) \in \gra{T}$.
A solution $x^{\star} \in \zer{T}$ satisfying a $\rho$-star-comonotonicity is called a weak-Minty solution.
Clearly, if $T$ is monotone (respectively, $\mu$-monotone or $\rho$-co-monotone), then it is also star-monotone (respectively, $\mu$-star-monotone or $\rho$-star-comonotone).
However, the reverse statement does not hold in general.

\vspace{0.5ex}
\noindent\textbf{$\mathrm{(b)}$~Cyclic monotonicity.}
We say that a mapping $T$ is $m$-cyclically monotone ($m\geq 2$) if $\sum_{i=1}^m \iprods{u^i, x^i - x^{i+1}} \geq 0$ for all $(x^i, u^i) \in \gra{T}$ and $x_1 = x_{m+1}$ (see \cite{Bauschke2011}).
We say that $T$ is cyclically monotone if it is $m$-cyclically monotone for every $m \geq 2$.
If $T$ is $m$-cyclically monotone, then it is also $\hat{m}$-cyclically monotone for any $2 \leq \hat{m} \leq m$.
Since a  $2$-cyclically monotone operator $T$ is monotone, any $m$-cyclically monotone operator $T$ is $2$-cyclically monotone, and thus is also monotone.
An $m$-cyclically monotone operator $T$ is called maximally $m$-cyclically monotone if $\gra{T}$ is not properly contained into the graph of any other $m$-cyclically monotone operator. 

As proven in \cite[Theorem 22.18]{Bauschke2011} that $T$ is maximally cyclically monotone iff $T = \partial{f}$, the subdifferential of a proper, closed, and convex function $f$.
However, there exist maximally $m$-cyclically monotone operators (e.g., rotation linear operators) that are not the subdifferential $\partial{f}$ of a proper, closed, and convex function $f$, see, e.g., \cite{Bauschke2011}.
Furthermore, as indicated in  \cite[Example 2.16]{bartz2007fitzpatrick}, there exist maximally $3$-cyclically monotone operators that are not maximally monotone.

\vspace{0.5ex}
\noindent\textbf{$\mathrm{(c)}$~Lipschitz continuity and contraction.}
A single-valued mapping $F$ is called $L$-Lipschitz continuous if $\norms{Fx - Fy} \leq L\norms{x - y}$ for all $x, y \in\dom{F}$, where $L \geq 0$ is the Lipschitz constant. 
If $L = 1$, then we say that $F$ is nonexpansive, while if $L \in [0, 1)$, then we say that $F$ is $L$-contractive. 

\vspace{0.5ex}
\noindent\textbf{$\mathrm{(d)}$~Normal cone.}
Given a nonempty, closed, and convex set $\Xc$ in $\R^p$, the normal cone of $\Xc$ is defined as $\Nc_{\Xc}(x) := \sets{w \in \R^p : \iprods{w, x - y} \geq 0, \ \forall y\in\Xc}$ if $x\in\Xc$ and $\Nc_{\Xc}(x) = \emptyset$, otherwise.
If $f := \delta_{\Xc}$, the indicator of a convex set $\Xc$,   then we have $\partial{f} = \Nc_{\Xc}$.
Moreover, $J_{\partial{f}}$ reduces to the projection onto $\Xc$.

\vspace{0.5ex}
\noindent\textbf{$\mathrm{(e)}$~Resolvent and proximal operators.}
Given a multivalued mapping $T$, the operator $J_Tx := \set{y \in \R^p : x \in y + Ty}$ is called the resolvent of $T$, denoted by $J_Tx = (\Id + T)^{-1}x$, where $\Id$ is the identity mapping.
If $T$ is $\rho$-monotone with $\rho > -1$, then evaluating $J_T$ requires solving a strongly monotone inclusion $0 \in y - x + Ty$.
Hence, $J_T$ is well-defined and single-valued.
If $T = \partial{f}$, the subdifferential of proper, closed, and convex function $f$, then $J_T$ reduces to the proximal operator of $f$, denoted by $\prox_f$, which can be computed as $\prox_f(x) := \mathrm{arg}\min_{y}\sets{f(y) + (1/2)\norms{y-x}^2}$.
In particular, if $T = \Nc_{\Xc}$, the normal cone of a closed and convex set $\Xc$, then $J_T$ is the projection onto $\Xc$, denoted by $\proj_{\Xc}$.
If $T$ is maximally monotone, then $\mathrm{ran}(\Id + T) = \R^p$ (by Minty's theorem) and $T$ is firmly nonexpansive (and thus nonexpansive).

\beforesubsec
\subsection{\bf Best-Iterate and Last-Iterate Convergence Rates}\label{subsec:types_of_rates}
\aftersubsec
The results presented in this paper are related to two types of  sublinear convergence rates: the \mytbi{best-iterate} and the \mytbi{last-iterate} convergence rates.
To elaborate on these concepts, we assume that $\mathcal{D}$ is a given metric (e.g., $\norms{Fx^k}^2$ or $e(x^k)^2$ defined by \eqref{eq:res_norm} below) defined on an iterate sequence $\sets{x^k}$ generated by the underlying algorithm for solving \eqref{eq:NI} or its special cases.
For any $k \geq 0$ and a given order $\nu > 0$,  if
\vspace{-0.75ex}
\begin{equation*}
\min_{0 \leq l \leq k} \mathcal{D}(x^l) \leq \frac{1}{k+1}\sum_{l=0}^k \mathcal{D}(x^l) = \BigO{\frac{1}{k^{\nu}}},
\vspace{-0.75ex}
\end{equation*}
then we say that $\sets{x^k}$ has a $\BigO{1/k^{\nu}}$-best-iterate convergence rate.
In this case, we can take $\hat{x}_k :=  x_{k_{\min}}$ with $k_{\min} := \mathrm{arg}\min_{0\leq l \leq k}\mathcal{D}(x^l)$ as the ``best'' output of the underlying algorithm. 

If we instead have $\mathcal{D}(x^k) =  \BigO{\frac{1}{k^{\nu}}}$ with $x^k$ being the $k$-th iterate generated by the algorithm, then we say that $\sets{x^k}$ has a $\BigO{1/k^{\nu}}$ last-iterate convergence rate.
We emphasize that the convergence on the metric $\mathcal{D}$ of $\sets{x^k}$ does not generally imply the convergence of $\sets{x^k}$ itself, especially when we characterize the rate of convergence in different metrics. 

\beforesubsec
\subsection{\bf Exact Solutions and Approximate Solutions}\label{subsec:exact_and_approx_sols}
\aftersubsec
There are different metrics to characterize exact and approximate solutions of \eqref{eq:NI}.
The most obvious one is the residual norm of $\Phi$, which is defined as
\begin{equation}\label{eq:res_norm}
e(x) := \min_{\xi \in Tx}\norms{Fx + \xi}, \quad x \in \dom{\Phi}.
\end{equation}
Clearly, if $e(x^{\star}) = 0$ for some $x^{\star}\in\dom{\Phi}$, then $x^{\star} \in \zer{\Phi}$, a solution of \eqref{eq:NI}.
If $T = 0$, then $e(x) = \norms{Fx}$.
However, if $e(\hat{x}) \leq \epsilon$ for a given tolerance $\epsilon > 0$, then $\hat{x}$ can be considered as an $\epsilon$-approximate solution of \eqref{eq:NI}.
The algorithms presented in this paper use this metric  to characterize an approximate solution of \eqref{eq:NI} and its special cases.

Other metrics often used for monotone \eqref{eq:VIP} are gap functions and restricted gap functions \cite{Facchinei2003,Konnov2001,Nesterov2007a}, which are respectively defined as follows:
\begin{equation}\label{eq:gap_func}
\mcal{G}(x) := \max_{y \in \Xc}\iprods{Fy, y - x} \quad \text{and} \quad \mcal{G}_{\mbb{B}}(x) := \max_{y\in\Xc\cap \mbb{B}}\iprods{Fy, y  - x},
\end{equation}
where $\mbb{B}$ is a given nonempty, closed, and bounded convex set.
Note that, under the monotonicity, $\mcal{G}(x) \geq 0$ for all $x\in\Xc$, and $\mcal{G}(x^{\star}) = 0$ iff $x^{\star}$ is a solution of \eqref{eq:VIP}.
Thus $\tilde{x}$ is an $\epsilon$-approximate solution of \eqref{eq:VIP} if $\mcal{G}(\tilde{x}) \leq \epsilon$.

For the restricted gap function $\mcal{G}_{\mbb{B}}$, if $x^{\star}$ is a solution of \eqref{eq:VIP} and $x^{\star} \in \mbb{B}$, then $\mcal{G}_{\mbb{B}}(x^{\star}) = 0$.
Conversely, if $\mcal{G}_{\mbb{B}}(x^{\star}) = 0$ and $x^{\star} \in \intx{\mbb{B}}$, the interior of $\mbb{B}$, then $x^{\star}$ is a solution of \eqref{eq:VIP} in $\mbb{B}$ (see \cite[Lemma 1]{Nesterov2007a}).
Gap functions have widely been used in the literature to characterize approximate solutions for many numerical methods, see, e.g., \cite{chen2017accelerated,Cong2012,Facchinei2003,Konnov2001,Nemirovskii2004,Nesterov2007a} as concrete examples.
Note that for nonmonotone cases, gap functions are not applicable in general.

If $J_{\eta T}$ is well-defined and single-valued for some $\eta > 0$, and $F$ is single-valued, then we can use  the following forward-backward splitting residual:
\begin{equation}\label{eq:FB_residual}
\Gc_{\eta }x := \tfrac{1}{\eta}\left(x - J_{\eta T}(x - \eta Fx)\right), 
\end{equation}
to characterize solutions of \eqref{eq:NI}.
It is clear that $\Gc_{\eta}x^{\star} = 0$ iff $x^{\star} \in \zer{\Phi}$.
In addition, if $J_{\eta T}$ is firmly nonexpansive, then we also have 
\begin{equation}\label{eq:FBR_bound2}
\norms{\Gc_{\eta }x} \leq \norms{Fx + \xi}, \quad (x, \xi) \in \gra{T}.
\end{equation}
Hence, for a given tolerance $\epsilon > 0$, if $\norms{\Gc_{\eta}\tilde{x}} \leq \epsilon$, then we can say that $\tilde{x}$ is an $\epsilon$-approximate solution of \eqref{eq:NI}.
If $T := \Nc_{\Xc}$, i.e. \eqref{eq:NI} reduces to \eqref{eq:VIP}, then, with $\eta = 1$, $\Gc_{1}x$ reduces to the classical natural map $\Pi_{F,\Xc}x = x - \proj_{\Xc}(x - Fx)$ of \eqref{eq:VIP}, and $r_n(x) := \norms{\Gc_{1}x} = \norms{\Pi_{F,\Xc}x}$ is the corresponding natural residual at $x$, see \cite{Facchinei2003}.
From \eqref{eq:FBR_bound2}, we have $r_n(x) \leq \norms{Fx + \xi}$ for any $\xi \in \Nc_{\Xc}(x)$.

\beforesubsec
\subsection{\bf The Forward-Type or Fixed-Point Methods}\label{subsec:GD_FW_methods}
\aftersubsec
Let us briefly recall the gradient/forward or fixed-point scheme for solving \eqref{eq:NE} as follows.
Starting from $x^0 \in\dom{F}$, at each iteration $k$, we update
\begin{equation}\label{eq:FW4NE}
x^{k+1} := x^k - \eta Fx^k,
\tag{FW}
\end{equation}
where $\eta > 0$ is a given constant stepsize.
The forward scheme is a classical method to solve \eqref{eq:NE}.
This scheme covers the gradient descent method for minimizing $f(x)$ as a special case with $F = \nabla{f}$.
If $G := \Id - \eta F$, then \eqref{eq:FW4NE} becomes the well-known \textbf{fixed-point iteration} to approximate a fixed-point of $G$.
We recall this method here to compare with EG methods.

If $F$ is $\rho$-co-coercive and $0 < \eta < \rho$, then $\sets{x^k}$ converges to $x^{\star}\in\zer{F}$ (see, e.g., \cite{Facchinei2003}).
Otherwise, if $F$ is only monotone and $L$-Lipschitz continuous, then there exist examples (e.g., $Fx = [x_2, -x_1]$) showing that \eqref{eq:FW4NE} is divergent for any choice of stepsize $\eta$ even if $F$ is monotone. 

To solve \eqref{eq:NI}, we can instead apply the forward-backward splitting method as follows.
Starting from $x^0 \in\dom{F}$, at each iteration $k \geq 0$, we update
\begin{equation}\label{eq:FBS4NI}
x^{k+1} := J_{\eta T}(x^k - \eta Fx^k),
\tag{FBS}
\end{equation}
where $\eta > 0$ is a given constant stepsize.
Similar to \eqref{eq:FW4NE}, if $F$ is $\rho$-co-coercive and $T$ is maximally monotone, then with $\eta \in (0, \rho)$, $\sets{x^k}$ generated by \eqref{eq:FBS4NI} converges to $x^{\star}\in\zer{\Phi}$.
If $F = \nabla{f}$, the gradient of a convex  and $L$-smooth function $f$, then $F$ is co-coercive.
However, imposing the co-coerciveness for a general mapping $F$ is often restrictive. 
Hence, both \eqref{eq:FW4NE} and \eqref{eq:FBS4NI} are less practical in applications due to the absence of the co-coerciveness. 

\beforesec
\section{A Class of Extragradient-Type Methods for Nonlinear Equations}\label{sec:EG4NE}
\aftersec
In this section, we unify and generalize EG to a wider class of algorithms for solve \eqref{eq:NE}.
Then, we analyze its convergence rates and study its special cases.

\beforesubsec
\subsection{\mytb{A Class of Extragradient Methods for Equations}}\label{subsec:EG4NE_scheme}
\aftersubsec
Our class of extragradient methods for solving \eqref{eq:NE} is presented as follows.
Starting from an initial point $x^0 \in\dom{F}$, at each iteration $k \geq 0$, we update 
\begin{equation}\label{eq:EG4NE}
\arraycolsep=0.2em
\left\{\begin{array}{lcl}
y^k &:= & x^k - \frac{\eta}{\beta} u^k, \vspace{1ex}\\
x^{k+1} &:= & x^k - \eta Fy^k,
\end{array}\right.
\tag{GEG}
\end{equation}
where $\eta > 0$ is a given constant stepsize, $\beta \in (0, 1]$ is a scaling factor, and $u^k \in \R^p$ satisfies the following condition:
\begin{equation}\label{eq:EG4NE_u_cond}
\norms{Fx^k - u^k}^2 \leq  \kappa_1 \norms{Fx^k - Fy^{k-1}}^2 + \kappa_2\norms{Fx^k - Fx^{k-1}}^2,
\end{equation}
for given  constants $\kappa_1 \geq 0 $ and $\kappa_2 \geq 0$ and $x^{-1} = y^{-1} := x^0$.
Under different choices of $u^k$, \eqref{eq:EG4NE} covers a wide range of methods generalizing the \textbf{extragradient method} (EG) in \cite{korpelevich1976extragradient}.
Note that condition \eqref{eq:EG4NE_u_cond} also makes \eqref{eq:EG4NE} different from the hybrid extragradient method in, e.g., \cite{Monteiro2011,solodov1999hybrid}.
If we define $d^k := F(x^k - \frac{\eta}{\beta} u^k)$, then \eqref{eq:EG4NE} becomes $x^{k+1} := x^k - \eta d^k$, which falls into a generic iterative scheme we have discussed earlier.

Let us consider the following three choices of $u^k$ which  fulfill \eqref{eq:EG4NE_u_cond}.
\begin{itemize}
\itemsep=0.2em
\item[(a)]\textbf{Variant 1: Extragradient method.} 
If we choose $u^k := Fx^k$, then we obtain the \mytb{extragradient} (EG) scheme \cite{korpelevich1976extragradient} from \eqref{eq:EG4NE} for solving \eqref{eq:NE}.
Clearly, $u^k$ satisfies \eqref{eq:EG4NE_u_cond}  with  $\kappa_1 = \kappa_2 = 0$.

\item[(b)]\textbf{Variant 2: Past-extragradient method.} 
If  we choose $u^k := Fy^{k-1}$, then we obtain the Popov's \mytb{past-extragradient} method \cite{popov1980modification} to solve \eqref{eq:NE}. 
This scheme is also equivalent to the \mytb{optimistic gradient} method in the literature, see also \cite{daskalakis2018training,mertikopoulos2019optimistic,mokhtari2020convergence} for more details.
Clearly, $u^k$ satisfies \eqref{eq:EG4NE_u_cond} with $\kappa_1 = 1$ and $\kappa_2 = 0$.

\item[(c)]\textbf{Variant 3: Generalization.} 
We can construct $u^k := \alpha_1Fx^k + \alpha_2Fy^{k-1} + (1-\alpha_1 - \alpha_2)Fx^{k-1}$ as an affine combination of $Fx^k$, $Fy^{k-1}$ and $Fx^{k-1}$ for given constants $\alpha_1, \alpha_2 \in \R$.
Then, $u^k$ satisfies \eqref{eq:EG4NE_u_cond} with $\kappa_1 = (1 + c)\alpha_2^2$ and $\kappa_2 = (1+c^{-1})(1-\alpha_1-\alpha_2)^2$ by Young's inequality for any $c > 0$.
This variant essentially has the same per-iteration complexity as \textbf{Variant 1}, but it covers all candidates as combinations of $x^k$, $y^{k-1}$, and $x^{k-1}$.
\end{itemize}
Figure~\ref{fig:EG_illustration} illustrates the three variants of \eqref{eq:EG4NE} presented above.
\begin{figure}[hpt!]
\vspace{-3ex}
\centering
\includegraphics[width=\linewidth]{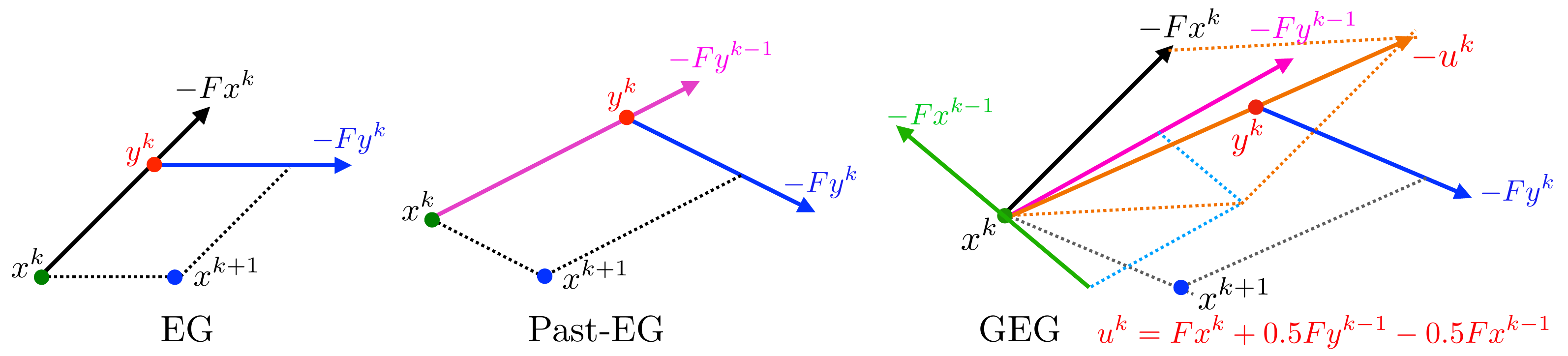}
\caption{
An illustration of three variants of \eqref{eq:EG4NE}: \mytb{Variant 1} (Left), \mytb{Variant 2} (Middle), and \mytb{Variant 3}  with $u^k := Fx^k + 0.5 Fy^{k-1} - 0.5Fx^{k-1}$ for $\alpha_1 = 1$ and $\alpha_2 = 0.5$ (Right).
}
\label{fig:EG_illustration}
\vspace{-4ex}
\end{figure}

\vspace{0.75ex}
\noindent\textbf{$\mathrm{(i)}$~Discussion of the extragradient method.}
With $u^k := Fx^k$,  if $\beta = 1$, then we obtain exactly the \myti{classical EG method}  \cite{korpelevich1976extragradient} for solving \eqref{eq:NE}.
If $\beta < 1$, then we recover the \mytb{extragradient-plus (EG$+$)} scheme from \cite{diakonikolas2021efficient} for solving \eqref{eq:NE}.
If we compute $x^k = y^k + \frac{\eta}{\beta} Fx^k$ from the first line of \eqref{eq:EG4NE} and substitute it into the second line of \eqref{eq:EG4NE}, then we get $x^{k+1} = y^k - \eta(Fy^k - \frac{1}{\beta}Fx^k)$.
In this case, we obtain  a \mytb{forward-backward-forward splitting} variant of Tseng's method in \cite{tseng2000modified} from \eqref{eq:EG4NE} as follows:
\begin{equation}\label{eq:FBFS4NE}
\arraycolsep=0.2em
\left\{\begin{array}{lcl}
y^k &:= & x^k - \frac{\eta}{\beta} Fx^k, \vspace{1ex}\\
x^{k+1} &:= & y^k - \eta (Fy^k - \frac{1}{\beta}Fx^k).
\end{array}\right.
\tag{FBFS}
\end{equation}
Clearly, if $\beta = 1$, then we recover  exactly \mytb{Tseng's method} for solving \eqref{eq:NE}.

\vspace{0.75ex}
\noindent\textbf{$\mathrm{(ii)}$~Discussion of the past-extragradient method and its relations to other methods.}
With $u^k := Fy^{k-1}$, we can show that it is equivalent to the following variants.
First, we can rewrite \eqref{eq:EG4NE} as
\begin{equation}\label{eq:PEG4NE}
\arraycolsep=0.2em
\left\{\begin{array}{lcl}
x^{k+1} & := &  x^k - \eta Fy^k \vspace{1ex}\\
y^{k+1} & := & x^{k+1} - \tfrac{\eta}{\beta}Fy^k.
\end{array}\right.
\tag{PEG}
\end{equation}
This form shows us that \eqref{eq:PEG4NE} saves one evaluation $Fx^k$ of $F$ at each iteration compared to \textbf{Variant 1}.
If $\beta = 1$, then we obtain exactly the \mytb{Popov's method} in  \cite{popov1980modification}.
If we rotate up the second line and use $\beta = 1$ as $y^k = x^k - \eta Fy^{k-1}$, then we get the \mytb{past-extragradient method}.

Now, under this choice of $u^k$, from the first line of \eqref{eq:EG4NE}, we have $x^k = y^k + \frac{\eta}{\beta}u^k = y^k + \frac{\eta}{\beta}Fy^{k-1}$.
Substituting this expression into the first line of \eqref{eq:PEG4NE}, we get $x^{k+1} = y^k - \eta Fy^k + \frac{\eta}{\beta}Fy^{k-1}$.
Substituting this relation into the second line of \eqref{eq:PEG4NE}, we can eliminate $x^{k+1}$ to get the following variant:
\begin{equation}\label{eq:FRBS4NE} 
y^{k+1} := y^k - \tfrac{\eta}{\beta}\big( (1 + \beta) Fy^k  - Fy^{k-1}\big).
\tag{FRBS}
\end{equation}
This scheme can be considered as a simplified variant of the \mytb{forward-reflected-backward splitting} scheme in \cite{malitsky2020forward} for solving \eqref{eq:NE} when we set $\beta := 1$ as $y^{k+1} := y^k - \eta ( 2Fy^k  - Fy^{k-1} )$.

Alternatively, from \eqref{eq:PEG4NE}, we have $x^{k-1} - x^k =  \eta Fy^{k-1}$ and $\beta (x^k - y^k) = \eta Fy^{k-1}$, leading to $x^{k-1} - x^k = \beta (x^k - y^k)$.
Therefore, we get $y^k = \frac{1}{\beta}( (1+\beta)x^k - x^{k-1})$.
Substituting this expression into the first line of \eqref{eq:PEG4NE}, we get
\vspace{-0.5ex}
\begin{equation}\label{eq:RGD4NE} 
x^{k+1} := x^k - \eta F\big( \tfrac{1}{\beta}( (1+\beta)x^k - x^{k-1}) \big).
\tag{RFB}
\vspace{-0.5ex}
\end{equation}
In particular, if $\beta = 1$, then we obtain $x^{k+1} := x^k - \eta F(2x^k - x^{k-1})$, which turns out to be the \mytb{reflected gradient} method in \cite{malitsky2015projected} or the \mytb{reflected forward-backward splitting} scheme in \cite{cevher2021reflected} for solving \eqref{eq:NE}.

Using the relation $x^{k-1} - x^k = \beta (x^k - y^k)$ above, we can compute that $x^{k} = \frac{\beta}{1+\beta}y^k + \frac{1}{1+\beta}x^{k-1} = \frac{(\omega-1)}{\omega}y^k + \frac{1}{\omega}x^{k-1}$, where $\omega := 1 + \beta$.
Combining the two lines of \eqref{eq:EG4NE}, we get $y^{k+1} :=  x^{k+1} - \tfrac{\eta}{\beta}Fy^k = x^k - \frac{\eta(1 + \beta)}{\beta}Fy^k$.
Putting both expressions together, we arrive at
\vspace{-0.5ex}
\begin{equation}\label{eq:GR4NE} 
\arraycolsep=0.2em
\left\{\begin{array}{lcl}
x^k  & := & \tfrac{(\omega-1)}{\omega}y^k + \tfrac{1}{\omega}x^{k-1}, \vspace{1ex}\\
y^{k+1}  & := &  x^{k} - \tfrac{\eta(1+\beta)}{\beta}Fy^k.
\end{array}\right.
\tag{GR}
\vspace{-0.5ex}
\end{equation}
This method is a simplified variant of the \mytb{golden ratio} method in \cite{malitsky2019golden} for solving \eqref{eq:NE}.
Overall, the template \eqref{eq:EG4NE} covers a class of EG algorithms with many common instances as discussed above. 

\vspace{0.75ex}
\noindent\textbf{$\mathrm{(iii)}$~Discussion of the generalization case.}
Suppose that $u_{\mathrm{GEG}}^k := \alpha_1Fx^k + \alpha_2Fy^{k-1} + (1-\alpha_1 - \alpha_2)Fx^{k-1}$ for fixed $\alpha_1, \alpha_2 \in \R$.

Let us define $v^k := Fy^k - Fx^k$, $\bar{v}^k := Fy^k - Fy^{k-1}$, and $\hat{v}^k := Fy^k - Fx^{k-1}$.
Then, we can compute $\norms{Fy^k - u^k_{\mathrm{GEG}}}^2 = \norms{\alpha_1v^k + \alpha_2\bar{v}^k + (1-\alpha_1-\alpha_2)\hat{v}^k}^2$.
Now, we compare \eqref{eq:EG4NE}  and EG in \textbf{Variant 1} and  PEG in \textbf{Variant 2}.
\begin{itemize}
\item For EG, we choose $u_{\textrm{EG}}^k := Fx^k$, and hence, $\norms{Fy^k - u_{\mathrm{EG}}^k}^2 = \norms{Fy^k - Fx^k}^2 = \norms{v^k}^2$.
If there exist $\alpha_1, \alpha_2 \in \R$ such that
\vspace{-0.5ex}
\begin{equation*} 
 \norms{v^k} > \norms{\alpha_1v^k + \alpha_2\bar{v}^k + (1-\alpha_1-\alpha_2)\hat{v}^k},
 \vspace{-0.5ex}
\end{equation*}
then, by Lemma~\ref{le:EG4NE_key_estimate1} below,  we can easily show that at each iteration $k$, $\norms{x^{k+1} - x^{\star}}^2$ decreases faster in GEG than in EG.
\item For PEG, we choose $u_{\mathrm{PEG}}^k := Fy^{k-1}$, and hence, $\norms{Fy^k - u^k_{\mathrm{PEG}}}^2 = \norms{\bar{v}^k}^2$.
If there exist $\alpha_1, \alpha_2 \in \R$ such that
\vspace{-0.5ex}
\begin{equation*} 
 \norms{\bar{v}^k} > \norms{\alpha_1v^k + \alpha_2\bar{v}^k + (1-\alpha_1-\alpha_2)\hat{v}^k},
 \vspace{-0.5ex}
\end{equation*}
then, again by Lemma~\ref{le:EG4NE_key_estimate1},  we can easily show that at each iteration $k$, $\norms{x^{k+1} - x^{\star}}^2$ decreases faster in GEG than in PEG.
\end{itemize}
Consequently, our \eqref{eq:EG4NE} scheme gives more freedom to improve the performance compared to existing EG variants as we will see in Section~\ref{sec:numerical_experiments}.

\beforesubsec
\subsection{\bf Key Estimates for Convergence Analysis}\label{subsec:EG4NE_convergence}
\aftersubsec
To analyze the convergence of \eqref{eq:EG4NE}, we first prove the following lemmas.

\begin{lemma}\label{le:EG4NE_key_estimate1}
If $\sets{(x^k, y^k)}$ is generated by \eqref{eq:EG4NE}, then for any $\gamma > 0$ and any $\hat{x}\in\dom{F}$, we have
\begin{equation}\label{eq:EG4NE_key_est1}
\arraycolsep=0.2em
\begin{array}{lcl}
\norms{x^{k+1} - \hat{x}}^2 & \leq & \norms{x^k - \hat{x}}^2 - \beta\norms{y^k - x^k}^2 +  \tfrac{\eta^2}{\gamma}\norms{Fy^k - u^k}^2 \vspace{1ex}\\
&& - {~}  2\eta\iprods{Fy^k, y^k - \hat{x}} - (\beta - \gamma)\norms{x^{k+1} - y^k}^2 \vspace{1ex}\\
&& - {~}  (1 - \beta)\norms{x^{k+1} - x^k}^2.
\end{array}
\end{equation}
\end{lemma}

\begin{proof}
First, for any $\hat{x}\in\dom{F}$, using $x^{k+1} - x^k = -\eta Fy^k$ from the second line of \eqref{eq:EG4NE}, we have
\begin{equation*} 
\arraycolsep=0.2em
\begin{array}{lcl}
\norms{x^{k+1} - \hat{x}}^2 &= & \norms{x^k - \hat{x}}^2 + 2\iprods{x^{k+1} - x^k, x^{k+1} - \hat{x}} - \norms{x^{k+1} - x^k}^2 \vspace{1ex}\\
&= & \norms{x^k - \hat{x}}^2 - 2\eta\iprods{Fy^k, x^{k+1} - \hat{x}} -  \norms{x^{k+1} - x^k}^2.
\end{array}
\end{equation*}
Next, using  $\eta u^k = \beta(x^k - y^k)$ from the first line of \eqref{eq:EG4NE}, 
the Cauchy-Schwarz inequality, the identity $2\iprods{x^{k+1} - y^k, x^k - y^k} = \norms{x^k - y^k}^2 + \norms{x^{k+1} - y^k}^2 - \norms{x^{k+1} - x^k}^2$, and an elementary inequality $2\iprods{w, z} \leq \gamma\norms{w}^2 + \frac{\norms{z}^2}{\gamma}$  for any $\gamma > 0$ and any vectors $w$ and $z$, we can  derive that
\begin{equation*} 
\arraycolsep=0.2em
\begin{array}{lcl}
2\eta\iprods{Fy^k, x^{k+1} - \hat{x}} & = & 2\eta\iprods{Fy^k, y^k - \hat{x}} + 2\eta\iprods{Fy^k - u^k, x^{k+1} - y^k} \vspace{1ex}\\
&& + {~} 2\eta\iprods{u^k, x^{k+1} - y^k} \vspace{1ex}\\
&\geq & 2\eta\iprods{Fy^k, y^k - \hat{x}} - \frac{\eta^2}{\gamma}\norms{Fy^k - u^k}^2 - \gamma \norms{x^{k+1} - y^k}^2 \vspace{1ex}\\
&& + {~} \beta\big[ \norms{x^k - y^k}^2 + \norms{x^{k+1} - y^k}^2 - \norms{x^{k+1} - x^k}^2\big] \vspace{1ex}\\
&= & 2\eta\iprods{Fy^k, y^k - \hat{x}} + \beta\norms{y^k - x^k}^2  - \frac{\eta^2}{\gamma}\norms{Fy^k - u^k}^2 \vspace{1ex}\\
&& + {~} (\beta - \gamma)\norms{x^{k+1} - y^k}^2 - \beta\norms{x^{k+1} - x^k}^2.
\end{array}
\end{equation*}
Finally, combining the last two expressions, we obtain \eqref{eq:EG4NE_key_est1}.
\Eproof
\end{proof}

\begin{lemma}\label{le:EG4NE_key_estimate1b}
Suppose that $\zer{F} \neq\emptyset$, $F$ is $L$-Lipschitz continuous,  and there exist $x^{\star} \in \zer{F}$ and $\rho \geq 0$ such that $\iprods{Fx, x - x^{\star}} \geq -\rho\norms{Fx}^2$ for all $x\in\dom{F}$.
Let $\sets{(x^k, y^k)}$ be generated by \eqref{eq:EG4NE} for a given $u^k$ satisfying \eqref{eq:EG4NE_u_cond}.
For any $\gamma > 0$ and $r > 0$, let us introduce the following function:
\begin{equation}\label{eq:EG4NE_Lyapunov_func}
\hspace{-1ex}
\arraycolsep=0.1em
\begin{array}{lcl}
\Pc_k &:= & \norms{x^k - x^{\star} }^2 +  \tfrac{\kappa_1(1+r)L^2\eta^2}{r \gamma}\norms{x^k - y^{k-1}}^2 +  \tfrac{\kappa_2(1+r)L^2\eta^2}{r \gamma}\norms{x^k - x^{k-1}}^2.
\end{array}
\hspace{-1ex}
\end{equation}
Then, for any  $s > 0$ and $\mu \in [0, 1]$, we have
\begin{equation}\label{eq:EG4NE_key_estimate1b}
\hspace{-2ex}
\arraycolsep=0.2em
\begin{array}{lcl}
\Pc_{k+1} & \leq & \Pc_k - \left( \beta - \frac{(1 + r)L^2\eta^2}{ \gamma} - \frac{2\mu\rho(1 + s)}{s \eta}  \right) \norms{y^k - x^k}^2 \vspace{1ex} \\
&& - {~} \left(\beta - \gamma - \frac{\kappa_1(1 + r) L^2\eta^2}{r \gamma} - \frac{2\mu\rho(1 + s)}{\eta}  \right)\norms{x^{k+1} - y^k}^2 \vspace{1ex}\\
&& - {~} \big(1 - \beta - \frac{\kappa_2(1 + r) L^2\eta^2}{r \gamma} - \frac{2(1-\mu)\rho}{\eta} \big)\norms{x^{k+1} - x^k}^2.
\end{array}
\hspace{-2ex}
\end{equation}
\end{lemma}

\begin{proof}
First, since $\iprods{Fx, x - x^{\star}} \geq -\rho\norms{Fx}^2$ for all $x\in\dom{F}$, using this condition with $x := y^k$, and then utilizing the second line of \eqref{eq:EG4NE} and Young's inequality, for any $s > 0$ and $\mu\in [0, 1]$, we can derive that
\begin{equation}\label{eq:EG4NE_lm32_proof1}
\arraycolsep=0.2em
\begin{array}{lcl}
\iprods{Fy^k, y^k - x^{\star}} & \geq & -\rho\norms{Fy^k}^2 \overset{\tiny \eqref{eq:EG4NE}}{ = } -\frac{\rho}{\eta^2}\norms{x^{k+1} - x^k}^2 \vspace{1ex}\\
& \geq  & -\frac{\mu\rho(1 + s)}{\eta^2}\norms{x^{k+1} - y^k}^2 - \frac{\mu\rho(1+s)}{s\eta^2}\norms{y^k - x^k}^2 \vspace{1ex}\\
&& - {~} \frac{(1-\mu)\rho}{\eta^2}\norms{x^{k+1} - x^k}^2.
\end{array}
\end{equation}
Next,  for any $r > 0$, it follows from Young's inequality, \eqref{eq:EG4NE_u_cond}, and the $L$-Lipschitz continuity of $F$ that
\begin{equation}\label{eq:EG4NE_lm32_proof2}
\hspace{-2ex}
\arraycolsep=0.1em
\begin{array}{lcl}
\norms{Fy^k - u^k}^2 &\leq & (1+r)\norms{Fx^k - Fy^k}^2 +  \frac{(1+r)}{r} \norms{Fx^k - u^k}^2 \vspace{1ex}\\
& \overset{\tiny \eqref{eq:EG4NE_u_cond} }{ \leq } &  (1+r) \norms{Fx^k - Fy^k}^2 + \frac{\kappa_1(1 + r)}{r}\norms{Fx^k - Fy^{k-1}}^2 \vspace{1ex}\\
&& + {~} \frac{\kappa_2(1 + r)}{r} \norms{Fx^k - Fx^{k-1}}^2 \vspace{1ex}\\
& \leq &  (1+r)L^2  \norms{x^k - y^k}^2 + \frac{\kappa_1(1 + r)L^2}{r}\norms{x^k - y^{k-1}}^2 \vspace{1ex}\\
&& + {~} \frac{\kappa_2(1 + r)L^2}{r} \norms{x^k - x^{k-1}}^2.
\end{array}
\hspace{-2ex}
\end{equation}
Now, substituting $x^{\star} \in \zer{F}$ for $\hat{x}$ into \eqref{eq:EG4NE_key_est1}, and using \eqref{eq:EG4NE_lm32_proof1} and \eqref{eq:EG4NE_lm32_proof2}, we have
\begin{equation*} 
\arraycolsep=0.2em
\begin{array}{lcl}
\norms{x^{k+1} - x^{\star}}^2 & \leq &  \norms{x^k - x^{\star} }^2 +  \tfrac{\kappa_1(1 + r)L^2\eta^2}{r \gamma}\norms{x^k - y^{k-1}}^2 +  \frac{\kappa_2(1 + r)L^2\eta^2}{r \gamma}\norms{x^k - x^{k-1}}^2 \vspace{1ex}\\
&& - {~} \big( \beta - \frac{(1 + r)L^2\eta^2}{ \gamma} - \frac{2\mu\rho(1 + s)}{s \eta } \big) \norms{y^k - x^k}^2 \vspace{1ex}\\
&& - {~}   \big(\beta - \gamma - \frac{2\mu\rho(1 + s)}{\eta} \big)\norms{x^{k+1} - y^k}^2 \vspace{1ex}\\
&& - {~}  \big( 1 - \beta - \frac{2(1-\mu)\rho}{\eta} \big)\norms{x^{k+1} - x^k}^2.
\end{array}
\end{equation*}
Finally, by rearranging this inequality, and then using $\Pc_k$ from \eqref{eq:EG4NE_Lyapunov_func}, we get
\begin{equation*} 
\hspace{-2ex}
\arraycolsep=0.2em
\begin{array}{lcl}
\Pc_{k+1} & \leq & \Pc_k -  \left( \beta - \frac{(1 + r)L^2\eta^2}{ \gamma} - \frac{2\mu\rho(1 + s)}{s\eta}  \right) \norms{y^k - x^k}^2 \vspace{1ex}\\ 
&& - {~} \left(\beta - \gamma - \frac{\kappa_1(1 + r) L^2\eta^2}{r \gamma} - \frac{2\mu\rho(1 + s)}{\eta}  \right)\norms{x^{k+1} - y^k}^2 \vspace{1ex}\\
&& - {~}  \big(1 - \beta - \frac{\kappa_2(1 + r) L^2\eta^2}{r \gamma} - \frac{2(1-\mu)\rho}{\eta} \big)\norms{x^{k+1} - x^k}^2.
\end{array}
\hspace{-2ex}
\end{equation*}
This expression is exactly  \eqref{eq:EG4NE_key_estimate1b}.
\Eproof
\end{proof}

\begin{lemma}\label{le:EG4NE_monotonicity}
Let $F$ be $\rho$-co-hypomonotone, i.e. there exists $\rho \geq 0$ such that $\iprods{Fx - Fy, x-y} \geq -\rho\norms{Fx - Fy}^2$ for all $x, y \in\dom{F}$ and $F$ be $L$-Lipschitz continuous.
Let $\sets{(x^k, y^k)}$ be generated by \eqref{eq:EG4NE}.
Then,  for any $s > 0$, $\omega \geq 0$, and $\hat{\omega} \geq 0$, the following estimate holds:
\begin{equation}\label{eq:EG4NE_monotonicity}
\arraycolsep=0.1em
\begin{array}{lcl}
\norms{Fx^{k+1}}^2 & + & \omega\norms{Fx^{k+1} - Fy^k}^2 + \hat{\omega}\norms{Fx^{k+1} - Fx^k}^2 \leq \norms{Fx^k}^2 \vspace{1ex}\\
&& - {~}  \big[1 - \frac{1+s}{s} \big(\hat{\omega} + \frac{2\rho}{\eta}\big)\big] \norms{Fy^k - Fx^k}^2 \vspace{1ex}\\
&& + {~} \big[1+ \omega + (1+s)\big(\hat{\omega} + \frac{2\rho}{\eta}\big) \big]\frac{L^2\eta^2}{\beta^2} \norms{\beta Fy^k - u^k}^2.
\end{array}
\end{equation}
\end{lemma}

\begin{proof}
Since  $F$ is $\rho$-co-hypomonotone, we have $\iprods{Fx^{k+1} - Fx^k, x^{k+1} - x^k} + \rho\norms{Fx^{k+1} - Fx^k}^2 \geq 0$.
Substituting $x^{k+1} - x^k = -\eta Fy^k$ from the second line of \eqref{eq:EG4NE} into this inequality and expanding it, we can show that
\begin{equation*} 
\arraycolsep=0.2em
\begin{array}{lcl}
0 & \leq & 2\iprods{Fx^k, Fy^k} - 2\iprods{Fx^{k+1}, Fy^k} + \frac{2\rho}{\eta}\norms{Fx^{k+1} - Fx^k}^2 \vspace{1ex}\\
& =  & \norms{Fx^k}^2 - \norms{Fy^k - Fx^k}^2 - \norms{Fx^{k+1}}^2  + \norms{Fx^{k+1} - Fy^k}^2 \vspace{1ex}\\
&& + {~} \frac{2\rho}{\eta}\norms{Fx^{k+1} - Fx^k}^2.
\end{array}
\end{equation*}
For any $\omega \geq 0$,  $\hat{\omega} \geq 0$, and $s > 0$, by Young's inequality, this estimate leads to
\begin{equation}\label{eq:EG4NE_lm33_proof2} 
\arraycolsep=0.2em
\begin{array}{lcl}
\Tc_{[k+1]} &:= & \norms{Fx^{k+1}}^2 + \omega \norms{Fx^{k+1} - Fy^k}^2  + \hat{\omega}\norms{Fx^{k+1} - Fx^k}^2 \vspace{1ex}\\
& \leq & \norms{Fx^k}^2 - \norms{Fy^k - Fx^k}^2  + (1+\omega)\norms{Fx^{k+1} - Fy^k}^2 \vspace{1ex}\\
&& + {~} \big(\hat{\omega} + \frac{2\rho}{\eta}\big)\norms{Fx^{k+1} - Fx^k}^2 \vspace{1ex} \\
& \leq & \norms{Fx^k}^2   + \big[1+ \omega + (1+s)\big(\hat{\omega} + \frac{2\rho}{\eta}\big) \big] \norms{Fx^{k+1} - Fy^k}^2 \vspace{1ex}\\
&& - {~} \big[1 - \frac{1+s}{s} \big(\hat{\omega} + \frac{2\rho}{\eta}\big)\big] \norms{Fy^k - Fx^k}^2.
\end{array}
\end{equation}
Now, by the $L$-Lipschitz continuity of $F$ and $x^{k+1} - y^k = -\eta (Fy^k - \frac{1}{\beta}u^k)$ from \eqref{eq:EG4NE}, we obtain $\norms{Fx^{k+1} - Fy^k}^2 \leq  L^2\norms{x^{k+1} - y^k}^2 = \frac{L^2\eta^2}{\beta^2}\norms{\beta Fy^k -  u^k}^2$.
Substituting this relation into \eqref{eq:EG4NE_lm33_proof2}, we arrive at \eqref{eq:EG4NE_monotonicity}.
\Eproof
\end{proof}

\beforesubsec
\subsection{\bf The Sublinear Best-Iterate Convergences Rate  of \eqref{eq:EG4NE}}\label{subsec:convergence_of_EG4NE}
\aftersubsec
For given constants $\kappa_1$ and $\kappa_2$ in \eqref{eq:EG4NE_u_cond}, we denote $r := \frac{\kappa_1 + \sqrt{\kappa_1^2 + 4\kappa_1}}{2} \geq 0$. 
First, for any  $\beta \in (0, 1]$,  we  define the following quantity:
\begin{equation}\label{eq:EG4NE_kappa_i_quantities}
\arraycolsep=0.2em
\begin{array}{lcl}
\Delta  & := & \frac{1}{16(1+r)} \cdot \min\Big\{ \big(1 + \frac{2\kappa_2}{r}\big)\beta^2, \ \frac{(1-\beta)^2r(r+2\kappa_2)}{\kappa_2^2 } \Big\}. 
\end{array}
\end{equation}
For  $r$ given above, let $\alpha := \frac{1+r}{r}$ and $\mu := \frac{r}{r + 2\kappa_2}$.
We also define
\begin{equation}\label{eq:EG4NE_eta_bounds}
\hspace{-2ex}
\arraycolsep=0.1em
\left\{\begin{array}{lcllcl}
\underline{\eta}_1 &:= & \frac{\beta - \sqrt{\beta^2 - 16(1+r)\mu L\rho}}{2(1+r)L} , \quad 
& \bar{\eta}_1 &:= & \frac{\beta + \sqrt{\beta^2 - 16(1+r)\mu L\rho}}{2(1+r)L}, \vspace{1ex}\\
\underline{\eta}_2 &:= & \frac{1-\beta -  \sqrt{(1-\beta)^2 - 8\alpha (1-\mu)\kappa_2L\rho}}{2 \alpha \kappa_2 L}, \quad 
& \bar{\eta}_2 &:= &  \frac{1-\beta +  \sqrt{(1-\beta)^2 - 8\alpha (1-\mu)\kappa_2L\rho}}{2\alpha \kappa_2 L}, \vspace{1ex}\\
\underline{\eta} &:= & \max\sets{\underline{\eta}_1, \underline{\eta}_2 },   \quad & \bar{\eta} & := & \min\sets{\bar{\eta}_1, \bar{\eta}_2}.
\end{array}\right.
\hspace{-2ex}
\end{equation}
Note that we can only choose $\beta = 1$ if $\kappa_2 = 0$ in \eqref{eq:EG4NE_u_cond}. 
In this case, we have $\mu = 1$, $\underline{\eta} = \underline{\eta}_1$, and $\bar{\eta} = \bar{\eta}_1$.
Lemma~\ref{le:EG4NE_para_conditions} in the appendix below shows that if $L\rho \leq \Delta$, then $0 \leq \underline{\eta} \leq \bar{\eta}$, i.e. $[\underline{\eta}, \bar{\eta}] \neq \emptyset$.

Now, given $\eta \in [\underline{\eta}, \bar{\eta}] $, we consider the following quantities:
\begin{equation}\label{eq:EG4NE_C_constants}
\begin{array}{lcl}
C_1 & := &  \beta - (1+r)L\eta - \frac{4\mu\rho}{\eta} \quad \text{and} \quad C_2  :=  1 - \beta - \alpha \kappa_2 L\eta - \frac{2(1-\mu)\rho}{\eta}.
\end{array}
\end{equation}
Then, we are ready to state the best-iterate convergence rate of \eqref{eq:EG4NE} for any $u^k$ satisfying the condition \eqref{eq:EG4NE_u_cond}.

\begin{theorem}[Best-iterate convergence rate]\label{th:EG4NE_convergence1}
For Equation \eqref{eq:NE}, suppose that $\zer{F}\neq\emptyset$, $F$  is $L$-Lipschitz continuous, and there exist $x^{\star} \in \zer{F}$ and $\rho \geq 0$ such that $\iprods{Fx, x - x^{\star}} \geq -\rho\norms{Fx}^2$ for all $x\in\dom{F}$.

Let $\Delta$ be defined by  \eqref{eq:EG4NE_kappa_i_quantities}, and $\underline{\eta}$ and $\bar{\eta}$ be defined by \eqref{eq:EG4NE_eta_bounds}.
Let $\sets{(x^k, y^k)}$ be generated by \eqref{eq:EG4NE} starting from $x^0\in\dom{F}$ and $y^{-1} = x^{-1} := x^0$ such that $u^k$ satisfies \eqref{eq:EG4NE_u_cond}.
If $L\rho \leq \Delta$ and $\eta$ is chosen such that $\eta \in [\underline{\eta}, \bar{\eta}]$, then $C_1$ and $C_2$ in \eqref{eq:EG4NE_C_constants}  are nonnegative.

Moreover, for any $K \geq 0$, we have
\begin{equation}\label{eq:EG4NE_convergence_est1}
\sum_{k=0}^K \big[  C_1 \norms{y^k - x^k}^2 + C_1 \norms{x^{k+1} - y^k}^2  + C_2 \norms{x^{k+1} - x^k}^2 \big] \leq \norms{x^0 - x^{\star}}^2.
\end{equation} 
If $C_1 > 0$, then for $\Lambda := \frac{C_1  + 2 C_2}{2}$ and $\Gamma :=   \frac{2L^2}{C_1} + \frac{2}{\Lambda\eta^2}$, we also have
\begin{equation}\label{eq:EG4NE_convergence_est1_details}
\left\{\begin{array}{ll}
{\displaystyle\min_{0\leq k \leq K}} \norms{u^k}^2 \leq \frac{1}{K+1}\sum_{k=0}^K\norms{u^k}^2 \leq \frac{\beta^2 \norms{x^0 - x^{\star}}^2 }{ C_1\eta^2 (K+1)}, \vspace{1ex}\\
{\displaystyle\min_{0\leq k \leq K}} \norms{Fy^k}^2 \leq \frac{1}{K+1}\sum_{k=0}^K\norms{Fy^k}^2 \leq \frac{\norms{x^0 - x^{\star}}^2 }{\Lambda  \eta^2 (K+1)}, \vspace{1ex}\\
{\displaystyle\min_{0\leq k \leq K}} \norms{Fx^k}^2 \leq \frac{1}{K+1}\sum_{k=0}^K\norms{Fx^k}^2 \leq \frac{\Gamma  \norms{x^0 - x^{\star}}^2 }{ K+1 }.
\end{array}\right.
\end{equation} 
Consequently, the following results hold:
\begin{equation}\label{eq:EG4NE_limits_best_iterates}
\arraycolsep=0.2em
\begin{array}{lll}
{\displaystyle\min_{0\leq k \leq K}}\norms{Fx^k} = \BigOs{\frac{1}{ \sqrt{K} } }, \quad {\displaystyle\min_{0\leq k \leq K}} \norms{Fy^k } = \BigOs{ \frac{1}{ \sqrt{K}} }, \vspace{1ex}\\
{\displaystyle\lim_{k\to\infty}} \norms{x^k - y^k} = {\displaystyle \lim_{k\to\infty}} \norms{Fx^k} = {\displaystyle \lim_{k\to\infty}} \norms{Fy^k} =  0.
\end{array}
\end{equation}
In addition, $\sets{\norms{x^k - x^{\star}}}$ is bounded and $\sets{x^k}$ converges to $x^{\star} \in \zer{F}$.
\end{theorem}

\begin{proof}
First, we choose $\gamma := L\eta$ in Lemma~\ref{le:EG4NE_key_estimate1b}.
Next, substituting  $\alpha := \frac{1 + r}{r}$ and $s := 1$ into the following quantities with $r := \frac{\kappa_1 + \sqrt{\kappa_1^2 + 4\kappa_1}}{2}$, we get
\begin{equation*}
\arraycolsep=0.2em
\left\{\begin{array}{lcllcl}
C_1 & := & \beta - \frac{(1 + r)L^2\eta^2}{ \gamma} - \frac{2\mu\rho(1 + s)}{s \eta} &= & \beta - (1+r)L\eta - \frac{4\mu\rho}{\eta}, \vspace{1ex}\\
\hat{C}_1 & := & \beta - \gamma - \frac{\kappa_1(1 + r) L^2\eta^2}{r \gamma} - \frac{2\mu\rho(1 + s)}{\eta} &= & \beta - (1+r) L\eta - \frac{4\mu\rho}{\eta}, \vspace{1ex}\\
C_2 & := & 1 - \beta - \frac{\kappa_2(1 + r) L^2\eta^2}{r \gamma} - \frac{2(1-\mu)\rho}{\eta} & = & 1 - \beta - \alpha \kappa_2L\eta  - \frac{2(1-\mu)\rho}{\eta}.
\end{array}\right.
\end{equation*}
This is exactly the constants defined in \eqref{eq:EG4NE_C_constants}.
By Lemma~\ref{le:EG4NE_para_conditions}, for $\beta \in (0, 1]$, we have $\underline{\eta} \leq \bar{\eta}$.
In addition, for any $\eta \in [\underline{\eta}, \bar{\eta}]$, $C_1 = \hat{C}_1 \geq 0$ and $C_2 \geq 0$.

Now, for  given $\eta$ and $C_1$ and $C_2$ defined by \eqref{eq:EG4NE_C_constants},  \eqref{eq:EG4NE_key_estimate1b} is rewritten as follows:
\begin{equation}\label{eq:EG4NE_th31_proof2}
\arraycolsep=0.2em
\begin{array}{lcl}
C_1 \norms{y^k - x^k}^2 +  C_1 \norms{x^{k+1} - y^k}^2  + C_2 \norms{x^{k+1} - x^k}^2  \leq \Pc_k - \Pc_{k+1}.
\end{array}
\end{equation}
Summing up \eqref{eq:EG4NE_th31_proof2} from $k=0$ to $K$, and noting that $\Pc_k \geq 0$, we obtain 
\begin{equation*} 
\arraycolsep=0.2em
\begin{array}{lcl}
\sum_{k=0}^K \big[  C_1 \norms{x^{k+1} - y^k}^2 + C_1 \norms{y^k - x^k}^2 + C_2 \norms{x^{k+1} - x^k}^2  \big] \leq \Pc_0.
\end{array}
\end{equation*}
Since we choose $y^{-1} = x^{-1} = x^0$, it follows from \eqref{eq:EG4NE_Lyapunov_func} that $\Pc_0 = \norms{x^0 - x^{\star} }^2$.
Substituting this fact into the last inequality, we obtain \eqref{eq:EG4NE_convergence_est1}.

Next, from the first line of \eqref{eq:EG4NE}, we have $\eta^2\norms{u^k}^2 = \beta^2\norms{y^k - x^k}^2$.
Combining this relation and \eqref{eq:EG4NE_convergence_est1} we get the first line of \eqref{eq:EG4NE_convergence_est1_details}.

Then, from the second line of \eqref{eq:EG4NE}, by Young's inequality, we have 
\begin{equation*} 
\arraycolsep=0.2em
\begin{array}{lcl}
\eta^2\norms{Fy^k}^2 & = & \norms{x^{k+1} - x^k}^2 \vspace{1ex}\\
& \leq &  a_1(1+a_2^{-1})\norms{y^k - x^k}^2 + a_1(1+a_2)\norms{x^{k+1} - y^k}^2 \vspace{1ex}\\
&& + {~} (1-a_1)\norms{x^{k+1} - x^k}^2,
\end{array}
\end{equation*}
for any $a_1 \in [0, 1]$ and $a_2 > 0$.
If $C_1  > 0$, then by setting $\Lambda := \frac{ C_1  + 2C_2}{2}$, $a_1 := 1 - \frac{C_2}{\Lambda}$, and $a_2 := 1$, the last inequality is equivalent to
\begin{equation*}
\begin{array}{lcl}
\Lambda\eta^2\norms{Fy^k}^2 \leq   C_1 \norms{y^k - x^k}^2 + C_1 \norms{x^{k+1} - y^k}^2 + C_2 \norms{x^{k+1} - x^k}^2.
\end{array}
\end{equation*}
Combining this inequality and \eqref{eq:EG4NE_convergence_est1} we get the second line of \eqref{eq:EG4NE_convergence_est1_details}.

Again, by Young's inequality and the $L$-Lipschitz continuity of $F$,  we have 
\begin{equation*}
\begin{array}{lcl}
\eta^2\norms{Fx^k}^2 \leq 2\eta^2\norms{Fy^k}^2 + 2\eta^2\norms{Fx^k - Fy^k}^2 \leq 2 \eta^2\norms{Fy^k}^2 + 2L^2\eta^2\norms{x^k - y^k}^2.
\end{array}
\end{equation*}
Combining this inequality, $\sum_{k=0}^K2L^2\eta^2\norms{y^k - x^k}^2 \leq \frac{2L^2\eta^2}{C_1}\norms{x^0 - x^{\star}}^2$, and $\sum_{k=0}^K2\eta^2\norms{Fy^k}^2 \leq \frac{2}{\Lambda}\norms{x^0 - x^{\star}}^2$, we obtain the third line of \eqref{eq:EG4NE_convergence_est1_details}, where $\Gamma := \frac{2L^2}{C_1} + \frac{2}{\Lambda\eta^2}$.

The bounds and limits in \eqref{eq:EG4NE_limits_best_iterates} are direct consequences of \eqref{eq:EG4NE_convergence_est1} and \eqref{eq:EG4NE_convergence_est1_details}.
From \eqref{eq:EG4NE_th31_proof2}, we can see that $\sets{\Pc_k}$ is non-increasing, and $\norms{x^k - x^{\star}}^2 \leq \Pc_k \leq \Pc_0 = \norms{x^0 - x^{\star}}^2$, showing that $\sets{\norms{x^k - x^{\star}}}$ is bounded.
Finally, the convergence of $\sets{x^k}$ to $x^{\star} \in \zer{F}$ can be proven using standard arguments as in, e.g., \cite{malitsky2015projected} by noticing that we are working on a finite-dimensional space $\R^p$.
\Eproof
\end{proof}

\begin{remark}\label{re:EG4NE_rm1}
If $\rho = 0$ in Theorem~\ref{th:EG4NE_convergence1}, i.e. $F$ is star-monotone (and in particular, monotone) and $\kappa_2 = 0$, then the choice $\eta \in [\underline{\eta}, \bar{\eta}]$ reduces to $0 < \eta < \frac{\beta}{(1+r)L}$.
For EG with $u^k := Fx^k$, we have $\kappa_1 = 0$, and hence $0 < \eta < \frac{\beta}{L}$.
For past-EG with $u^k := Fy^{k-1}$, we have $\kappa_1 = 1$, and hence $0 < \eta < \frac{2\beta}{(3+\sqrt{5})L}$. 
These choices are standard in the literature \cite{Facchinei2003,popov1980modification}. 
\end{remark}

\beforesubsec
\subsection{\bf The Sublinear Last-Iterate Convergence Rate of \eqref{eq:EG4NE}}\label{subsec:EG4NE_last_iterate}
\aftersubsec
For simplicity of our presentation, let us first define the following quantities:
\begin{equation}\label{eq:EG4NE_last_iterate_constant}
\arraycolsep=0.2em
\begin{array}{llcl}
& m := [(1+\sqrt{2}) (\kappa_1 + \sqrt{2})]^{1/2}, \quad & &\hat{\Delta} := \frac{1}{16m}, \vspace{1ex}\\
& \underline{\hat{\eta}} := \frac{1 - \sqrt{1 - 16mL\rho}}{2mL},  \quad & \text{and}  \quad & \bar{\hat{\eta}} := \frac{1 + \sqrt{1 - 16mL\rho} }{ 2mL }.
\end{array}
\end{equation}
Then, for $\underline{\eta}$ and $\bar{\eta}$ defined by \eqref{eq:EG4NE_eta_bounds}, by Lemma~\ref{le:EG4NE_nonemptyness2} in the appendix, under given conditions on the parameters specified below,  we have $[ \underline{\eta}, \bar{\eta}] \cap [\underline{\hat{\eta}}, \bar{\hat{\eta}}] \neq \emptyset$.

Now, we can state the last-iterate convergence rate of \eqref{eq:EG4NE} as follows.

\begin{theorem}[Last-iterate convergence rate]\label{th:EG4NE_convergence2}
For Equation \eqref{eq:NE}, suppose that $\zer{F}\neq\emptyset$, and $F$ is $L$-Lipschitz continuous and $\rho$-co-hypomonotone.

Let $\Delta$, $\underline{\eta}$ and $\bar{\eta}$ be defined by \eqref{eq:EG4NE_kappa_i_quantities} and \eqref{eq:EG4NE_eta_bounds}, respectively, and $\hat{\Delta}$,  $\underline{\hat{\eta}}$, and $\bar{\hat{\eta}}$ be defined by \eqref{eq:EG4NE_last_iterate_constant}.
Let $\sets{(x^k, y^k)}$ be generated by \eqref{eq:EG4NE} starting from $x^0 \in\dom{F}$ with $y^{-1} = x^{-1} := x^0$ and $\beta := 1$ such that $u^k$ satisfies \eqref{eq:EG4NE_u_cond} with $\kappa_2 := 0$.
Assume that $L\rho \leq \hat{\Delta}$ and $\eta$ is chosen such that $\eta \in [ \underline{\eta}, \bar{\eta}] \cap [\underline{\hat{\eta}}, \bar{\hat{\eta}}] \neq \emptyset$.
Then, for $\omega := \frac{2(1+\sqrt{2})\kappa_1L^2\eta^2}{1 - (1+\sqrt{2})\kappa_1L^2\eta^2} \geq 0$, we have
\begin{equation}\label{eq:EG4NE_last_iterate_convergence1_general}
\arraycolsep=0.1em
\begin{array}{lcl}
& \norms{Fx^{k+1}}^2   +   \omega \norms{Fx^{k+1} - Fy^k}^2  \leq  \norms{Fx^k}^2 + \omega \norms{Fx^k - Fy^{k-1}}^2, \vspace{1ex} \\
& \norms{Fx^K}^2 \leq \norms{Fx^K}^2 + \omega\norms{Fx^K - Fy^{K-1}}^2   \leq \left(\frac{\omega L^2}{C_1} + \Gamma \right)\frac{ \norms{x^0 - x^{\star}}^2}{K+1}.
\end{array}
\end{equation}
Thus we conclude that $\norms{Fx^K} = \BigOs{1/\sqrt{K}}$ on the last-iterate $x^K$.
\end{theorem}

\begin{proof}
First, since $\beta = 1$ and $\kappa_2 = 0$, for any $b > 0$, it follows from Young's inequality and \eqref{eq:EG4NE_u_cond} that 
\begin{equation}\label{eq:EG4NE_th32_proof1}
\arraycolsep=0.2em
\begin{array}{lcl}
\norms{Fy^k - u^k}^2 & \leq & (1+ b)\norms{Fx^k - Fy^k}^2 + \frac{(1 + b)}{b}\norms{Fx^k - u^k}^2 \vspace{1ex}\\
& \leq & (1+b)\norms{Fx^k - Fy^k}^2 + \frac{(1+b)\kappa_1}{b}\norms{Fx^k - Fy^{k-1}}^2.
\end{array}
\end{equation}
We can choose $\hat{\omega} = 0$ and substitute $\beta = 1$ into \eqref{eq:EG4NE_monotonicity} to get
\begin{equation*} 
\arraycolsep=0.1em
\begin{array}{lcl}
\norms{Fx^{k+1}}^2 & + & \omega\norms{Fx^{k+1} - Fy^k}^2  \leq \norms{Fx^k}^2 -  \big[1 - \frac{2(1+s)\rho}{s\eta} \big] \norms{Fy^k - Fx^k}^2 \vspace{1ex}\\
&& + {~} L^2\eta^2 \big[1+ \omega + \frac{2(1+s)\rho}{\eta} \big] \norms{Fy^k - u^k}^2.
\end{array}
\end{equation*}
Let us denote $A_1 := 1 -  \frac{2(1+s)\rho}{s\eta}$ and $A_2 :=  \big[1+ \omega + \frac{2(1+s)\rho}{\eta} \big] L^2\eta^2$.
Then, utilizing \eqref{eq:EG4NE_th32_proof1}, $A_1$, and $A_2$ into the last inequality, we can show that
\begin{equation}\label{eq:EG4NE_last_iterate_proof2} 
\arraycolsep=0.1em
\begin{array}{lcl}
\norms{Fx^{k+1}}^2 & + &  \omega \norms{Fx^{k+1} - Fy^k}^2  \leq  \norms{Fx^k}^2 + \frac{\kappa_1(1+b)A_2}{b} \norms{Fx^k - Fy^{k-1}}^2 \vspace{1ex}\\
&& - {~}  [ A_1 - (1+b)A_2 ] \norms{Fy^k - Fx^k}^2. 
\end{array}
\end{equation}
Suppose that there exists $m > 0$ such that $1 - \frac{2(1+s)\rho}{s\eta} - m L\eta \geq 0$.
This implies that $\frac{2(1+s)\rho}{\eta} \leq s$ and $L\eta \leq \frac{1}{m}$.
Hence, we can upper bound $A_2$ as $A_2 \leq \bar{A}_2 := (1 + s + \omega)L^2\eta^2$.
If we impose $\omega =  \frac{\kappa_1(1+b)\bar{A}_2}{b}$, then using $A_2 \leq \bar{A}_2$, we have $ \frac{\kappa_1(1+b)A_2}{b} \leq \omega$.
Clearly, $\omega := \frac{\kappa_1(1+b)(1+s)L^2\eta^2}{b - \kappa_1(1+b)L^2\eta^2}$ satisfies $\omega =  \frac{\kappa_1(1+b)\bar{A}_2}{b}$, provided that $\kappa_1(1+b)L^2\eta^2 < b$.

For the choice of $\omega$ above, we assume that $(1+b)\bar{A}_2 \leq m L\eta$.  
Using $L\eta \leq \frac{1}{m}$, we can easily show that if we choose $m := \sqrt{\frac{(1+b)[\kappa_1 + b(1+s)] }{b}}$, then $(1+b)\bar{A}_2 \leq m L\eta$.
As a consequence, we have
\begin{equation*} 
\arraycolsep=0.1em
\begin{array}{lcl}
A_1 - (1+b)A_2 &\geq & 1 - \frac{2(1+s)\rho}{s\eta} - (1+b)\bar{A}_2 \geq  1 - \frac{2(1+s)\rho}{s\eta} - m L\eta \geq 0.
\end{array}
\end{equation*}
Utilizing $A_1 - (1+b)A_2 \geq 0$ and $\frac{\kappa_1(1+b)A_2}{b} \leq \omega$ proven above, \eqref{eq:EG4NE_last_iterate_proof2} reduces to
\begin{equation*} 
\arraycolsep=0.1em
\begin{array}{lcl}
\norms{Fx^{k+1}}^2 & + &  \omega \norms{Fx^{k+1} - Fy^k}^2  \leq  \norms{Fx^k}^2 + \omega \norms{Fx^k - Fy^{k-1}}^2,
\end{array}
\end{equation*}
which proves the first line of \eqref{eq:EG4NE_last_iterate_convergence1_general}.

On the one hand,  if $\underline{\hat{\eta}}  := \frac{1 - \sqrt{1 - 8(1+s)mL\rho/s}}{2mL} \leq \eta \leq \bar{\hat{\eta}} := \frac{1 + \sqrt{1 - 8(1+s)mL\rho/s }}{2mL}$, then the condition $1 - \frac{2(1+s)\rho}{s\eta} - m L\eta \geq 0$ holds,  provided that $L\rho \leq \frac{s}{8(1+s)m}$.
On the other hand, for $\beta = 1$ and $\kappa_2 = 0$, $\Delta$ defined by \eqref{eq:EG4NE_kappa_i_quantities} becomes $\Delta = \frac{1}{16(1+r)}$.
Moreover, $\underline{\eta}$ and $\bar{\eta}$ defined by \eqref{eq:EG4NE_eta_bounds} reduce to $\underline{\eta} = \frac{1 - \sqrt{1 - 16(1+r)L\rho}}{2(1+r)L}$ and $\bar{\eta} = \frac{1 - \sqrt{1 - 16(1+r)L\rho}}{2(1+r)L}$, respectively.

If we choose $s := 1$ and $b := \frac{1}{\sqrt{2}}$, then $m = [(1+\sqrt{2})(\kappa_1 + \sqrt{2})]^{1/2}$ and $L\rho \leq \frac{s}{8(1+s)m} = \frac{1}{16m} =:\hat{\Delta}$.
In addition, we can verify that $\omega :=  \frac{2(1+\sqrt{2})\kappa_1L^2\eta^2}{1 - (1+\sqrt{2})\kappa_1L^2\eta^2}$,  $\underline{\hat{\eta}} := \frac{1 - \sqrt{1 - 16mL\rho}}{2mL}$, and $\bar{\hat{\eta}} := \frac{1 + \sqrt{1 - 16 mL\rho }}{2mL}$.
Applying Lemma~\ref{le:EG4NE_nonemptyness2} in the appendix below, we can easily show that $[\underline{\eta}, \bar{\eta}] \cap [\underline{\hat{\eta}}, \bar{\hat{\eta}}] \neq \emptyset$.

Finally, from \eqref{eq:EG4NE_convergence_est1}, the third line of \eqref{eq:EG4NE_convergence_est1_details}, and the $L$-Lipschitz continuity of $F$, it is straightforward to get
\begin{equation*}
\arraycolsep=0.1em
\begin{array}{llcl}
&\frac{\omega}{K+1}\sum_{k=0}^K \norms{Fx^{k} - Fy^{k-1}}^2 \leq \frac{\omega L^2}{K+1}\sum_{k=0}^K \norms{x^{k} - y^{k-1}}^2 & \leq &  \frac{\omega L^2\norms{x^0 - x^{\star}}^2 }{C_1(K+1)}, \vspace{1ex}\\
& \frac{1}{K+1}\sum_{k=0}^K\norms{Fx^{k}}^2 & \leq & \frac{ \Gamma  \norms{x^0 - x^{\star}}^2 }{ K+1 }.
\end{array}
\end{equation*}
Adding both inequalities and using the first line of \eqref{eq:EG4NE_last_iterate_convergence1_general}, we can show that 
\begin{equation*}
\arraycolsep=0.1em
\begin{array}{lcl}
\norms{Fx^{K}}^2 & \leq &  \norms{Fx^{K}}^2 + \omega \norms{Fx^{K} - Fy^{K-1}}^2 \vspace{1ex}\\
& \leq & \frac{1}{K+1}\sum_{k=0}^K[ \norms{Fx^{k}}^2 + \omega \norms{Fx^{k} - Fy^{k-1}}^2 ] \vspace{1ex}\\
& \leq & \big( \frac{\omega L^2}{C_1} + \Gamma \big) \frac{\norms{x^0 - x^{\star}}^2 }{K+1},
\end{array}
\end{equation*}
which proves the second line of \eqref{eq:EG4NE_last_iterate_convergence1_general}.
\Eproof
\end{proof}

\beforesubsec
\subsection{\bf The Sublinear Convergences Rates  of Two Instances}\label{subsec:special_case_of_EG}
\aftersubsec
Now, we specify \eqref{eq:EG4NE} for two instances: \textbf{Variant 1} -- the EG method, and \textbf{Variant 2} -- the past-EG method.
Then, we apply Theorem~\ref{th:EG4NE_convergence1} and Theorem~\ref{th:EG4NE_convergence2} to derive the following corollaries. 

\begin{corollary}[The EG method]\label{co:EG4NE_convergence}
For Equation~\eqref{eq:NE}, suppose that $\zer{F} \neq \emptyset$, $F$ is $L$-Lipschitz continuous, and there exist $x^{\star}\in\zer{F}$ and $\rho \geq 0$ such that $\iprods{Fx, x - x^{\star}} \geq -\rho\norms{Fx}^2$ for all $x\in\dom{F}$ $($the last condition holds if, in particular, $F$ is $\rho$-co-hypomonotone$)$.

Let $\sets{(x^k, y^k)}$ be generated by \eqref{eq:EG4NE} using $u^k := Fx^k$ $($\textbf{Variant 1}$)$.
Suppose that $L\rho \leq \frac{\beta^2}{16}$ and $\eta$ is chosen as
\begin{equation}\label{eq:EG4NE_stepsize1}
\arraycolsep=0.2em
\begin{array}{l}
0 \leq \frac{\beta - \sqrt{\beta^2 - 16L\rho}}{2L} < \eta < \frac{\beta + \sqrt{\beta^2 - 16L\rho}}{2L} \leq \frac{\beta}{L}.
\end{array}
\end{equation}
Then, with $C_1 := \beta - L\eta - \frac{4\rho}{\eta} > 0$, we have
\begin{equation}\label{eq:EG4NE_convergence_est1a}
{\displaystyle\min_{0\leq k \leq K}} \norms{Fx^k}^2 \leq \frac{1}{K+1}\sum_{k=0}^K\norms{Fx^k}^2 \leq \frac{\beta^2 \norms{x^0 - x^{\star}}^2}{ C_1\eta^2 ( K+1 ) }. 
\end{equation} 
Thus we conclude that $\min_{0\leq k \leq K}\norms{Fx^k} = \BigOs{1/\sqrt{K}}$ showing a $\BigOs{1/\sqrt{K}}$ best-iterate convergence rate of $\sets{x^k}$.

In addition, $\sets{\norms{x^k - x^{\star}}}$ is nonincreasing and $\sets{x^k}$ converges to $x^{\star} \in \zer{F}$.
Moreover, we have
\begin{equation}\label{eq:EG4NE_limits}
\arraycolsep=0.2em
\begin{array}{l}
\lim_{k\to\infty}\norms{x^k - y^k} = \lim_{k\to\infty}\norms{Fx^k} = \lim_{k\to\infty}\norms{Fy^k} =  0.
\end{array}
\end{equation}
In particular, if  $F$ is $\rho$-co-hypomonotone such that $L\rho \leq \frac{1}{16\sqrt{2}} \approx 0.044194$, then for $\beta := 1$ and $\eta$ is chosen such that
\begin{equation}\label{eq:EG4NE_EG_last_iterate_eta}
\begin{array}{ll}
\frac{1 - \sqrt{1 - 16\sqrt{2}L\rho}}{2\sqrt{2}L} \leq  \eta \leq \frac{1 + \sqrt{1 - 16\sqrt{2}L\rho}}{2\sqrt{2}L},
\end{array}
\end{equation}
we have
\begin{equation}\label{eq:EG4NE_convergence_est1b}
\norms{Fx^{k+1}}^2   \leq   \norms{Fx^k}^2 - \psi \norms{Fx^k - Fy^k}^2  \quad \text{and} \quad \norms{Fx^K}^2  \leq  \frac{\norms{x^0 - x^{\star}}^2 }{C_1\eta^2 (K+1 ) },
\end{equation}
where $\psi := 1 - \frac{4\rho}{\eta} - \big(1 + \frac{4\rho}{\eta}\big)L^2\eta^2 \geq 0$.
The last bound shows that $\norms{Fx^K} = \BigOs{1/\sqrt{K}}$ on the last-iterate $x^K$.
\end{corollary}

\begin{proof}
Since we choose $u^k := Fx^k$ in \eqref{eq:EG4NE}, \eqref{eq:EG4NE_u_cond} holds with $\kappa_1 = \kappa_2 = 0$.
The condition $L\rho \leq \Delta$ of Theorem~\ref{th:EG4NE_convergence1} for $\Delta$ defined by \eqref{eq:EG4NE_kappa_i_quantities} reduces to $L\rho \leq \frac{\beta^2}{16}$ as we stated.
The condition $\eta \in [\underline{\eta}, \bar{\eta}]$ in Theorem~\ref{th:EG4NE_convergence1} becomes \eqref{eq:EG4NE_stepsize1}.
Consequently, \eqref{eq:EG4NE_stepsize1} holds strictly and we can derive \eqref{eq:EG4NE_convergence_est1a} directly from the first line of \eqref{eq:EG4NE_convergence_est1_details}.
The next two statements of Corollary~\ref{co:EG4NE_convergence} are also consequences of Theorem~\ref{th:EG4NE_convergence1}.

To obtain the last-iterate convergence rates, we set $\beta = 1$ and $\kappa_1 = \kappa_2 = 0$ in Theorem~\ref{th:EG4NE_convergence2}.
Hence, the condition  $L\rho \leq \hat{\Delta}$ in Theorem~\ref{th:EG4NE_convergence2} reduces to $L\rho \leq \frac{1}{16\sqrt{2}}$ as in Corollary~\ref{co:EG4NE_convergence}.
The choice of $\eta \in [ \underline{\eta}, \bar{\eta}] \cap [\underline{\hat{\eta}}, \bar{\hat{\eta}}]$ in Theorem~\ref{th:EG4NE_convergence2} becomes \eqref{eq:EG4NE_EG_last_iterate_eta}.
As a consequence, \eqref{eq:EG4NE_convergence_est1b} is obtained from \eqref{eq:EG4NE_last_iterate_convergence1_general} of Theorem~\ref{th:EG4NE_convergence2}.
Here, the term $\psi \norms{Fx^k - Fy^k}^2$ comes from \eqref{eq:EG4NE_last_iterate_proof2} since $\psi \leq A_1 - (1+b)A_2$.
\Eproof
\end{proof}

\begin{corollary}[The past-EG method]\label{co:EG4NE_convergence_option2}
For Equation \eqref{eq:NE}, suppose that $\zer{F}\neq\emptyset$, $F$ is $L$-Lipschitz continuous,  and there exist $x^{\star}\in\zer{F}$ and $\rho \geq 0$ such that $\iprods{Fx, x - x^{\star}} \geq -\rho\norms{Fx}^2$ for all $x\in\dom{F}$ $($the last condition holds if $F$ is $\rho$-co-hypomonotone$)$.

Let $\sets{(x^k, y^k)}$ be generated by \eqref{eq:EG4NE} using $u^k := Fy^{k-1}$ $($\textbf{Variant 2}$)$.
For any $\beta \in (0, 1]$, assume that $L\rho \leq  \frac{\beta^2}{8(3+\sqrt{5})}$ and $\eta$ is chosen as
\begin{equation}\label{eq:EG4NE_stepsize2}
\arraycolsep=0.2em
\begin{array}{l}
0 \leq \frac{\beta - \sqrt{\beta^2 - 8(3+\sqrt{5})L\rho}}{(3+\sqrt{5})L} < \eta < \frac{\beta + \sqrt{\beta^2 - 8(3+\sqrt{5})L\rho}}{(3+\sqrt{5})L} \leq \frac{2\beta}{(3+\sqrt{5})L}.
\end{array}
\end{equation}
Then, we have
\begin{equation}\label{eq:EG4NE_convergence_est2a}
\arraycolsep=0.2em
\begin{array}{lclcl}
{\displaystyle\min_{0\leq k \leq K}}  \norms{Fy^k }^2 &  \leq &  \frac{1}{K + 1}\sum_{k=0}^K \norms{Fy^k}^2 & \leq & \frac{ \norms{x^0 - x^{\star}}^2}{\Lambda \eta^2 (K+1) }, \vspace{1ex}\\
{\displaystyle\min_{0\leq k \leq K}}  \norms{Fx^k }^2 &  \leq & \frac{1}{K + 1}\sum_{k=0}^K \norms{Fx^k}^2  &  \leq & \frac{\Gamma \norms{x^0 - x^{\star}}^2}{K+1}, 
\end{array}
\end{equation} 
where $\Lambda := \frac{C_1 + 2C_2}{2}$ and $\Gamma := \frac{2}{\Lambda\eta^2} + \frac{2L^2}{C_1} > 0$ with $C_1 := 1 - \frac{(3+\sqrt{5})L\eta}{2} - \frac{4\rho}{\eta} > 0$ and $C_2 = 1-\beta \geq 0$.
Consequently, we have $\min_{0\leq k \leq K}\norms{Fx^k} = \BigOs{1/\sqrt{K}}$ showing the $\BigOs{1/\sqrt{K}}$ best-iterate convergence rate of $\sets{x^k}$.
Moreover, $\sets{x^k}$ converges to $x^{\star} \in \zer{F}$, and \eqref{eq:EG4NE_limits} still holds.

In particular, if $F$ is $\rho$-co-hypomonotone such that $L\rho \leq \frac{1}{16(1+\sqrt{2})} \approx 0.025888$, then for $\beta := 1$ and $\eta$ is chosen as in \eqref{eq:EG4NE_stepsize2}, we have
\begin{equation}\label{eq:EG4NE_convergence_est2b}
\arraycolsep=0.2em
\begin{array}{lcl}
\norms{Fx^{k+1}}^2 + \omega \norms{Fx^{k+1} - Fy^k}^2 & \leq &  \norms{Fx^k}^2 + \omega \norms{Fx^k - Fy^{k-1}}^2, \vspace{1ex}\\
\norms{Fx^K}^2  & \leq & \frac{\hat{C} \norms{x^0 - x^{\star}}^2 }{K+1},
\end{array}
\end{equation}
where $\omega := \frac{2(1+\sqrt{2}) L^2\eta^2}{1 - (1+\sqrt{2})L^2\eta^2} > 0$  and $\hat{C} := \Gamma + \frac{\omega L^2}{C_1}$.
Hence, we conclude that $\norms{Fx^K} = \BigOs{1/\sqrt{K}}$ on the last-iterate $x^K$.
\end{corollary}

\begin{proof}
Since we choose $u^k := Fy^{k-1}$ in \eqref{eq:EG4NE}, \eqref{eq:EG4NE_u_cond} holds with $\kappa_1 = 1$ and $\kappa_2 = 0$.
In this case, $r := \frac{\kappa_1 + \sqrt{\kappa_1^2 + 4\kappa_1}}{2} = \frac{1+\sqrt{5}}{2}$ in Theorem~\ref{th:EG4NE_convergence1}.
The condition $L\rho \leq \Delta$ of Theorem~\ref{th:EG4NE_convergence1} for $\Delta$ defined by \eqref{eq:EG4NE_kappa_i_quantities} reduces to $L\rho \leq \frac{\beta^2}{8(3+\sqrt{5})}$ as we stated.
The condition $\eta \in [\underline{\eta}, \bar{\eta}]$ in Theorem~\ref{th:EG4NE_convergence1} becomes \eqref{eq:EG4NE_stepsize2}.
Consequently, we can derive \eqref{eq:EG4NE_convergence_est2a} directly from the second and third lines of \eqref{eq:EG4NE_convergence_est1_details}, respectively.
The next two statements of Corollary~\ref{co:EG4NE_convergence_option2} are also consequences of Theorem~\ref{th:EG4NE_convergence1}.

To obtain the last-iterate convergence rates, we set $\beta = 1$, $\kappa_1 = 1$, and $\kappa_2 = 0$ in Theorem~\ref{th:EG4NE_convergence2}.
Hence, the condition  $L\rho \leq \hat{\Delta}$ in Theorem~\ref{th:EG4NE_convergence2} reduces to $L\rho \leq \frac{1}{16(1+\sqrt{2})}$ as in Corollary~\ref{co:EG4NE_convergence_option2}.
The choice of $\eta \in [ \underline{\eta}, \bar{\eta}] \cap [\underline{\hat{\eta}}, \bar{\hat{\eta}}]$ in Theorem~\ref{th:EG4NE_convergence2} becomes \eqref{eq:EG4NE_stepsize2}.
As a consequence, \eqref{eq:EG4NE_convergence_est2b} is obtained from \eqref{eq:EG4NE_last_iterate_convergence1_general}.
\Eproof
\end{proof}

\begin{remark}\label{re:EG4NE_last_iterate_remark1}
If $\rho = 0$, i.e., there exists $x^{\star}\in\zer{F}$ such that $\iprods{Fx, x-x^{\star}} \geq 0$ for all $x \in\dom{F}$, then the condition $\eta \in [ \underline{\eta}, \bar{\eta}] \cap [\underline{\hat{\eta}}, \bar{\hat{\eta}}]$ in Theorem~\ref{th:EG4NE_convergence2} reduces to $0 < \eta \leq \frac{1}{\sigma L}$ for $\sigma := [(1+\sqrt{2}) (\kappa_1 + \sqrt{2})]^{1/2}$.
In particular, for EG with $u^k := Fx^k$, we have $\kappa_1 = 0$, leading to $0 < \eta \leq \frac{1}{\sqrt{2+\sqrt{2}}L}$.
Similarly, for past-EG with $u^k := Fy^{k-1}$, we have $\kappa_1 = 1$, leading to $0 < \eta \leq \frac{1}{(1+\sqrt{2})L}$.
Hence, the ranges of stepsize $\eta$ for the last-iterate convergence rates are smaller than the ones for the best-iterate convergence rates. 
\end{remark}

\begin{remark}\label{re:EG4NE_comparison}
The two instances in Corollary~\ref{co:EG4NE_convergence} and Corollary~\ref{co:EG4NE_convergence_option2} were proven in our unpublished report \cite{luo2022last}, and then were further revised in \cite{gorbunov2022convergence}.
The last-iterate convergence rates for these two instances were proven in previous works such as \cite{golowich2020last} for the monotone case but with an additional assumption.
Note that the best-iterate rates for the monotone or the star-monotone case are classical, which can be found, e.g., in \cite{Facchinei2003,korpelevich1976extragradient}.
The last-iterate convergence for the monotone case can be found in recent works such as \cite{golowich2020last,gorbunov2022extragradient}.
The best-iterate rates for the co-hypomonotone or the star-co-hypomonotone case can be found in \cite{diakonikolas2021efficient}, while the last-iterate convergence rates were  proven in \cite{luo2022last}.
Nevertheless, in this subsection, we derive Corollary~\ref{co:EG4NE_convergence} and Corollary~\ref{co:EG4NE_convergence_option2} as special cases of Theorem~\ref{th:EG4NE_convergence1} and Theorem~\ref{th:EG4NE_convergence2}.
\end{remark}

\beforesec
\section{A Class of Extragradient Methods for Monotone Inclusions}\label{sec:EG4NI}
\aftersec
In this section, we go beyond \eqref{eq:NE} to generalize EG for a class of ``monotone'' \eqref{eq:NI}.
Our best-iterate convergence rate analysis generalizes classical results to a broader class of methods and allows $T$ to be  maximally $3$-cyclically monotone instead of a normal cone.
Our last-iterate convergence rate analysis is new and elementary, compared to, e.g., \cite{gorbunov2022convergence}.
We also specify our results to cover two classical instances similar to our results in Section~\ref{sec:EG4NE}.

\beforesubsec
\subsection{\mytb{A Class of Extragradient Methods for Solving \eqref{eq:NI}} }\label{subsect:EG4NI}
\aftersubsec
We propose the following generalized EG scheme (called GEG2) for solving \eqref{eq:NI}.
Starting from $x^0 \in \dom{\Phi}$, at each iteration $k\geq 0$, we update
\begin{equation}\label{eq:EG4NI}
\arraycolsep=0.2em
\left\{\begin{array}{lcl}
y^k &:= & J_{\frac{\eta}{\beta}T}(x^k - \frac{\eta}{\beta} u^k), \vspace{1ex}\\
x^{k+1} &:= & J_{\eta T}(x^k - \eta Fy^k),
\end{array}\right.
\tag{GEG2}
\end{equation}
where $J_{\eta T}$ is the resolvent of $\eta T$, $\eta > 0$ is a given stepsize, $\beta > 0$ is a scaling factor, and $u^k \in \R^p$ satisfies the following condition:
\begin{equation}\label{eq:EG4NI_u_cond}
\norms{Fx^k - u^k}^2 \leq \kappa_1\norms{Fx^k - Fy^{k-1}}^2 +  \kappa_2\norms{Fx^k - Fx^{k-1}}^2,
\end{equation}
for given parameters $\kappa_1 \geq 0$ and $\kappa_2 \geq 0$ and $x^{-1} = y^{-1} := x^0$.
Note that \eqref{eq:EG4NE_u_cond} also makes \eqref{eq:EG4NI} different from the hybrid extragradient method in \cite{Monteiro2011,solodov1999hybrid}.

Similar to \eqref{eq:EG4NE}, we have at least three concrete choices of $u^k$ as follows.
\begin{itemize}
\item[(a)]\mytb{Variant 1.} If $u^k := Fx^k$ (i.e. $\kappa_1 = \kappa_2 = 0$ in \eqref{eq:EG4NI_u_cond}), then we obtain $y^k := J_{\frac{\eta}{\beta}T}(x^k - \frac{\eta}{\beta} Fx^k)$.
Clearly, if $\beta = 1$, then we get the well-known \mytb{extragradient method}.
If $\beta \in (0, 1)$, then we obtain the \mytb{extragradient-plus} -- EG+ in \cite{diakonikolas2021efficient} but to solve \eqref{eq:NI}.

\item[(b)]\mytb{Variant 2.} If $u^k := Fy^{k-1}$ (i.e. $\kappa_1 = 1$ and $\kappa_2 = 0$), then we obtain $y^k := J_{\frac{\eta}{\beta}T}(x^k - \frac{\eta}{\beta} Fy^{k-1})$, leading to the \mytb{past-extragradient method} (or equivalently, \mytb{Popov's method} \cite{popov1980modification}) for solving \eqref{eq:NI}.

\item[(c)]\mytb{Variant 3.} We can construct $u^k := \alpha_1Fx^k + \alpha_2Fy^{k-1} + (1-\alpha_1 - \alpha_2)Fx^{k-1}$ as an affine combination of $Fx^k$, $Fy^{k-1}$ and $Fx^{k-1}$ for given constants $\alpha_1, \alpha_2 \in \R$.
Then, $u^k$ satisfies \eqref{eq:EG4NI_u_cond} with $\kappa_1 = (1 + c)(1-\alpha_1)^2$ and $\kappa_2 = (1+c^{-1})(1-\alpha_1-\alpha_2)^2$ by Young's inequality for any $c > 0$.
\end{itemize}
Clearly, when $T = \Nc_{\Xc}$, the normal cone of a nonempty, closed, and convex set $\Xc$, then $J_{\gamma T} = \proj_{\Xc}$, the projection onto $\Xc$ and hence, \eqref{eq:EG4NI} with \textbf{Variant 1} reduces to the extragradient variant for solving \eqref{eq:VIP} widely studied in the literature \cite{Facchinei2003,Konnov2001}.
In terms of computational complexity, \eqref{eq:EG4NI} with \mytb{Variants 1} and \mytb{3} requires two evaluations of $F$ at $x^k$ and $y^k$, and two evaluations of the resolvent $J_{\eta T}$ at each iteration.
It costs as twice as one iteration of the forward-backward splitting method \eqref{eq:FBS4NI}.
However, its does not require the co-coerciveness of $F$ to guarantee convergence.
Again, we use a scaling factor $\beta$ as in \eqref{eq:EG4NE}, which covers EG+ in \cite{diakonikolas2021efficient} as a special case.

Now, for given $\zeta^k \in Ty^k$ and $\xi^{k+1} \in Tx^{k+1}$, we denote  $\tilde{w}^k := Fx^k + \zeta^k$, and $\hat{w}^{k+1} := Fy^k + \xi^{k+1}$.
Then, we can rewrite \eqref{eq:EG4NI} equivalently to
\begin{equation}\label{eq:EG4NI_reform}
\arraycolsep=0.2em
\left\{\begin{array}{lclcl}
y^k &:= & x^k - \frac{\eta}{\beta}(u^k + \zeta^k) & = & x^k - \frac{\eta}{\beta}(\tilde{w}^k + u^k - Fx^k), \vspace{1ex}\\
x^{k+1} &:= & x^k - \eta (Fy^k + \xi^{k+1}) & = & x^k - \eta\hat{w}^{k+1},
\end{array}\right.
\end{equation}
where $\zeta^k  \in  Ty^k$ and $\xi^{k+1}  \in  Tx^{k+1}$.
This representation makes \eqref{eq:EG4NI} look like \eqref{eq:EG4NE}, and it is a key step for our convergence analysis below.

\beforesubsec
\subsection{\bf Key Estimates for Convergence Analysis}\label{subsec:EG4NI_key_lemmas}
\aftersubsec
We establish convergence rates of \eqref{eq:EG4NI} under the assumptions that $F$ is monotone and $L$-Lipschitz continuous, and $T$ is maximally $3$-cyclically monotone.
We first define
\begin{equation}\label{eq:EG4NI_w}
w^k := Fx^k + \xi^k \quad \text{for some}\quad \xi^k \in Tx^k.
\end{equation}
Then, the following lemmas are keys  to establishing the convergence of \eqref{eq:EG4NI}.

\begin{lemma}\label{le:EG4NI_key_estimate}
Suppose that $T$ is maximally $3$-cyclically monotone and $x^{\star} \in \zer{\Phi}$ exists.
Let  $\sets{(x^k, y^k)}$ be generated by \eqref{eq:EG4NI} and $w^k$ be defined by \eqref{eq:EG4NI_w}.
Then, for any $\gamma > 0$, we have
\begin{equation}\label{eq:EG4NI_key_est1}
\arraycolsep=0.2em
\begin{array}{lcl}
\norms{x^{k+1} - x^{\star}}^2 & \leq & \norms{x^k - x^{\star}}^2 -  (1 - \beta)\norms{x^{k+1} - x^k}^2 - \beta\norms{x^k - y^k}^2 \vspace{1ex}\\
&& - {~}  (\beta - \gamma)\norms{x^{k+1} - y^k}^2  + \frac{\eta^2}{\gamma}\norms{Fy^k - u^k}^2 \vspace{1ex}\\
&& - {~} 2\eta\iprods{Fy^k - Fx^{\star}, y^k - x^{\star}}.
\end{array}
\end{equation}
\end{lemma}

\begin{proof}
First, since $\xi^{k+1} \in Tx^{k+1}$, $\zeta^k \in Ty^k$, and  $\xi^{\star} = -Fx^{\star} \in Tx^{\star}$,  by the maximally $3$-cyclic monotonicity of $T$, we have 
\begin{equation*}
\iprods{\xi^{k+1}, x^{k+1} - x^{\star}} + \iprods{\xi^{\star}, x^{\star} - y^k} + \iprods{\zeta^k, y^k - x^{k+1}} \geq 0.
\end{equation*}
This inequality leads to 
\begin{equation*}
\iprods{\xi^{k+1} - \zeta^k, x^{k+1} - x^{\star}}  \geq \iprods{\zeta^k - \xi^{\star}, x^{\star} - y^k} = -\iprods{Fx^{\star} + \zeta^k, y^k - x^{\star}}.
\end{equation*}
Utilizing this estimate and the second line $x^k - x^{k+1} = \eta (Fy^k + \xi^{k+1})$ of \eqref{eq:EG4NI_reform}, for any $x^{\star} \in \zer{\Phi}$, we can derive that
\begin{equation}\label{eq:EG4NI_proof5}
\hspace{-2ex}
\arraycolsep=0.2em
\begin{array}{lcl}
\norms{x^{k+1} - x^{\star}}^2 & = & \norms{x^k - x^{\star}}^2 - 2\iprods{x^k - x^{k+1}, x^{k+1} - x^{\star}} - \norms{x^{k+1} - x^k}^2 \vspace{1ex}\\
&= & \norms{x^k - x^{\star}}^2 - 2\eta\iprods{Fy^k + \xi^{k+1}, x^{k+1} - x^{\star}} - \norms{x^{k+1} - x^k}^2 \vspace{1ex}\\
&= & \norms{x^k -x^{\star}}^2 - 2\eta\iprods{Fy^k + \zeta^k, x^{k+1} - x^{\star}} - \norms{x^{k+1} - x^k}^2 \vspace{1ex}\\
&& - {~} 2\eta\iprods{\xi^{k+1} - \zeta^k, x^{k+1} - x^{\star}} \vspace{1ex}\\
&\leq & \norms{x^k - x^{\star}}^2  - \norms{x^{k+1} - x^k}^2 - 2\eta\iprods{Fy^k + \zeta^k, x^{k+1} - y^k} \vspace{1ex}\\
&& - {~} 2\eta\iprods{Fy^k - Fx^{\star}, y^k - x^{\star}}.
\end{array}
\hspace{-2ex}
\end{equation}
Next, from the first line of  \eqref{eq:EG4NI_reform}, we have $\eta(Fy^k + \zeta^k) = \beta(x^k - y^k) + \eta(Fy^k - u^k)$.
Therefore, by the Cauchy-Schwarz inequality and Young's inequality, for any $\gamma > 0$, we can prove that
\begin{equation*}
\arraycolsep=0.2em
\begin{array}{lcl}
2\eta\iprods{Fy^k + \zeta^k, x^{k+1} - y^k} &= & 2\beta \iprods{x^k - y^k, x^{k+1} - y^k} + 2\eta\iprods{Fy^k - u^k, x^{k+1} - y^k} \vspace{1ex}\\
&\geq & \beta\left[\norms{x^k - y^k}^2 + \norms{x^{k+1} - y^k}^2 - \norms{x^{k+1} - x^k}^2\right]  \vspace{1ex}\\
&& - {~}  2\eta\norms{Fy^k - u^k}\norms{x^{k+1} - y^k} \vspace{1ex}\\
&\geq & \beta \norms{x^k - y^k}^2 + (\beta - \gamma) \norms{x^{k+1} - y^k}^2 \vspace{1ex}\\
&& - {~} \beta \norms{x^{k+1} - x^k}^2 - \frac{\eta^2}{\gamma}\norms{Fy^k - u^k}^2.
\end{array}
\end{equation*}
Finally, substituting this estimate into \eqref{eq:EG4NI_proof5}, we obtain \eqref{eq:EG4NI_key_est1}.
\Eproof
\end{proof}

\begin{lemma}\label{le:EG4NI_key_estimate2}
Suppose that  $F$ is monotone and $T$ is maximally $3$-cyclically monotone.
Let $\sets{(x^k, y^k)}$ be generated by \eqref{eq:EG4NI} and $w^k$ be defined by \eqref{eq:EG4NI_w}
Then, for $\omega \geq 0$, $\gamma > 0$, and $s > 0$, we have
\begin{equation}\label{eq:EG4NI_monotone_est1}
\arraycolsep=0.2em
\begin{array}{lcl}
\norms{w^{k+1}}^2  & + &  \omega \norms{w^{k+1} - \hat{w}^{k+1}}^2 \leq  \norms{w^k}^2 - (1 - \gamma)\norms{w^k -  \tilde{w}^k}^2  \vspace{1ex}\\
&& + {~} \left[\frac{1}{\gamma} + \frac{(1+\omega)(1+s) L^2\eta^2}{s} \right] \norms{Fx^k - u^k}^2  \vspace{1ex}\\
&& - {~}  \big[ 1 -  (1+\omega) (1+s) L^2\eta^2 \big] \norms{\hat{w}^{k+1} - \tilde{w}^k}^2.
\end{array}
\end{equation}
\end{lemma}

\begin{proof}
First, using the $3$-cyclic monotonicity of $T$ but with $\xi^k \in Tx^k$, we have 
\begin{equation*}
\iprods{\xi^{k+1}, x^{k+1} - x^k}  + \iprods{\xi^k, x^k - y^k} + \iprods{\zeta^k, y^k - x^{k+1}}  \geq 0.
\end{equation*}
Next, by the monotonicity of $F$, we get $\iprods{Fx^{k+1} - Fx^k, x^{k+1} - x^k} \geq 0$.
Summing up these inequalities and using $w^k = Fx^k + \xi^k$ and $\tilde{w}^k := Fx^k + \zeta^k$, we have 
\begin{equation}\label{eq:EG4NI_lm2_proof1}
\iprods{w^{k+1} - \tilde{w}^k, x^{k+1} - x^k} + \iprods{w^k - \tilde{w}^k, x^k - y^k}  \geq 0.
\end{equation}
From the second line of \eqref{eq:EG4NI_reform}, we have $x^{k+1} - x^k = -\eta(Fy^k + \xi^{k+1}) = -\eta\hat{w}^{k+1}$.
From the first line of  \eqref{eq:EG4NI_reform} and $\beta = 1$, we also have $x^k - y^k = \frac{\eta}{\beta}(\tilde{w}^k + u^k - Fx^k) = \eta \tilde{w}^k + \eta(u^k - Fx^k)$.
Substituting these expressions into \eqref{eq:EG4NI_lm2_proof1}, and using an elementary inequality $2\iprods{z, s} \leq \gamma \norms{s}^2 +  \frac{\norms{z}^2}{\gamma}$ for any $\gamma > 0$ and vectors $z$ and $s$, we can show that
\begin{equation*} 
\arraycolsep=0.2em
\begin{array}{lcl}
0 &\leq & 2 \iprods{\tilde{w}^k, \hat{w}^{k+1}} - 2 \iprods{w^{k+1}, \hat{w}^{k+1}}  + 2\iprods{w^k, \tilde{w}^k} - 2\norms{\tilde{w}^k}^2 + 2\iprods{w^k - \tilde{w}^k, u^k - Fx^k} \vspace{1ex}\\
&= & \norms{w^k}^2  - \norms{w^{k+1}}^2  + \norms{w^{k+1} - \hat{w}^{k+1}}^2 - \norms{\hat{w}^{k+1} - \tilde{w}^k}^2 \vspace{1ex}\\
&&  - {~} \norms{w^k - \tilde{w}^k}^2  + 2\iprods{w^k - \tilde{w}^k, u^k - Fx^k} \vspace{1ex}\\
&\leq & \norms{w^k}^2 - \norms{w^{k+1}}^2  + \norms{w^{k+1} - \hat{w}^{k+1}}^2 - \norms{\hat{w}^{k+1} - \tilde{w}^k}^2 \vspace{1ex}\\
&& - {~} (1 - \gamma)\norms{w^k - \tilde{w}^k}^2  + \frac{1}{\gamma}\norms{Fx^k - u^k}^2.
\end{array}
\end{equation*}
This inequality leads to 
\begin{equation*}
\arraycolsep=0.2em
\begin{array}{lcl}
\norms{w^{k+1}}^2 &\leq & \norms{w^k}^2 + \norms{w^{k+1} - \hat{w}^{k+1}}^2 + \frac{1}{\gamma}\norms{Fx^k - u^k}^2 \vspace{1ex}\\
&& - {~}  (1- \gamma)\norms{w^k - \tilde{w}^k}^2 - \norms{\hat{w}^{k+1} - \tilde{w}^k}^2.
\end{array}
\end{equation*}
Now, by the $L$-Lipschitz continuity of $F$, $x^{k+1} - y^k =  -\eta (\hat{w}^{k+1} - \tilde{w}^k) - \eta(Fx^k - u^k)$ from \eqref{eq:EG4NI_reform} with $\beta = 1$, and Young's inequality, for any $s > 0$, we have 
\begin{equation*} 
\arraycolsep=0.2em
\begin{array}{lcl}
\norms{w^{k+1} - \hat{w}^{k+1}}^2 & = & \norms{Fx^{k+1} - Fy^k}^2 \leq L^2\norms{x^{k+1} - y^k}^2 \vspace{1ex}\\
& = & L^2\eta^2\norms{\hat{w}^{k+1} - \tilde{w}^k + Fx^k - u^k}^2 \vspace{1ex}\\
& \leq & (1 + s)L^2\eta^2\norms{\hat{w}^{k+1} - \tilde{w}^k }^2 + \frac{(1 + s)L^2\eta^2}{s}\norms{Fx^k - u^k}^2.
\end{array}
\end{equation*}
Multiplying this inequality by $1+\omega \geq 0$ and adding the result to the last inequality above, we obtain \eqref{eq:EG4NI_monotone_est1}.
\Eproof
\end{proof}

\beforesubsec
\subsection{\bf The Sublinear Convergence Rates of \eqref{eq:EG4NI}}\label{subsec:EG4NI_analysis}
\aftersubsec
For simplicity of our analysis, we denote $r := \frac{\kappa_1 + \sqrt{\kappa_1^2 + 4\kappa_1}}{2}$ and define
\begin{equation}\label{eq:EG4NI_C_constants}
\begin{array}{lcl}
C_1 := \beta - (1+r)L\eta  \quad \text{and} \quad C_2 := 1 - \beta - \frac{\kappa_2(1+r)}{r}L\eta.
\end{array}
\end{equation}
Then, we can prove the convergence of \eqref{eq:EG4NI} under \eqref{eq:EG4NI_u_cond} as follows.

\begin{theorem}[Best-iterate convergence rate]\label{th:EG4NI_convergence1}
For Inclusion \eqref{eq:NI}, suppose that $\zer{\Phi} \neq \emptyset$, $F$  is $L$-Lipschitz continuous, $T$ is maximally 3-cyclically monotone, and there exists $x^{\star} \in \zer{\Phi}$  such that $\iprods{Fx - Fx^{\star}, x - x^{\star}} \geq 0$ for all $x\in\dom{F}$.

Let $\sets{(x^k, y^k)}$ be generated by \eqref{eq:EG4NI} starting from $x^0 \in\dom{\Phi}$ and $y^{-1} = x^{-1} := x^0$ such that $u^k$ satisfies \eqref{eq:EG4NI_u_cond}.
Suppose that $\eta$ is chosen as
\begin{equation}\label{eq:EG4NI_eta_choice}
\arraycolsep=0.2em
\begin{array}{lcl}
0 < \eta \leq \min\Big\{ \frac{\beta}{(1+r)L}, \frac{(1-\beta)r}{\kappa_2(1+r)L} \Big\} \quad\text{with} \quad r := \frac{\kappa_1 + \sqrt{\kappa_1^2 + 4\kappa_1}}{2}.
\end{array}
\end{equation}
Then, $C_1$ and $C_2$ given by \eqref{eq:EG4NI_C_constants} are nonnegative, and 
\begin{equation}\label{eq:EG4NI_convergence_est1a}
\sum_{k=0}^K \big[  C_1 \norms{y^k - x^k}^2 + C_1\norms{x^{k+1} - y^k}^2  + C_2 \norms{x^{k+1} - x^k}^2 \big] \leq \norms{x^0 - x^{\star}}^2.
\end{equation} 
Moreover, if $C_1 > 0$, then we also have
\begin{equation}\label{eq:EG4NI_convergence_est1b}
\hspace{-2ex}
\left\{\begin{array}{ll}
{\displaystyle\min_{0\leq k \leq K}} \norms{u^k + \zeta^k}^2 \leq \frac{1}{K+1}\sum_{k=0}^K\norms{u^k + \zeta^k}^2 \leq \frac{\beta^2 \norms{x^0 - x^{\star}}^2 }{ C_1\eta^2 (K+1)}, \vspace{1ex}\\
{\displaystyle\min_{0\leq k \leq K}} \norms{Fx^{k+1} + \xi^{k+1}}^2 \leq \frac{1}{K+1}\sum_{k=0}^K\norms{Fx^{k+1} + \xi^{k+1}}^2 \leq \frac{\Lambda \norms{x^0 - x^{\star}}^2 }{  \eta^2 (K+1)},
\end{array}\right.
\hspace{-3ex}
\end{equation} 
where $\Lambda := \frac{3[3C_1 + 2(C_1 + 3C_2)L^2\eta^2]}{3C_1(C_1 + 3C_2)}$.
Hence, we conclude that $\min_{0\leq k \leq K}\norms{Fx^k + \xi^k } = \BigOs{1/\sqrt{K}}$.
Furthermore, $\sets{\norms{x^k - x^{\star}}}$ is bounded and $\sets{x^k}$ converges to $x^{\star} \in \zer{\Phi}$.
In addition, we have
\begin{equation*}
\lim_{k\to\infty}\norms{x^k - y^k} = 0 \quad \text{and} \quad \lim_{k\to\infty}\norms{Fx^k + \xi^k } = 0.
\end{equation*}
\end{theorem}

\begin{proof}
First, for any $r > 0$, by Young's inequality, \eqref{eq:EG4NI_u_cond}, and the $L$-Lipschitz continuity of $F$, we can show that
\begin{equation*} 
\arraycolsep=0.2em
\begin{array}{lcl}
\norms{Fy^k - u^k}^2 & \leq & (1 + r)\norms{Fx^k - Fy^k}^2 + \frac{(1+r)}{r}\norms{Fx^k - u^k}^2 \vspace{1ex}\\
& \leq & (1+r)L^2\norms{x^k - y^k}^2 + \frac{(1+r)\kappa_1}{r}\norms{Fx^k - Fy^{k-1}}^2 \vspace{1ex}\\
&& + {~} \frac{(1+r)\kappa_2}{r}\norms{Fx^k - Fx^{k-1}}^2 \vspace{1ex}\\
& \leq & (1+r)L^2\norms{x^k - y^k}^2 + \frac{(1+r)\kappa_1L^2}{r}\norms{x^k - y^{k-1}}^2 \vspace{1ex}\\
&& + {~} \frac{(1+r)\kappa_2L^2}{r}\norms{x^k - x^{k-1}}^2.
\end{array}
\end{equation*}
Utilizing this expression and the condition $\iprods{Fy^k - Fx^{\star}, y^k - x^{\star}} \geq 0$, we can show from \eqref{eq:EG4NI_key_est1} that
\begin{equation*} 
\arraycolsep=0.1em
\begin{array}{lcl}
\norms{x^{k+1} - x^{\star}}^2 & + &  \frac{\kappa_1(1+r)L^2\eta^2}{r\gamma}\norms{x^{k+1} - y^k}^2 +  \frac{\kappa_2(1+r)L^2\eta^2}{r\gamma}\norms{x^{k+1} - x^k}^2   \leq  \norms{x^k - x^{\star}}^2 \vspace{1ex}\\
&& + {~} \frac{\kappa_1(1+r)L^2\eta^2}{r\gamma} \norms{x^k - y^{k-1}}^2 + \frac{ \kappa_2(1+r)L^2\eta^2}{r\gamma} \norms{x^k - x^{k-1}}^2  \vspace{1ex}\\
&& - {~} \big( \beta - \frac{(1+r)L^2\eta^2}{\gamma}\big) \norms{x^k - y^k}^2 -  \big(\beta - \gamma - \frac{\kappa_1(1+r)L^2\eta^2}{r\gamma} \big)\norms{x^{k+1} - y^k}^2 \vspace{1ex}\\
&& - {~}  \big( 1 - \beta - \frac{\kappa_2(1+r)L^2\eta^2}{r\gamma} \big)\norms{x^{k+1} - x^k}^2.
\end{array}
\end{equation*}
Now, we define a new potential function:
\begin{equation*}
\arraycolsep=0.2em
\begin{array}{lcl}
\hat{\Pc}_k :=  \norms{x^k - x^{\star}}^2 + \frac{\kappa_1(1+r)L^2\eta^2}{r\gamma} \norms{x^k - y^{k-1}}^2 + \frac{ \kappa_2(1+r)L^2\eta^2}{r\gamma} \norms{x^k - x^{k-1}}^2.
\end{array}
\end{equation*}
We also denote
\begin{equation*}
\arraycolsep=0.2em
\begin{array}{lcl}
C_1 :=  \beta - \frac{(1+r)L^2\eta^2}{\gamma}, \  \hat{C}_1 := \beta - \gamma - \frac{\kappa_1(1+r)L^2\eta^2}{r\gamma}, \ \text{and} \ C_2 :=  1 - \beta - \frac{\kappa_2(1+r)L^2\eta^2}{r\gamma}.
\end{array}
\end{equation*}
Then, it is obvious that $\hat{\Pc}_k \geq 0$.
Moreover, the last inequality is equivalent to
\begin{equation}\label{eq:EG4NI_th41_proof1} 
\arraycolsep=0.2em
\begin{array}{lcl}
C_1\norms{y^k - x^k}^2 + \hat{C}_1 \norms{x^{k+1} - y^k}^2 + C_2\norms{x^{k+1} - x^k}^2 \leq \hat{\Pc}_k - \hat{\Pc}_{k+1}.
\end{array}
\end{equation}
Let us choose $\gamma := L\eta$ and $r := \frac{\kappa_1 + \sqrt{\kappa_1^2 + 4\kappa_1}}{2}$. 
Then, we obtain $C_1 = \hat{C}_1 =  \beta - (1+r)L\eta$ and $C_2 = 1 - \beta - \frac{\kappa_2(1+r)L\eta}{r}$ as in \eqref{eq:EG4NI_C_constants}.
Furthermore, if we choose $\eta$ as in \eqref{eq:EG4NI_eta_choice}, then $C_1$ and $C_2$ are nonnegative.

Now, we sum up \eqref{eq:EG4NI_th41_proof1} from $k= 0$ to $K$, and note that $\hat{\Pc}_{K+1} \geq 0$ and $\hat{\Pc}_0 := \norms{x^0 - x^{\star}}^2$, we obtain \eqref{eq:EG4NI_convergence_est1a}.

Next, from the first line of \eqref{eq:EG4NI_reform}, we get $\norms{u^k + \zeta^k}^2 = \frac{\beta^2}{\eta^2}\norms{y^k - x^k}^2$.
Using this relation and \eqref{eq:EG4NI_convergence_est1a}, we obtain the first line of \eqref{eq:EG4NI_convergence_est1b}.

From the second line of \eqref{eq:EG4NI_reform}, for any $a_1 > 0$, $a_2 \in [0, 1]$, and $a_3 > 0$, by Young's inequality and the $L$-Lipschitz continuity of $F$, we can show that 
\begin{equation*}
\arraycolsep=0.1em
\begin{array}{lcl}
\eta^2\norms{Fx^{k+1} + \xi^{k+1}}^2 & \overset{\tiny \eqref{eq:EG4NI_reform} }{=} & \norms{x^{k+1} - x^k - \eta(Fx^{k+1} - Fy^k)}^2 \vspace{1ex}\\
& \leq & (1 + a_1)\norms{x^{k+1} - x^k}^2 + (1+a_1^{-1})L^2\eta^2\norms{x^{k+1} - y^k}^2 \vspace{1ex}\\
& \leq & a_2(1 + a_1)(1+a_3)\norms{y^k - x^k}^2 \vspace{1ex}\\
&& + {~} \big[ a_2(1+a_1)(1+a_3^{-1}) + (1+a_1^{-1})L^2\eta^2 \big] \norms{x^{k+1} - y^k}^2 \vspace{1ex}\\
&&  + {~} (1+a_1)(1-a_2)\norms{x^{k+1} - x^k}^2.
\end{array}
\end{equation*}
In this case, we assume that $\Lambda C_1 = a_2(1 + a_1)(1+a_3) = \big[ a_2(1+a_1)(1+a_3^{-1}) + (1+a_1^{-1})L^2\eta^2\big]$, and $\Lambda C_2 =   (1+a_1)(1-a_2)$ for some $\Lambda > 0$.
By appropriately chosen $a_1$, $a_2$, and $a_3 := 2$, we obtain 
\begin{equation*}
\arraycolsep=0.1em
\begin{array}{lcl}
\Lambda := \frac{3[3C_1 + 2(C_1 + 3C_2)L^2\eta^2]}{3C_1(C_1 + 3C_2)}.
\end{array}
\end{equation*}
Hence, the last inequality is equivalent to 
\begin{equation*}
\arraycolsep=0.1em
\begin{array}{lcl}
\frac{\eta^2}{\Lambda} \norms{Fx^{k+1} + \xi^{k+1} }^2 & \leq & C_1\norms{x^k - y^k}^2 + C_1\norms{x^{k+1} - y^k}^2 + C_2\norms{x^{k+1} - x^k}^2.
\end{array}
\end{equation*}
Combining this inequality and \eqref{eq:EG4NI_convergence_est1a}, we obtain the second line of \eqref{eq:EG4NI_convergence_est1b}.
The remaining statements are proven similar to Theorem~\ref{th:EG4NE_convergence1} and we omit.
\Eproof
\end{proof}

Next, we prove the last-iterate convergence rate of \eqref{eq:EG4NI} when $\kappa_2 = 0$.

\begin{theorem}[Last-iterate convergence rate]\label{th:EG4NI_convergence1b}
For Inclusion \eqref{eq:NI}, suppose that $\zer{\Phi} \neq \emptyset$, $F$ is monotone and $L$-Lipschitz continuous, and $T$ is maximally 3-cyclically monotone.

Let $\sets{(x^k, y^k)}$ be generated by \eqref{eq:EG4NI} starting from $x^0 \in \dom{\Phi}$ and $y^{-1} = x^{-1} := x^0$ such that $u^k$ satisfies \eqref{eq:EG4NI_u_cond}.
Suppose that $\eta$ is chosen as
\begin{equation}\label{eq:EG4NI_eta_choice_1b}
\arraycolsep=0.2em
\begin{array}{lcl}
0 < \eta <  \frac{\beta}{(1 + r)L} \quad \text{with} \quad r := \frac{\kappa_1 + \sqrt{\kappa_1^2 + 4\kappa_1}}{2}.
\end{array}
\end{equation}
Then, for $w^k := Fx^k + \xi^k$ with $\xi^k \in Tx^k$ and $\omega :=  \frac{\kappa_1[ s + (1+s)L^2\eta^2]}{s - (1 + s)\kappa_1L^2\eta^2}$ with $s := \frac{(1+r)^2 - 2\kappa_1 - 1 + \sqrt{[(1+r)^2 - 2\kappa_1 - 1]^2 - 4\kappa_1(1+\kappa_1) }}{2(1 + \kappa_1)} \geq 0$, we have
\begin{equation}\label{eq:EG4NI_monotone_1b}
\arraycolsep=0.2em
\begin{array}{ll}
& \norms{w^{k+1}}^2 + \omega\norms{Fx^{k+1} - Fy^k}^2 \leq \norms{w^k}^2 + \omega\norms{Fx^k - Fy^{k-1}}^2, \vspace{2ex}\\
& \norms{Fx^{K+1} + \xi^{K+1}}^2 \leq \frac{\Gamma \norms{x^0 - x^{\star}}^2}{K+1},
\end{array}
\end{equation}
where $\Gamma := \frac{\Lambda}{\eta^2} + \frac{L^2\omega}{C_1}$.
Hence, we conclude that $\norms{Fx^K + \xi^K} = \BigOs{1/\sqrt{K}}$ on the last iterate $x^K$ for some $\xi^K \in Tx^K$.
\end{theorem}

\begin{proof}
First, since $w^k = Fx^k + \xi^k$ and $\hat{w}^k = Fy^{k-1} + \xi^k$ for $\xi^k \in Tx^k$, we have $Fx^k - Fy^{k-1} = w^k - \hat{w}^{k-1}$.
Next, since $\kappa_2 = 0$,  \eqref{eq:EG4NI_u_cond} reduces to $\norms{Fx^k - u^k}^2 \leq \kappa_1\norms{Fx^k - Fy^{k-1}}^2 = \kappa_1\norms{w^k - \hat{w}^k}^2$ due to $Fx^k - Fy^{k-1} = w^k - \hat{w}^k$.
Using this relation into \eqref{eq:EG4NI_monotone_est1} and choosing $\gamma = 1$, we get
\begin{equation}\label{eq:EG4NI_th42_proof2} 
\hspace{-3ex}
\arraycolsep=0.1em
\begin{array}{lcl}
\norms{w^{k+1}}^2  & + &  \omega \norms{w^{k+1} - \hat{w}^{k+1}}^2 \leq  \norms{w^k}^2 + \kappa_1\left[1 + \frac{(1+\omega)(1+s) L^2\eta^2}{s} \right] \norms{w^k - \hat{w}^k}^2  \vspace{1ex}\\
&& - {~}  \big[ 1 -  (1+\omega) (1+s) L^2\eta^2 \big] \norms{\hat{w}^{k+1} - \tilde{w}^k}^2.
\end{array}
\hspace{-5ex}
\end{equation}
Now, we choose $\omega = \frac{\kappa_1[ s + (1+s)L^2\eta^2]}{s - (1 + s)\kappa_1L^2\eta^2}$, provided that $(1 + s)\kappa_1L^2\eta^2 < s$.
Then, we have  $\omega = \kappa_1\big[1 + \frac{(1+\omega)(1+s) L^2\eta^2}{s} \big]$.

We also need to guarantee $(1+\omega) (1 + s) L^2\eta^2 \leq 1$, which is equivalent to $L^2\eta^2 \leq \frac{s}{(1 + s)(\kappa_1s + \kappa_1 + s)}$.
If we choose $s := \frac{(1+r)^2 - 2\kappa_1 - 1 + \sqrt{[(1+r)^2 - 2\kappa_1 - 1]^2 - 4\kappa_1(1+\kappa_1) }}{2(1 + \kappa_1)}$ for $r := \frac{\kappa_1 + \sqrt{\kappa_1^2 + 4\kappa_1}}{2}$, then we have $\frac{s}{(1 + s)(\kappa_1s + \kappa_1 + s)} = \frac{1}{(1+r)^2}$.
Furthermore, if we choose $\eta < \frac{1}{(1+r)L}$ as in \eqref{eq:EG4NI_eta_choice_1b}, then $L^2\eta^2 \leq \frac{s}{(1 + s)(\kappa_1s + \kappa_1 + s)}$.

Next, under \eqref{eq:EG4NI_eta_choice_1b} and the choice of $\omega$, \eqref{eq:EG4NI_th42_proof2}  reduces to
\begin{equation}\label{eq:EG4NI_th42_proof3} 
\arraycolsep=0.2em
\begin{array}{lcl}
\norms{w^{k+1}}^2  & + &  \omega \norms{w^{k+1} - \hat{w}^{k+1}}^2 \leq  \norms{w^k}^2 +  \omega \norms{w^k - \hat{w}^k}^2.
\end{array}
\end{equation}
which proves the first line of \eqref{eq:EG4NI_monotone_1b} due to $Fx^k - Fy^{k-1} = w^k - \hat{w}^k$.

Finally, from \eqref{eq:EG4NI_convergence_est1a} and the second line of \eqref{eq:EG4NI_convergence_est1b}, we have
\begin{equation*}
\arraycolsep=0.2em
\begin{array}{llcl}
& \frac{1}{K+1} \sum_{k=0}^K \norms{Fx^{k+1} - Fy^{k}}^2 & \leq & \frac{1}{K+1}\sum_{k=0}^KL^2\norms{x^{k+1} - y^{k}}^2 \leq \frac{L^2\norms{x^0 - x^{\star}}^2}{ C_1(K+1)}, \vspace{1ex}\\
& \frac{1}{K+1} \sum_{k=0}^K\norms{Fx^{k+1} + \xi^{k+1}}^2 & \leq & \frac{\Lambda \norms{x^0 - x^{\star}}^2 }{  \eta^2 (K+1)}.
\end{array}
\end{equation*}
Multiplying the first line by $\omega$ and adding the result to the second line, and then using \eqref{eq:EG4NI_th42_proof3} with $w^{k+1} := Fx^{k+1} + \xi^{k+1}$ and $w^{k+1} - \hat{w}^{k+1} = Fx^{k+1} - Fy^k$, we obtain the second line of \eqref{eq:EG4NI_monotone_1b}, where $\Gamma := \frac{L^2\omega}{C_1} + \frac{\Lambda}{\eta^2}$.
\Eproof
\end{proof}

\beforesubsec
\subsection{\bf The Sublinear Convergence Rates of Two Instances}\label{subsec:EG4NI_special_cases}
\aftersubsec
Now, we specify Theorems~\ref{th:EG4NI_convergence1} and \ref{th:EG4NI_convergence1b} for two instances: the extragradient method (\textbf{Variant 1}) and Popov's past-extragradient method (\textbf{Variant 2}).

\begin{corollary}[The EG Method]\label{co:EG4NI_convergence_case1}
For Inclusion~\eqref{eq:NI}, suppose that $\zer{\Phi} \neq \emptyset$, $F$  is $L$-Lipschitz continuous, $T$ is maximally $3$-cyclically monotone, and there exists  $x^{\star} \in \zer{\Phi}$ such that $\iprods{Fx - Fx^{\star}, x - x^{\star}} \geq 0$ for all $x\in\dom{F}$.
Let $\sets{(x^k, y^k)}$ be generated by \eqref{eq:EG4NI} starting from $x^0 \in \dom{\Phi}$ using $u^k := Fx^k$ $($\textbf{Variant 1}$)$  and  $0 < \eta \leq \frac{\beta}{L}$.

\noindent$\mathrm{(a)}$~\textbf{\textit{The best-iterate convergence rate of EG}}.
Then, we have
\begin{equation}\label{eq:EG4NI_convergence_case1_01}
\arraycolsep=0.2em
\begin{array}{l}
{\displaystyle\min_{0\leq k \leq K}}\norms{Fx^{k+1} + \xi^{k+1}}^2 \leq \frac{1}{K+1}\sum_{k=0}^{K}\norms{Fx^{k+1} + \xi^{k+1}}^2 \leq \frac{\Lambda\norms{x^0 - x^{\star}}^2}{\eta^2( K+1 )},
\end{array}
\end{equation} 
where $\xi^k \in Tx^k$ and  $\Lambda$ is defined in \eqref{eq:EG4NI_convergence_est1b}.

Furthermore, $\sets{\norms{x^k - x^{\star}}}$ is nonincreasing and $\sets{x^k}$ converges to $x^{\star} \in \zer{\Phi}$.
We also have
\begin{equation}\label{eq:EG4NI_convergence_case1_01b}
\lim_{k\to\infty}\norms{x^k - y^k} =  \lim_{k\to\infty}\norms{Fx^k + \xi^k} = \lim_{k\to\infty}\norms{Fy^k + \zeta^k} = 0,
\end{equation}
where $\xi^k \in Tx^k$ and $\zeta^k \in Ty^k$.

\noindent$\mathrm{(b)}$~\textbf{\textit{The last-iterate convergence rate of EG}}.
If, in addition, $F$ is monotone, then $\sets{\norms{Fx^k + \xi^k}}$ is monotonically non-increasing and
\begin{equation}\label{eq:EG4NI_monotone_est1_case1}
\hspace{-1ex}
 \norms{Fx^{K+1} + \xi^{K+1} }^2 \leq \frac{\Lambda \norms{x^0 - x^{\star}}^2}{\eta^2( K+1) }.
\hspace{-2ex}
\end{equation}
Hence, we conclude that $\norms{Fx^K + \xi^K} = \BigOs{1/\sqrt{K}}$ on the last-iterate $x^k$.

\noindent$\mathrm{(c)}$~\textbf{\textit{Convergence in $\norms{\Gc_{\eta}x^k}$}}.
For $\Gc_{\eta }$ defined by \eqref{eq:FB_residual}, we also have
\begin{equation}\label{eq:EG4NI_convergence_case1_02}
\arraycolsep=0.2em
\begin{array}{l}
{\displaystyle\min_{0\leq k \leq K}} \norms{ \Gc_{\eta}x^k} = \BigO{\frac{1}{\sqrt{K}}} \quad \text{and} \quad \norms{ \Gc_{\eta }x^K} = \BigO{\frac{1}{ \sqrt{K}}}.
\end{array}
\end{equation}
\end{corollary}

\begin{proof}
For \textbf{Variant 1} with $u^k := Fx^k$, we have $\kappa_1 = \kappa_2 = 0$, leading to $r = 0$.
Hence, \eqref{eq:EG4NI_eta_choice} reduces to $0 < \eta \leq \frac{\beta}{L}$, and \eqref{eq:EG4NI_convergence_case1_01} is a direct consequence of the second line of \eqref{eq:EG4NI_convergence_est1b}.
Combining  \eqref{eq:EG4NI_convergence_case1_01} and \eqref{eq:EG4NI_convergence_est1a}, we obtain \eqref{eq:EG4NI_convergence_case1_01b}.

Next, since $\kappa_1 = 0$, we have $\omega = 0$ in Theorem~\ref{th:EG4NI_convergence1b}.
Therefore, the second estimate of \eqref{eq:EG4NI_monotone_1b} reduces to \eqref{eq:EG4NI_monotone_est1_case1}, where $\Gamma = \frac{\Lambda}{\eta^2}$.
Finally, combining our results and \eqref{eq:FBR_bound2} we obtain the remaining conclusions. 
\Eproof
\end{proof}

\begin{corollary}[The Past-EG Method]\label{co:EG4NI_convergence_case2}
For Inclusion~\eqref{eq:NI}, suppose that $\zer{\Phi} \neq \emptyset$, $F$ is $L$-Lipschitz continuous, $T$ is maximally $3$-cyclically monotone, and there exists $x^{\star} \in \zer{\Phi}$ such that  $\iprods{Fx - Fx^{\star}, x - x^{\star}} \geq 0$ for all $x\in\dom{F}$. 
Let $\sets{(x^k, y^k)}$ be generated by \eqref{eq:EG4NI} starting from $x^0 \in \dom{\Phi}$ and $y^{-1} := x^0$ using $u^k := Fy^{k-1}$ $($\textbf{Variant 2}$)$ and $\eta$ such that $0 < \eta \leq \frac{2\beta}{(3+\sqrt{5})L}$.

\noindent$\mathrm{(a)}$~\textbf{\textit{The best-iterate convergence rate of Past-EG}}.
Then, we have
\begin{equation}\label{eq:EG4NI_convergence_case2_01}
\arraycolsep=0.2em
\begin{array}{lcl}
{\displaystyle\min_{0\leq k \leq K }\norms{Fx^{k+1} + \xi^{k+1} }^2} & \leq &  \frac{1}{K+1}\sum_{k=0}^{K}  \norms{Fx^{k+1} + \xi^{k+1} }^2  \leq  \frac{\Lambda \norms{x^0 - x^{\star}}^2}{\eta^2( K+ 1) },
\end{array}
\end{equation} 
where $\xi^k \in Tx^k$ and  $\Lambda$ is defined in \eqref{eq:EG4NI_convergence_est1b}.
Moreover, $\sets{\norms{x^k - x^{\star}}}$ is bounded and $\sets{x^k}$ converges to  $x^{\star}$.
Furthermore, \eqref{eq:EG4NI_convergence_case1_01b} still holds.

\noindent$\mathrm{(b)}$~\textbf{\textit{The last-iterate convergence rate of Past-EG}}.
If, in addition, $F$ is monotone, then we have 
\begin{equation}\label{eq:EG4NI_monotone_est2_case2}
\hspace{-3ex}
\arraycolsep=0.0em
\begin{array}{ll}
& \norms{Fx^{k+1} + \xi^{k+1}}^2 + \omega \norms{Fx^{k+1} - Fy^k}^2 \leq \norms{Fx^k + \xi^k}^2 + \omega \norms{Fx^k - Fy^{k-1}}^2 \vspace{2ex}\\
& \norms{Fx^{K+1} + \xi^{K+1} }^2 \leq   \frac{\Lambda \norms{x^0 - x^{\star}}^2}{ \eta^2 (K + 1) }.
\end{array}
\hspace{-5ex}
\end{equation}
Thus we conclude that $\norms{Fx^K + \xi^K} = \BigOs{1/\sqrt{K}}$ on the last iterate $x^K$.

\noindent$\mathrm{(c)}$~\textbf{\textit{Convergence in $\norms{\Gc_{\eta}x^k}$}}.
For $\Gc_{\eta }$ defined by \eqref{eq:FB_residual}, we also have
\begin{equation}\label{eq:EG4NI_convergence_case2_02}
\hspace{-2ex}
\arraycolsep=0.0em
\begin{array}{l}
{\displaystyle\min_{0 \leq k \leq K}} \norms{\Gc_{\eta}x^{k} } = \BigO{\frac{1}{\sqrt{K}}} \quad \text{and} \quad  \norms{\Gc_{\eta}x^K } = \BigO{\frac{1}{ \sqrt{K}}}.
\end{array}
\hspace{-2ex}
\end{equation}
\end{corollary}

\begin{proof}
For \textbf{Variant 2} with $u^k := Fy^{k-1}$, we have $\kappa_1 = 1$ and $\kappa_2 = 0$ leading to $r = \frac{1 + \sqrt{5}}{2}$.
Hence, \eqref{eq:EG4NI_eta_choice} reduces to $0 < \eta \leq \frac{2\beta}{(3+\sqrt{5})L}$, and \eqref{eq:EG4NI_convergence_case2_01} is a direct consequence of the second line of \eqref{eq:EG4NI_convergence_est1b}.
Next, since $\kappa_1 = 1$, we have $\omega > 0$ in Theorem~\ref{th:EG4NI_convergence1b}.
Therefore, the second estimate of \eqref{eq:EG4NI_monotone_1b} reduces to \eqref{eq:EG4NI_monotone_est2_case2}.
Finally, combining our results and \eqref{eq:FBR_bound2} we obtain the remaining conclusions. 
\Eproof
\end{proof}

\begin{remark}\label{re:EG4NI_remark1}
If $\beta = 1$, then the condition $0 < \eta < \frac{1}{L}$ in Corollary~\ref{co:EG4NI_convergence_case1} is the same as in the classical EG method, see \cite{Facchinei2003}.
Similarly, when $\beta = 1$, the condition $0 < \eta \leq \frac{2}{(3+\sqrt{5})L}$ in Corollary~\ref{co:EG4NI_convergence_case2} is slightly relaxed than the condition $0 < \eta \leq \frac{1}{3L}$ in \cite{popov1980modification}.
It remains open to establish both the best-iterate and last-iterate convergence rates of \eqref{eq:EG4NI} under the weak-Minty solution condition \cite{diakonikolas2021efficient} and the co-hypomonotonicity of $\Phi$.
\end{remark}

\beforesec
\section{A Class of Forward-Backward-Forward Splitting Methods for \eqref{eq:NI}}\label{sec:FBFS4NI}
\aftersec
Alternative to \eqref{eq:EG4NI}, we now generalize the FBFS method in \cite{tseng2000modified} for solving \eqref{eq:NI}, but when a weak-Minty solution exists.

\beforesubsec
\subsection{\mytb{A Class of Forward-Backward-Forward Splitting Methods}}\label{subsec:FBF4NI}
\aftersubsec
The forward-backward-forward splitting (FBFS) method was proposed by P. Tseng in \cite{tseng2000modified} for solving \eqref{eq:NI}, which is originally called a \textit{modified forward-backward splitting} method.

Now, we propose to consider the following generalized FBFS scheme for solving \eqref{eq:NI}.
Starting from $x^0 \in \dom{\Phi}$, at each iteration $k\geq 0$, we update
\begin{equation}\label{eq:FBFS4NI}
\arraycolsep=0.2em
\left\{\begin{array}{lcl}
y^k &:= & J_{\frac{\eta}{\beta} T}(x^k - \frac{\eta}{\beta} u^k), \vspace{1ex}\\
x^{k+1} &:= & \beta y^k + (1-\beta)x^k - \eta(Fy^k - u^k),
\end{array}\right.
\tag{GFBFS2}
\end{equation}
where $J_{\frac{\eta}{\beta} T}$ is the resolvent of $\frac{\eta}{\beta} T$, $\eta > 0$ is a given stepsize, and $\beta > 0$ is a scaling factor, and $u^k \in \R^p$ satisfies the following condition:
\begin{equation}\label{eq:FBFS4NI_u_cond}
\norms{Fx^k - u^k}^2 \leq \kappa_1\norms{Fx^k - Fy^{k-1}}^2 +  \kappa_2\norms{Fx^k - Fx^{k-1}}^2,
\end{equation}
for given constants $\kappa_1 \geq 0$ and $\kappa_2 \geq 0$ and $x^{-1} = y^{-1} := x^0$.
Here, we have our flexibility to choose $u^k$.
Let us consider three special cases of $u^k$ as before.

\begin{itemize}
\itemsep=0.1em
\item[(a)]\mytb{Variant 1.} If we choose $u^k := Fx^k$, then we obtain a variant of \mytb{Tseng's FBFS method}.
In particular, if $\beta = 1$, then we get exactly \mytb{Tseng's FBFS method} in \cite{tseng2000modified} for solving \eqref{eq:NI}.
Note that one can extend \eqref{eq:FBFS4NI} to cover the case $\zer{\Phi}\cap\Cc \neq\emptyset$ for some subset $\Cc$ of $\R^p$ as presented in \cite{tseng2000modified}.
Nevertheless, for simplicity, we assume that $\Cc = \R^p$.
As shown in \eqref{eq:FBFS4NE}, if $T = 0$, then \eqref{eq:FBFS4NI} reduces to the extragradient method \eqref{eq:EG4NE} for \eqref{eq:NE}.
However, if $T\neq 0$, then \eqref{eq:FBFS4NI} is different from \eqref{eq:EG4NI}.

\item[(b)]\mytb{Variant 2.} If we choose $u^k := Fy^{k-1}$, where $y^{-1} := x^0$, then we obtain a \mytb{past-FBFS variant}.
This variant can also be referred to as a generalized variant of the \mytb{optimistic gradient (OG) method}, see, e.g., \cite{daskalakis2018training,mokhtari2020unified,mokhtari2020convergence}.
If $\beta = 1$, then $y^{k+1} = J_{\eta T}(x^{k+1} - \eta Fy^k)$ and $x^{k+1} = y^k - \eta(Fy^k - Fy^{k-1})$.
Combining these two expressions, \eqref{eq:FBFS4NI} reduces to 
\begin{equation}\label{eq:FRBS4NI2}
\arraycolsep=0.2em
\begin{array}{lcl}
y^{k+1} & = &  J_{\eta T}\left(y^k - \eta (2Fy^k - Fy^{k-1}) \right).
\end{array}
\tag{FRBS2}
\end{equation}
This is exactly the \mytb{forward-reflected-backward splitting (FRBS) method} proposed in \cite{malitsky2020forward}.

\item[(c)]\mytb{Variant 3.} We can form $u^k := \alpha_1Fx^k + \alpha_2Fy^{k-1} + (1-\alpha_1 - \alpha_2)Fx^{k-1}$ as an affine combination of $Fx^k$, $Fy^{k-1}$ and $Fx^{k-1}$ for given constants $\alpha_1, \alpha_2 \in \R$.
Then, $u^k$ satisfies \eqref{eq:FBFS4NI_u_cond} with $\kappa_1 = (1 + c)(1-\alpha_1)^2$ and $\kappa_2 = (1 + c^{-1})(1-\alpha_1-\alpha_2)^2$ for some $c > 0$ by Young's inequality.
\end{itemize}
Compared to \eqref{eq:EG4NI}, we do not require $T$ to be monotone in \eqref{eq:FBFS4NI}.
However, to guarantee the well-definedness of $\sets{(x^k, y^k)}$, we need $y^k \in \ran{J_{\eta T}}$ and $y^k \in \dom{F}$.
Hence, we can assume that $\ran{J_{\eta T}}\subseteq\dom{F} = \R^p$ and $\dom{J_{\eta T}} = \R^p$. 
This requirement makes \eqref{eq:FBFS4NI} cover a broader class of problems than \eqref{eq:EG4NI}, and it obviously holds if $T$ is maximally monotone and $F$ is partially monotone and Lipschitz continuous as in \eqref{eq:EG4NI}.
Here, we assume that $J_{\eta T}$ is single-valued, but it can be extended to multivalued $J_{\eta T}$.
In addition, \eqref{eq:FBFS4NI} only requires one evaluation of $J_{\eta T}$ instead of two  as in \eqref{eq:EG4NI}, reducing the per-iteration complexity when $J_{\eta T}$ is expensive.

Similar to \eqref{eq:EG4NI_reform}, we can rewrite  \eqref{eq:FBFS4NI} equivalently to 
\begin{equation}\label{eq:FBFS4NI_reform}
\arraycolsep=0.2em
\left\{\begin{array}{lcl}
y^k &:= &  x^k - \frac{\eta}{\beta}(u^k + \zeta^k), \quad \zeta^k \in Ty^k, \vspace{1ex}\\
x^{k+1} &:= & x^k + \beta (y^k - x^k) - \eta(Fy^k - u^k) = x^k - \eta(Fy^k + \zeta^k).
\end{array}\right.
\end{equation}
This representation is an important step for our convergence analysis below.

\beforesubsec
\subsection{\bf Key Estimates for Convergence Analysis}\label{subsec:FBFS4NI_key_lemmas}
\aftersubsec
The following lemma is key to establishing convergence of \eqref{eq:FBFS4NI}.

\begin{lemma}\label{le:FBFS4NI_key_estimate}
Suppose that  $\sets{(x^k, y^k)}$ is generated by \eqref{eq:FBFS4NI} and $T$ is not necessary monotone, but $\ran{J_{\eta T}} \subseteq \dom{F} = \R^p$ and $\dom{J_{\eta T}} = \R^p$.
Then, for any $\gamma > 0$, any $x^{\star} \in \zer{\Phi}$, we have
\begin{equation}\label{eq:FBFS4NI_key_est1}
\arraycolsep=0.2em
\begin{array}{lcl}
\norms{x^{k+1} - x^{\star}}^2 & \leq & \norms{x^k - x^{\star}}^2 - \beta\norms{x^k - y^k}^2 +  \frac{\eta^2}{\gamma}\norms{Fy^k - u^k}^2  \vspace{1ex}\\
&& - {~}  (\beta - \gamma)\norms{x^{k+1} - y^k}^2 -  2\eta\iprods{Fy^k + \zeta^k, y^k - x^{\star}} \vspace{1ex}\\
&& - {~}  (1 - \beta)\norms{x^{k+1} - x^k}^2.
\end{array}
\end{equation}
For any $r > 0$, let us introduce the following potential function:
\begin{equation}\label{eq:FBFS4NI_Lyapunov_func}
\hspace{-1ex}
\arraycolsep=0.1em
\begin{array}{lcl}
\Pc_k &:= & \norms{x^k - x^{\star} }^2 +  \tfrac{\kappa_1(1+r)L^2\eta^2}{r \gamma}\norms{x^k - y^{k-1}}^2 +  \tfrac{\kappa_2(1+r)L^2\eta^2}{r \gamma}\norms{x^k - x^{k-1}}^2.
\end{array}
\hspace{-1ex}
\end{equation}
If $u^k$ satisfies \eqref{eq:FBFS4NI_u_cond} and there exist $x^{\star} \in \zer{\Phi}$ and $\rho \geq 0$ such that $\iprods{w, x - x^{\star}} \geq -\rho\norms{w}^2$ for all $(x, w) \in \gra{\Phi}$, then for any  $s > 0$ and $\mu \in [0, 1]$, we have
\begin{equation}\label{eq:FBFS4NI_key_estimate1b}
\hspace{-2ex}
\arraycolsep=0.2em
\begin{array}{lcl}
\Pc_{k+1} & \leq & \Pc_k - \left( \beta - \frac{(1 + r)L^2\eta^2}{ \gamma} - \frac{2\mu\rho(1 + s)}{s \eta}  \right) \norms{y^k - x^k}^2 \vspace{1ex} \\
&& - {~} \left(\beta - \gamma - \frac{\kappa_1(1 + r) L^2\eta^2}{r \gamma} - \frac{2\mu\rho(1 + s)}{\eta}  \right)\norms{x^{k+1} - y^k}^2 \vspace{1ex}\\
&& - {~} \big(1 - \beta - \frac{\kappa_2(1 + r) L^2\eta^2}{r \gamma} - \frac{2(1-\mu)\rho}{\eta} \big)\norms{x^{k+1} - x^k}^2.
\end{array}
\hspace{-2ex}
\end{equation}
\end{lemma}

\begin{proof}
First, combining the first and second lines of  \eqref{eq:FBFS4NI_reform}, we get $x^{k+1} = x^k - \beta(x^k - y^k) + \eta(u^k - Fy^k) = x^k - \eta(Fy^k + \zeta^k)$.
Using this relation, we have
\begin{equation*} 
\arraycolsep=0.2em
\begin{array}{lcl}
\norms{x^{k+1} - x^{\star}}^2  &= & \norms{x^k - x^{\star}}^2 - 2\eta\iprods{Fy^k + \zeta^k, x^{k+1} - y^k} - \norms{x^{k+1} - x^k}^2 \vspace{1ex}\\
&& - {~}  2\eta\iprods{Fy^k + \zeta^k, y^k - x^{\star}}.
\end{array} 
\end{equation*}
Next, from the first line of \eqref{eq:FBFS4NI_reform}, we have $\eta(Fy^k + \zeta^k) = \beta(x^k - y^k) + \eta(Fy^k - u^k)$.
Hence, by the Cauchy-Schwarz's and Young's inequality, for any $\gamma > 0$, we can derive that
\begin{equation*}
\arraycolsep=0.2em
\begin{array}{lcl}
2\eta\iprods{Fy^k + \zeta^k, x^{k+1} - y^k} &= & 2\beta \iprods{x^k - y^k, x^{k+1} - y^k} + 2\eta\iprods{Fy^k - u^k, x^{k+1} - y^k} \vspace{1ex}\\
&\geq & \beta\left[\norms{x^k - y^k}^2 + \norms{x^{k+1} - y^k}^2 - \norms{x^{k+1} - x^k}^2\right] \vspace{1ex}\\
&& - {~} 2\eta\norms{Fy^k - u^k}\norms{x^{k+1} - y^k} \vspace{1ex}\\
&\geq & \beta \norms{x^k - y^k}^2 + (\beta - \gamma) \norms{x^{k+1} - y^k}^2  \vspace{1ex}\\
&& - {~} \beta \norms{x^{k+1} - x^k}^2 -  \frac{\eta^2}{\gamma}\norms{Fy^k - u^k}^2.
\end{array}
\end{equation*}
Finally, combining the last two expressions, we obtain \eqref{eq:FBFS4NI_key_est1}.
The proof of \eqref{eq:FBFS4NI_key_estimate1b} is similar to the proof of Lemma~\ref{le:EG4NE_key_estimate1b}, and we omit.
\Eproof
\end{proof}

\beforesubsec
\subsection{\bf The Sublinear Best-Iterate Convergence Rate of \eqref{eq:FBFS4NI}}\label{subsec:FBFS4NI_analysis}
\aftersubsec
Now, we are ready to state the best-iterate convergence rate of \eqref{eq:FBFS4NI} for any $u^k$ satisfying \eqref{eq:FBFS4NI_u_cond}.

\begin{theorem}[Best-iterate convergence rate]\label{th:FBFS4NI_convergence1}
For Inclusion~\eqref{eq:NI}, suppose that $\zer{\Phi} \neq \emptyset$, $F$  is $L$-Lipschitz continuous, and there exist $x^{\star} \in \zer{\Phi}$ and $\rho \geq 0$ such that $\iprods{u, x - x^{\star}} \geq - \rho \norms{u}^2$ for all $(x, u) \in\gra{\Phi}$.
Suppose further that $T$ is not necessarily monotone, but $\ran{J_{\eta T}} \subseteq \dom{F} = \R^p$ and $\dom{J_{\eta T}} = \R^p$.

Let $\Delta$ be defined by  \eqref{eq:EG4NE_kappa_i_quantities}, and $\underline{\eta}$ and $\bar{\eta}$ be defined by \eqref{eq:EG4NE_eta_bounds}.
Let $\sets{(x^k, y^k)}$ be generated by \eqref{eq:FBFS4NI} starting from $x^0 \in \dom{\Phi}$ and $y^{-1} = x^{-1} := x^0$ such that $u^k$ satisfies \eqref{eq:FBFS4NI_u_cond}.
If $L\rho \leq \Delta$ and $\eta$ is chosen such that $\eta \in [\underline{\eta}, \bar{\eta}]$, then $C_1$ and $C_2$ in \eqref{eq:EG4NE_C_constants}  are nonnegative.

Moreover, for any $K \geq 0$, we also have
\begin{equation}\label{eq:FBFS4NI_convergence_est1a}
\sum_{k=0}^K \big[  C_1 \norms{y^k - x^k}^2 + C_1 \norms{x^{k+1} - y^k}^2  + C_2 \norms{x^{k+1} - x^k}^2 \big] \leq \norms{x^0 - x^{\star}}^2.
\end{equation} 
If $C_1 > 0$, the following bounds also hold:
\begin{equation}\label{eq:FBFS4NI_convergence_est1b}
\left\{\begin{array}{ll}
{\displaystyle\min_{0\leq k \leq K}} \norms{Fy^k + \zeta^k}^2 \leq \frac{1}{K+1}\sum_{k=0}^K\norms{Fy^k + \zeta^k }^2 \leq \frac{\norms{x^0 - x^{\star}}^2 }{\Lambda  \eta^2 (K+1)}, \vspace{1ex}\\
{\displaystyle\min_{0\leq k \leq K}} \norms{Fx^k + \xi^k }^2 \leq \frac{1}{K+1}\sum_{k=0}^K\norms{Fx^k + \xi^k }^2 \leq \frac{ \Gamma \norms{x^0 - x^{\star}}^2 }{ K+1 }, 
\end{array}\right.
\end{equation} 
where $\zeta^k \in Ty^k$, $\xi^k \in Tx^k$, $\Lambda := \frac{C_1  + 2 C_2}{2}$, and $\Gamma :=   \frac{2L^2}{C_1} + \frac{2}{\Lambda\eta^2}$.
Thus we have
\begin{equation*}
\min_{0\leq k \leq K}\norms{Fx^k + \xi^k } = \BigO{\frac{1}{\sqrt{K}}} \quad  \text{and} \quad  \min_{0\leq k \leq K}\norms{Fy^k + \zeta^k } = \BigO{\frac{1}{\sqrt{K}}}.
\end{equation*}
The sequence $\sets{\norms{x^k - x^{\star}}}$ is bounded and $\sets{x^k}$ converges to $x^{\star}\in\zer{\Phi}$.
In addition, we have
\begin{equation*}
\lim_{k\to\infty}\norms{x^k - y^k} = \lim_{k\to\infty}\norms{Fx^k + \xi^k } = \lim_{k\to\infty}\norms{Fy^k + \zeta^k} =  0.
\end{equation*}
\end{theorem}

\begin{proof}
The proof of this theorem is very similar to the proof of Theorem~\ref{th:EG4NE_convergence1}, but using Lemma~\ref{le:FBFS4NI_key_estimate}, and we do not repeat it here.
\Eproof
\end{proof}

\beforesubsec
\subsection{\bf The Sublinear Convergence Rates of Two Instances}\label{subsec:FBFS4NI_special_cases}
\aftersubsec
Now, we specify Theorem~\ref{th:FBFS4NI_convergence1} for two instances: Tseng's method (\textbf{Variant 1}) and the forward-reflected-backward splitting method (\textbf{Variant 2}).
These results are very similar to the ones in Subsection~\ref{subsec:special_case_of_EG}, but we state them here for completeness without proof.

\begin{corollary}[FBFS method]\label{co:FBFS4NI_convergence_case1}
For Inclusion \eqref{eq:NI}, suppose that $\zer{\Phi} \neq \emptyset$, $F$  is $L$-Lipschitz continuous, and there exist $x^{\star} \in \zer{\Phi}$ and $\rho \geq 0$ such that $\iprods{w, x - x^{\star}} \geq - \rho \norms{w}^2$ for all $(x, w) \in\gra{\Phi}$.
Suppose further that $T$ is not necessarily monotone, but $\ran{J_{\eta T}} \subseteq \dom{F} = \R^p$ and $\dom{J_{\eta T}} = \R^p$.

Let $\sets{(x^k, y^k)}$ be generated by \eqref{eq:FBFS4NI} using  $u^k := Fx^k$.
Assume that $L\rho \leq \frac{\beta^2}{16}$ for any $0 < \beta \leq 1$, and $\eta$ is chosen as
\begin{equation}\label{eq:FBFS4NI_stepsize} 
\arraycolsep=0.2em
\begin{array}{lcl}
0 \leq \frac{\beta - \sqrt{\beta^2 - 16L\rho}}{2L} < \eta < \frac{\beta  + \sqrt{\beta^2 - 16L\rho} }{2L} \leq \frac{\beta}{L}.
\end{array}
\end{equation}
Then, for some $\xi^k \in Tx^k$  and $\Gamma$ given in Theorem~\ref{th:FBFS4NI_convergence1}, we have
\begin{equation}\label{eq:FBFS4NI_convergence_02}
\min_{0 \leq k \leq K}\norms{Fx^{k} + \xi^{k}}^2 \leq \frac{1}{K+1}\sum_{k=0}^{K}\norms{Fx^{k} + \xi^{k}}^2 \leq \frac{ \Gamma \norms{x^0 - x^{\star}}^2}{ K+1 }.
\end{equation} 
Moreover, $\sets{\norms{x^k - x^{\star}}}$ is nonincreasing and bounded, and $\sets{x^k}$ converges to $x^{\star} \in \zer{\Phi}$.
For $\xi^k \in Tx^k$ and $\zeta^k \in Ty^k$, we also have
\begin{equation}\label{eq:FBFS4NI_convergence_01}
\lim_{k\to\infty}\norms{x^k - y^k}  =  \lim_{k\to\infty}\norms{Fx^k + \xi^k} = \lim_{k\to\infty}\norms{Fy^k + \zeta^k} = 0.
\end{equation}
If $J_{\eta T}$ is nonexpansive and $\Gc_{\eta }x := \frac{1}{\eta}(x - J_{\eta T}(x - \eta Fx))$ is given by \eqref{eq:FB_residual}, then 
\begin{equation}\label{eq:FBFS4NI_convergence_02_03}
\min_{0 \leq k \leq K}\norms{\Gc_{\eta }x^k} = \BigO{\frac{1}{\sqrt{K}}} \quad \text{and} \quad \min_{0 \leq k \leq K}\norms{ \Gc_{\eta }y^k} = \BigO{\frac{1}{ \sqrt{K}}}.
\end{equation}
\end{corollary}

\begin{corollary}[Past-FBFS/OG method]\label{co:FBFS4NI_convergence_case2}
For Inclusion~\eqref{eq:NI}, suppose that $\zer{\Phi} \neq \emptyset$, $F$  is $L$-Lipschitz continuous, and  there exist $x^{\star} \in \zer{\Phi}$ and $\rho \geq 0$ such that $\iprods{u, x - x^{\star}} \geq - \rho \norms{u}^2$ for all $(x, u) \in\gra{\Phi}$.
Suppose further that $T$ is not necessarily monotone, but $\ran{J_{\eta T}} \subseteq \dom{F} = \R^p$ and $\dom{J_{\eta T}} = \R^p$.

Let $\sets{(x^k, y^k)}$ be generated by \eqref{eq:FBFS4NI} starting from $x^0 \in\dom{\Phi}$ and $y^{-1} := x^0$ using $u^k := Fy^{k-1}$.
Assume that $L\rho \leq  \frac{\beta^2}{8(3+\sqrt{5})}$ and $\eta$ is chosen as
\begin{equation}\label{eq:FBFS4NI_stepsize2}
\arraycolsep=0.2em
\begin{array}{l}
0 \leq \frac{\beta - \sqrt{\beta^2 - 8(3+\sqrt{5})L\rho}}{(3+\sqrt{5})L} < \eta < \frac{\beta + \sqrt{\beta^2 - 8(3+\sqrt{5})L\rho}}{(3+\sqrt{5})L} \leq \frac{2\beta}{(3+\sqrt{5})L}.
\end{array}
\end{equation}
Then, we have
\begin{equation}\label{eq:FBFS4NI_convergence_03}
\hspace{-2ex}
\arraycolsep=0.1em
\begin{array}{lcl}
{\displaystyle\min_{0\leq k \leq K}}\norms{Fx^{k} + \xi^{k} }^2 & \leq & {\displaystyle\frac{1}{K+1}\sum_{k=0}^{K}}  \norms{Fx^{k} + \xi^{k} }^2 \leq \dfrac{  \Gamma \norms{x^0 - x^{\star}}^2}{ K + 1 },
\end{array}
\hspace{-2ex}
\end{equation} 
where $\xi^k \in Tx^k$ and $\Gamma$ is given in Theorem~\ref{th:FBFS4NI_convergence1}.
Moreover, $\sets{\norms{x^k - x^{\star}}}$ is nonincreasing and bounded, and $\sets{x^k}$ converges to $x^{\star} \in \zer{\Phi}$.
If $J_{\eta T}$ is  nonexpansive, then \eqref{eq:FBFS4NI_convergence_02_03} remains valid.
\end{corollary}

\begin{remark}\label{re:FBFS4NI_last_iterate}
If $\rho = 0$, i.e. there exist $x^{\star}\in\zer{\Phi}$ such that $\iprods{w, x - x^{\star}} \geq 0$ for all $(x, w) \in \gra{\Phi}$, then the condition on $\eta$ from Theorem~\ref{th:FBFS4NI_convergence1} reduces to $0 < \eta < \frac{\beta^2}{(1+r)L}$ as we have  seen in Remark~\ref{re:EG4NE_rm1}.
For FBFS with $u^k := Fx^k$,  we have $0 < \eta < \frac{\beta^2}{L}$, while for  past-FBFS with $u^k := Fy^{k-1}$, we have $0 < \eta < \frac{2\beta^2}{(3+\sqrt{5})L}$.
Hitherto, we have only proven the best-iterate convergence rates of \eqref{eq:FBFS4NI}.
The last-iterate convergence rate of this method remains open.
\end{remark}

\beforesec
\section{Two Other Variants of The Extragradient Method}\label{sec:other_methods}
\aftersec
In this section, we first provide a new convergence analysis for the reflected forward-backward splitting  (RFBS) algorithm \cite{cevher2021reflected,malitsky2015projected}.
Our best-iterate convergence rate analysis is different from the one in \cite{cevher2021reflected} for monotone inclusions.
Our last-iterate convergence rate analysis is new, simple, and general than the one in  \cite{cai2022baccelerated}, which is only for \eqref{eq:VIP}.
Next, we extend the convergence analysis for the  golden ratio (GR) method \cite{malitsky2019golden} to a wider range of its parameters.
Our new range shows a potential improvement of GR on concrete examples in Section~\ref{sec:numerical_experiments}.
Nevertheless, its last-iterate convergence rate is still open.

\beforesubsec
\subsection{\mytb{New Convergence Analysis of The RFBS Method for \eqref{eq:NI}}}\label{subsec:RFBS4NI}
\aftersubsec
The RFBS method was proposed by Malitsky in  \cite{malitsky2015projected} to solve \eqref{eq:VIP} and it is called the \mytb{projected reflected gradient} method.
It was generalized to solve monotone \eqref{eq:NI} in \cite{cevher2021reflected}. 
The last iterate convergence rate of the \mytb{projected reflected gradient} method for \eqref{eq:VIP} was recently proven in \cite{cai2022baccelerated}.
Here, we provide a new and full convergence rate analysis for RFBS to solve \eqref{eq:NI} by exploiting our analysis techniques in the previous sections.

\beforesubsubsec
\subsubsection{The RFBS Method for \eqref{eq:NI}}\label{subsubsec:RFBFS_scheme}
\aftersubsubsec
The reflected forward-backward splitting (RFBS) method to approximate a solution of \eqref{eq:NI} is described as follows.
Starting from $x^0 \in \dom{\Phi}$, we set $x^{-1} := x^0$ and at each iteration $k\geq 0$, we update
\vspace{-0.5ex}
\begin{equation}\label{eq:RFBS4NI}
\arraycolsep=0.2em
\left\{\begin{array}{lcl}
y^k &:= & 2x^k - x^{k-1}, \vspace{1ex}\\
x^{k+1} &:= & J_{\eta T}(x^k - \eta Fy^k),
\end{array}\right.
\tag{RFBS2}
\vspace{-0.5ex}
\end{equation}
where $\eta > 0$ is a given step-size, determined later.

Clearly, if we eliminate $y^k$, then \eqref{eq:RFBS4NI} can be rewritten in one line:
\vspace{-0.5ex}
\begin{equation*}
x^{k+1} := J_{\eta T}(x^k - \eta F(2x^k - x^{k-1})).
\vspace{-0.5ex}
\end{equation*}
From the second line of \eqref{eq:RFBS4NI}, we have $\xi^{k+1} := \frac{1}{\eta}( x^k - \eta Fy^k - x^{k+1})  \in  Tx^{k+1}$.
As before, if we denote
\vspace{-0.5ex}
\begin{equation}\label{eq:wk_terms}
w^k := Fx^k + \xi^k \quad \text{and} \quad \hat{w}^k := Fy^{k-1} + \xi^k,
\vspace{-0.5ex}
\end{equation}
then we can rewrite \eqref{eq:RFBS4NI} equivalently to 
\begin{equation}\label{eq:RFBS4NI_reform}
\vspace{-0.5ex}
\arraycolsep=0.2em
\left\{\begin{array}{lclcl}
y^k &:= & x^k + x^k - x^{k-1} &= & x^k - \eta\hat{w}^k, \vspace{1ex}\\
x^{k+1} &:= & x^k - \eta\hat{w}^{k+1}.
\end{array}\right.
\vspace{-0.5ex}
\end{equation}
This expression leads to $x^{k+1}  =  y^k -\eta(\hat{w}^{k+1} - \hat{w}^k)$. 

\beforesubsubsec
\subsubsection{Key Estimates for Convergence Analysis}\label{subsubsec:RFBS4NI_key_estimate}
\aftersubsubsec
Next, we prove the following lemmas for our analysis.

\begin{lemma}\label{le:RFBS4NI_key_est1}
For Inclusion~\eqref{eq:NI}, assume that $\zer{\Phi} \neq \emptyset$, $T$  is maximally monotone and $F$ is  $L$-Lipschitz continuous and satisfies $\iprods{Fx - Fx^{\star}, x - x^{\star}} \geq 0$ for all $x\in\dom{F}$ and some $x^{\star}\in\zer{\Phi}$.
Let $\sets{(x^k, y^k)}$ be generated by \eqref{eq:RFBS4NI} using $\eta > 0$ and $\Vc_k$ be defined as
\begin{equation}\label{eq:RFBS4NI_potential_func}
\arraycolsep=0.2em
\begin{array}{lcl}
\Vc_k & := & \norms{x^k - x^{\star}}^2 + 2\norms{x^k - x^{k-1}}^2 +  \big( 1 - \sqrt{2}L \eta \big) \norms{x^k - y^{k-1}}^2 \vspace{1ex}\\
& & + {~} 2\eta \iprods{Fy^{k-1} - Fx^{\star}, x^k - x^{k-1}}.
\end{array}
\end{equation}
Then, we have
\vspace{-0.5ex}
\begin{equation}\label{eq:RFBS4NI_est1} 
\arraycolsep=0.2em
\begin{array}{lcl}
\Vc_k  & \geq & \Vc_{k+1} + \big[ 1 - (1+\sqrt{2})L\eta \big] \left[ \norms{y^k - x^k}^2 + \norms{x^k - y^{k-1}}^2 \right], \vspace{1.5ex}\\
\Vc_k  & \geq &   (1 - L\eta) \norms{x^k - x^{\star}}^2 + (1 - (1+\sqrt{2})L\eta) \norms{x^k - y^{k-1}}^2 \vspace{1ex}\\
&& + {~} 2(1 -  L\eta) \norms{x^k - x^{k-1}}^2.
\end{array}
\vspace{-0.5ex}
\end{equation}
\end{lemma}

\begin{proof}
First, since \eqref{eq:RFBS4NI} is equivalent to \eqref{eq:RFBS4NI_reform}, we have $x^{k+1} - x^k = -\eta\hat{w}^{k+1}$ from line 2 of \eqref{eq:RFBS4NI_reform}.
Using this expression, for any $x^{\star}\in\zer{\Phi}$, we get
\begin{equation}\label{eq:RFBS4NI_proof1}
\arraycolsep=0.2em
\begin{array}{lcl}
\norms{x^{k+1} - x^{\star}}^2 &= & \norms{x^k - x^{\star}}^2 - 2\eta\iprods{\hat{w}^{k+1}, x^{k+1} - x^{\star}} - \norms{x^{k+1} - x^k}^2.
\end{array} 
\end{equation}
Next, since  $Fx^{\star} + \xi^{\star} = 0$ from the fact that $x^{\star} \in \zer{\Phi}$ and $T$ is monotone, it is obvious to show that $\iprods{\xi^{k+1}, x^{k+1} - x^{\star}} \geq \iprods{\xi^{\star}, x^{k+1} - x^{\star}} = -\iprods{Fx^{\star}, x^{k+1} - x^{\star}}$, where $\xi^{\star}\in Tx^{\star}$.
Using this relation, we can prove that
\begin{equation}\label{eq:RFBS4NI_proof2a} 
\arraycolsep=0.2em
\begin{array}{lcl}
\iprods{\hat{w}^{k+1}, x^{k+1} - x^{\star}} &= &  \iprods{Fy^k, x^{k+1} - x^{\star}} + \iprods{\xi^{k+1}, x^{k+1} - x^{\star}} \vspace{1ex}\\
& \geq & \iprods{Fy^k - Fx^{\star}, y^k - x^{\star}} - \iprods{Fy^k - Fx^{\star},  y^k - x^{k+1} }.
\end{array} 
\end{equation}
Utilizing $y^k - x^k = x^k - x^{k-1}$ from the first line of \eqref{eq:RFBS4NI}, we can further expand the term $\Tc_{[1]} :=  \iprods{Fy^k - Fx^{\star}, y^k - x^{k+1}}$ as
\begin{equation}\label{eq:RFBS4NI_proof2}
\hspace{-0ex}
\arraycolsep=0.2em
\begin{array}{lcl}
\Tc_{[1]} &:= &  \iprods{Fy^k - Fx^{\star}, y^k - x^{k+1}} \vspace{1ex}\\
&= & \iprods{Fy^{k-1} - Fx^{\star}, y^k - x^k} - \iprods{Fy^k - Fx^{\star}, x^{k+1} - x^k}  \vspace{1ex}\\
 && + {~}  \iprods{Fy^k - Fy^{k-1}, y^k - x^k}   \vspace{1ex}\\
&= & \iprods{Fy^{k-1} - Fx^{\star}, x^k - x^{k-1}} - \iprods{Fy^k - Fx^{\star}, x^{k+1} - x^k}  \vspace{1ex}\\
&& + {~} \iprods{Fy^k - Fy^{k-1}, y^k - x^k}.
\end{array} 
\hspace{-6ex}
\end{equation}
Now, from the second line of \eqref{eq:RFBS4NI_reform}, we have $\eta \xi^{k+1} = x^k - x^{k+1} - \eta Fy^k$, leading to $\eta (\xi^{k+1} - \xi^k) = 2x^k - x^{k-1} - x^{k+1} - \eta(Fy^k - Fy^{k-1}) = y^k - x^{k+1} - \eta(Fy^k - Fy^{k-1})$.
By the monotonicity of $T$, we have $ \iprods{y^k - x^{k+1} - \eta(Fy^k - Fy^{k-1}), x^{k+1} - x^k} = \eta\iprods{\xi^{k+1} - \xi^k, x^{k+1} - x^k} \geq 0$, leading to
\begin{equation*} 
\arraycolsep=0.2em
\begin{array}{lcl}
2\eta\iprods{Fy^k - Fy^{k-1}, x^{k+1} - x^k} &\leq & 2\iprods{y^k - x^{k+1}, x^{k+1} - x^k} \vspace{1ex}\\
&= & \norms{y^k - x^k}^2 - \norms{x^{k+1} - x^k}^2 - \norms{x^{k+1} - y^k}^2.
\end{array} 
\end{equation*}
By the Cauchy-Schwarz inequality, the Lipschitz continuity of $F$, and Young's inequality, we can show that
\begin{equation*} 
\arraycolsep=0.2em
\begin{array}{lcl}
2\eta\iprods{Fy^k - Fy^{k-1},y^k - x^{k+1}} 
& \leq & \eta L(\sqrt{2} + 1)\norms{y^k - x^k}^2 + \eta L\norms{x^k - y^{k-1}}^2 \vspace{1ex}\\
&& + {~}  \sqrt{2}\eta L \norms{x^{k+1} - y^k}^2.
\end{array} 
\end{equation*}
Summing up the last two inequalities, we get
\begin{equation*} 
\arraycolsep=0.2em
\begin{array}{lcl}
2\eta \iprods{Fy^k - Fy^{k-1}, y^k - x^k} &\leq & \left[ 1 +  \eta L(\sqrt{2} + 1)\right] \norms{y^k - x^k}^2 + \eta L\norms{x^k - y^{k-1}}^2 \vspace{1ex}\\
&& - {~} (1 - \sqrt{2} L \eta ) \norms{x^{k+1} - y^k}^2 - \norms{x^{k+1} - x^k}^2.
\end{array} 
\end{equation*}
Combining this inequality, \eqref{eq:RFBS4NI_proof2a}, and \eqref{eq:RFBS4NI_proof2} with \eqref{eq:RFBS4NI_proof1}, we can  prove that
\begin{equation*} 
\arraycolsep=0.2em
\begin{array}{lcl}
\norms{x^{k+1} - x^{\star}}^2 
&\leq & \norms{x^k - x^{\star}}^2 +  2\eta \iprods{Fy^{k-1} - Fx^{\star}, x^k - x^{k-1}}  \vspace{1ex}\\
&& + {~} \left[ 1 +   L\eta(\sqrt{2} + 1)\right] \norms{y^k - x^k}^2 - 2\norms{x^{k+1} - x^k}^2 \vspace{1ex}\\
&& +  {~} L\eta \norms{x^k - y^{k-1}}^2 -  (1 - \sqrt{2}L\eta ) \norms{x^{k+1} - y^k}^2 \vspace{1ex}\\
&&  - {~} 2\eta \iprods{Fy^k - Fx^{\star}, x^{k+1} - x^k} - 2\eta \iprods{Fy^k - Fx^{\star}, y^k - x^{\star}}.
\end{array} 
\end{equation*}
Using the definition \eqref{eq:RFBS4NI_potential_func} of $\Vc_k$, $\iprods{Fy^k - Fx^{\star}, y^k - x^{\star}} \geq 0$, and $y^k - x^k = x^k - x^{k-1}$, this estimate becomes
\begin{equation*} 
\arraycolsep=0.2em
\begin{array}{lcl}
\Vc_{k+1} &:= & \norms{x^{k+1} - x^{\star}}^2 + 2\norms{x^{k+1} - x^k}^2 +  (1 - \sqrt{2}L\eta ) \norms{x^{k+1} - y^k}^2 \vspace{1ex}\\
&& + {~} 2\eta \iprods{Fy^k - Fx^{\star}, x^{k+1} - x^k} \vspace{1ex}\\
&\leq & \norms{x^k - x^{\star}}^2 +  \left[ 1 +  (\sqrt{2} + 1)L\eta \right] \norms{y^k - x^k}^2 +  L\eta \norms{x^k - y^{k-1}}^2 \vspace{1ex}\\
&& + {~}  2\eta \iprods{Fy^{k-1} - Fx^{\star}, x^k - x^{k-1}} \vspace{1ex}\\
&= & \Vc_k - \left[1 - (\sqrt{2} + 1)L\eta \right] \left(  \norms{y^k - x^k}^2  +  \norms{x^k - y^{k-1}}^2 \right),
\end{array} 
\end{equation*}
which proves the first inequality of \eqref{eq:RFBS4NI_est1}.

\noindent Next, using Young's inequality twice, we can show that
\begin{equation*}
\arraycolsep=0.2em
\begin{array}{lcl}
2\eta \iprods{Fy^{k-1} - Fx^{\star}, x^k - x^{k-1}}  
& \geq & - \frac{L\eta}{2} \norms{y^{k-1} - x^{\star}}^2 - 2 L\eta \norms{x^k - x^{k-1}}^2 \vspace{1ex}\\
& \geq & - L\eta \norms{x^k - x^{\star}}^2 - L\eta \norms{x^k - y^{k-1}}^2 \vspace{1ex}\\
&& - {~} 2 L\eta \norms{x^k - x^{k-1}}^2.
\end{array} 
\end{equation*}
Substituting this estimate into \eqref{eq:RFBS4NI_potential_func}, we get
\begin{equation*}
\arraycolsep=0.2em
\begin{array}{lcl}
\Vc_k & := &  \norms{x^k - x^{\star}}^2 + 2\norms{x^k - x^{k-1}}^2 +  \left(1 - \sqrt{2}L \eta\right) \norms{x^k - y^{k-1}}^2 \vspace{1ex}\\
&& + {~} 2\eta \iprods{Fy^{k-1} - Fx^{\star}, x^k - x^{k-1}} \vspace{1ex}\\
&\geq & (1 - L\eta) \norms{x^k - x^{\star}}^2 + (1 - (1+\sqrt{2})L\eta) \norms{x^k - y^{k-1}}^2 \vspace{1ex}\\
&& + {~} 2(1 -  L\eta) \norms{x^k - x^{k-1}}^2,
\end{array} 
\end{equation*}
which proves the second line of \eqref{eq:RFBS4NI_est1}.
\Eproof
\end{proof}

\begin{lemma}\label{le:RFBS4NI_key_est2}
For Inclusion \eqref{eq:NI}, suppose  that $\Phi$  is maximally monotone and $F$ is $L$-Lipschitz continuous.
Let $\sets{(x^k, y^k)}$ be generated by \eqref{eq:RFBS4NI}.
Then
\begin{equation}\label{eq:RFBS4NI_est2}
\hspace{-0ex}
\arraycolsep=0.2em
\begin{array}{lcl}
& \norms{ Fx^k + \xi^k}^2 \leq  \frac{5L^2\eta^2 + 3}{3\eta^2}  \norms{x^k - y^k}^2 + \frac{5L^2\eta^2 + 3}{5\eta^2} \norms{x^k - y^{k-1}}^2.
\end{array}
\end{equation}
Moreover, if $\sqrt{2}L\eta < 1$, then  with $\omega := \frac{2L^2\eta^2}{1 - 2L^2\eta^2} > 0$, we also have
\begin{equation}\label{eq:RFBS4NI_est3}
\hspace{-2ex}
\arraycolsep=0.1em
\begin{array}{lcl}
\norms{Fx^{k+1} + \xi^{k+1}}^2 & + &  \omega \norms{Fx^{k+1} - Fy^k}^2 \leq   \norms{Fx^k + \xi^k}^2 + \omega \norms{Fx^k - Fy^{k-1}}^2 \vspace{1ex}\\
&& - {~} \left( \frac{1 - 4L^2\eta^2 }{1 - 2L^2\eta^2} \right) \norms{Fy^k - Fx^k + \xi^{k+1} - \xi^k}^2.
\end{array} 
\hspace{-4ex}
\end{equation}
\end{lemma}

\begin{proof}
First, by Young's inequality, $x^k - y^k = \eta(Fy^{k-1} + \xi^k)$ from \eqref{eq:RFBS4NI_reform}, and the $L$-Lipschitz continuity of $F$, we have
\begin{equation*}
\arraycolsep=0.2em
\begin{array}{lcl}
\norms{ Fx^k + \xi^k }^2 & \leq &   \big(1 + \frac{5L^2\eta^2}{3}\big) \norms{ Fy^{k-1} + \xi^k }^2  + \left(1 + \frac{3}{5L^2\eta^2}\right)\norms{ Fx^k - Fy^{k-1} }^2  \vspace{1ex}\\
& \leq & \frac{5L^2\eta^2 + 3}{3\eta^2}\norms{x^k - y^k}^2 + \frac{5L^2\eta^2 + 3}{5\eta^2} \norms{x^k - y^{k-1}}^2.
\end{array} 
\end{equation*}
This estimate is exactly \eqref{eq:RFBS4NI_est2}.

Next, by the monotonicity of $\Phi$, we have $\iprods{w^{k+1} - w^k, x^{k+1} - x^k}  \geq 0$.
Substituting $x^{k+1} - x^k = -\eta\hat{w}^{k+1}$ into this inequality, we get
\begin{equation*}
\arraycolsep=0.2em
\begin{array}{lcl}
0 &\leq &  2\iprods{w^k, \hat{w}^{k+1}} -2\iprods{w^{k+1}, \hat{w}^{k+1}}  \vspace{1ex}\\
& = &  \norms{w^k}^2 - \norms{w^{k+1}}  + \norms{w^{k+1} - \hat{w}^{k+1}}^2  - \norms{\hat{w}^{k+1} - w^k}^2.
\end{array} 
\end{equation*}
This inequality can be simplified as 
\begin{equation*}
\arraycolsep=0.2em
\begin{array}{lcl}
\norms{w^{k+1}}^2 &\leq & \norms{w^k}^2 + \norms{w^{k+1} - \hat{w}^{k+1}}^2  - \norms{\hat{w}^{k+1} - w^k}^2.
\end{array} 
\end{equation*}
By the Lipschitz continuity of $F$ and $x^{k+1} - y^k = -\eta(\hat{w}^{k+1} - \hat{w}^k)$, we have 
\begin{equation*}
\arraycolsep=0.2em
\begin{array}{lcl}
\norms{w^{k+1} - \hat{w}^{k+1}}^2 & = &  \norms{Fx^{k+1} - Fy^k}^2 \leq L^2\norms{x^{k+1} - y^k}^2 \vspace{1ex}\\
& = &   L^2\eta^2\norms{\hat{w}^{k+1} - \hat{w}^k}^2 \vspace{1ex}\\
& \leq &  2L^2\eta^2\norms{\hat{w}^{k+1} - w^k}^2 + 2L^2\eta^2\norms{w^k - \hat{w}^k}^2.
\end{array}
\end{equation*}  
Multiplying this expression by $\omega + 1$ for $\omega \geq 0$, and adding the result to the last inequality, we get
\begin{equation*}
\arraycolsep=0.2em
\begin{array}{lcl}
\norms{w^{k+1}}^2 + \omega \norms{w^{k+1} - \hat{w}^{k+1}}^2 &\leq & \norms{w^k}^2  + 2(\omega + 1) L^2\eta^2\norms{w^k - \hat{w}^k}^2 \vspace{1ex}\\
&&  - {~} (1 -  2(\omega + 1) L^2\eta^2) \norms{\hat{w}^{k+1} - w^k}^2.
\end{array} 
\end{equation*}
Finally, let us choose $\omega \geq 0$ such that $\omega  = 2(\omega + 1) L^2\eta^2$. 
If $2L^2\eta^2 < 1$, then $\omega := \frac{2L^2\eta^2}{ 1 - 2L^2\eta^2}$ satisfies $\omega  = 2(\omega + 1) L^2\eta^2$. 
Consequently, the last estimate leads to \eqref{eq:RFBS4NI_est3}.
\Eproof
\end{proof}

\beforesubsubsec
\subsubsection{The Sublinear Convergence Rate Analysis of \eqref{eq:RFBS4NI} }\label{subsubsec:RFBS4NI_convergence}
\aftersubsubsec
Now, we are ready to establish the convergence rates of \eqref{eq:RFBS4NI}.

\begin{theorem}\label{th:RFBS_convergence}
For Inclusion \eqref{eq:NI}, suppose that $\zer{\Phi} \neq\emptyset$,  $F$  is  $L$-Lipschitz continuous, there exists $x^{\star} \in \zer{\Phi}$ such that $\iprods{Fx - Fx^{\star}, x - x^{\star}} \geq 0$ for all $x\in\dom{F}$, and $T$ is maximally monotone.
Let $\sets{(x^k, y^k)}$ be generated by \eqref{eq:RFBS4NI} using $\eta \in \big(0, \frac{\sqrt{2}-1}{L} \big)$.
Then, we have the following statements.
\begin{itemize}
\item[$\mathrm{(a)}$] \textbf{The $\BigOs{1/\sqrt{k}}$ best-iterate rate.} The following bound holds:
\begin{equation}\label{eq:RFBS4NI_convergence1}
\hspace{-5ex}
\arraycolsep=0.2em
\begin{array}{lcl}
{\displaystyle\frac{1}{K+1}\sum_{k=0}^K } \norms{Fx^k + \xi^k }^2  & \leq & {\displaystyle \frac{1}{K+1}\sum_{k=0}^K } \left[ \norms{Fx^k + \xi^k }^2 + \omega \norms{Fx^k - Fy^{k-1}}^2\right]  \vspace{1ex}\\
& \leq & \dfrac{C_0 \norms{x^0 - x^{\star}}^2}{K+1},
\end{array}
\hspace{-5ex}
\end{equation}
where $\omega :=  \frac{2L^2\eta^2}{1 - L^2\eta^2} > 0$ and $C_0 := \frac{5L^2\eta^2 + 3}{3\eta^2\left[ 1 - (1 + \sqrt{2})L\eta \right]} > 0$.

\item[$\mathrm{(b)}$] \textbf{The $\BigOs{1/\sqrt{k}}$ last-iterate rate.}
If $\Phi$ is additionally monotone, then 
\begin{equation}\label{eq:RFBS4NI_convergence2}
\norms{ Fx^K + \xi^K }^2 \leq \norms{ Fx^K + \xi^K }^2 + \omega \norms{Fx^K - Fy^{K-1}}^2   \leq  \frac{C_0\norms{x^0 - x^{\star}}^2}{K+1},
\end{equation}
where $\omega := \frac{2L^2\eta^2}{1 - 2L^2\eta^2} > 0$.
Hence, we conclude that $\norms{\Gc_{\eta}x^K} \leq \norms{Fx^K + \xi^K} = \BigO{\frac{1}{\sqrt{K}}}$ on the last-iterate $x^K$.
\end{itemize}
\end{theorem}

\begin{proof}
First, since $0 < \eta < \frac{\sqrt{2} - 1}{L}$, we have $1 - (1 + \sqrt{2})L\eta > 0$ and $\omega := \frac{2L^2\eta^2}{1 - 2L^2\eta^2} < \frac{2}{3}$.
From  \eqref{eq:RFBS4NI_est2}, we have
\begin{equation*}
\arraycolsep=0.2em
\begin{array}{lcl}
\norms{ Fx^k + \xi^k}^2 & + &   \omega \norms{Fx^k - Fy^{k-1}}^2    \leq   \frac{5L^2\eta^2 + 3}{3\eta^2}\norms{x^k - y^k}^2 \vspace{1ex}\\
&& + {~}  \left( \frac{5L^2\eta^2 + 3}{5\eta^2} + \frac{2L^2}{3}\right)\norms{x^k - y^{k-1}}^2 \vspace{1ex}\\
& \leq & \frac{5L^2\eta^2 + 3}{3\eta^2}\left[ \norms{x^k - y^k}^2 + \norms{x^k - y^{k-1}}^2 \right].
\end{array} 
\end{equation*}
Combining this estimate and \eqref{eq:RFBS4NI_est1}, we get
\begin{equation*}
\arraycolsep=0.2em
\begin{array}{lcl}
\Tc_{[1]} & := & \frac{3\eta^2\left[ 1 - (1 + \sqrt{2})L\eta \right]}{5L^2\eta^2 + 3} \left[ \norms{ Fx^k + \xi^k}^2 +  \omega \norms{Fx^k - Fy^{k-1}}^2  \right]  \vspace{1ex}\\
&\leq &  (1 - (1 + \sqrt{2})L\eta)\left[\norms{x^k - y^k}^2 + \norms{x^k - y^{k-1}}^2  \right] \vspace{1ex}\\
&  \leq & \Vc_k - \Vc_{k+1}.
\end{array} 
\end{equation*}
Summing this inequality from $k=0$ to $k=K$, and using $x^{-1} = y^{-1} := x^0$, we get
\begin{equation*}
\arraycolsep=0.2em
\begin{array}{lcl}
\frac{3\eta^2\left[ 1 - (1 + \sqrt{2})L\eta \right]}{5L^2\eta^2 + 3} \sum_{k=0}^K \left[ \norms{ Fx^k + \xi^k}^2 +  \omega \norms{Fx^k - Fy^{k-1}}^2  \right]  
&  \leq & \Vc_0 - \Vc_{K+1} \vspace{1ex}\\
& \leq &  \Vc_0 = \norms{x^0 - x^{\star}}^2,
\end{array} 
\end{equation*}
which leads to \eqref{eq:RFBS4NI_convergence1}.
Finally, combining \eqref{eq:RFBS4NI_convergence1} and \eqref{eq:RFBS4NI_est3}, we obtain \eqref{eq:RFBS4NI_convergence2}.
\Eproof
\end{proof}

\begin{remark}\label{re:RFBS4NI_remark1}
We can modify \eqref{eq:RFBS4NI} to capture adaptive parameters as $y^k := x^k + \beta_k(x^k - x^{k-1})$ and $x^{k+1} := J_{\eta_kT}(x^k - \eta_kFy^k)$, where $\eta_k := \beta_k\eta_{k-1}$ for some $\beta_k > 0$.
Then, by imposing appropriate bounds on $\eta_k$, we can still prove the convergence of this variant by modifying the proof of Theorem~\ref{th:RFBS_convergence}.
The best-iterate and last-iterate convergence rates of  \eqref{eq:RFBS4NI} under weaker conditions (e.g., a weak-Minty solution and the co-hypomonotonicity) is still  open.
\end{remark}

\beforesubsec
\subsection{\mytb{New Golden Ratio Method for \eqref{eq:NI}}}\label{subsec:GR4NI}
\aftersubsec
The golden ratio (GR) method  was proposed by Malitsky in \cite{malitsky2019golden} to solve monotone \eqref{eq:MVIP}, where the ratio parameter $\tau := \frac{\sqrt{5} + 1}{2}$, leading to a so-called \textit{golden ratio}.
Here, we extend it to solve \eqref{eq:NI} for the case $F$ is monotone and $L$-Lipschitz continuous, and $T$ is maximally $3$-cyclically monotone.
Moreover, we extend our analysis for any $\tau \in (1,  1 + \sqrt{3})$.

\beforesubsubsec
\subsubsection{The GR Method for \eqref{eq:NI}}\label{subsubsec:GR4NI_scheme}
\aftersubsubsec
The golden ratio (GR) method for solving \eqref{eq:NI} is presented as follows.
Starting from $x^0 \in \dom{\Phi}$, we set $y^{-1} := x^0$, and at each iteration $k \geq 0$, we update
\begin{equation}\label{eq:GR4NI}
\left\{\begin{array}{lcl}
y^k &:= & \frac{\tau -1}{\tau}x^k + \frac{1}{\tau}y^{k-1}, \vspace{1ex}\\
x^{k+1} &:= & J_{\eta T}(y^k - \eta Fx^k),
\end{array}\right.
\tag{GR2+}
\end{equation}
where $J_{\eta T}$ is the resolvent of $\eta T$, $\tau > 1$ is given, and $\eta \in (0, \frac{\tau}{2L})$.

Let us denote $\breve{w}^k := Fx^{k-1} + \xi^k$ for $\xi^k \in Tx^k$.
Then, we can rewrite the second line of \eqref{eq:GR4NI} as $x^{k+1} :=  y^k - \eta(Fx^k + \xi^{k+1}) = y^k - \eta \breve{w}^{k+1}$ for $\xi^{k+1} \in Tx^{k+1}$.
In this case, we have $x^k = y^{k-1} - \eta(Fx^{k-1} + \xi^k)$, leading to $y^{k-1} = x^k + \eta \breve{w}^k$.
Combining this expression and the first line of \eqref{eq:GR4NI}, we have $y^k = \frac{\tau - 1}{\tau}x^k + \frac{1}{\tau}(x^k + \eta\breve{w}^k) = x^k + \frac{\eta}{\tau}\breve{w}^k$.
Consequently, we can rewrite \eqref{eq:GR4NI} equivalently to the following one:
\begin{equation}\label{eq:GR4NI_reform}
\arraycolsep=0.2em
\left\{\begin{array}{lcllcl}
y^k &:= & x^k + \frac{\eta}{\tau}\breve{w}^k, \vspace{1ex}\\
x^{k+1} &:= & y^k - \eta(Fx^k + \xi^{k+1}) &= & y^k - \eta\breve{w}^{k+1}.
\end{array}\right.
\end{equation}
If we eliminate $y^k$, then we obtain
\begin{equation}\label{eq:GR4NI_reform2}
\arraycolsep=0.1em
\left\{\begin{array}{lcl}
x^{k+1} &:= & J_{\eta T}\left(x^k - \eta \left( Fx^k - \frac{1}{\tau}(Fx^{k-1} + \xi^k)  \right) \right), \vspace{1ex}\\  
\xi^{k+1} &:= & \frac{1}{\eta}(x^k - x^{k+1}) -  \left( Fx^k - \frac{1}{\tau}(Fx^{k-1} + \xi^k)  \right).
\end{array}\right.
\end{equation}
This scheme is another interpretation of  \eqref{eq:GR4NI}.

\beforesubsubsec
\subsubsection{Key Estimates for Convergence Analysis}\label{subsubsec:GR4NI_key_estimates}
\aftersubsubsec
The convergence of \eqref{eq:GR4NI} is established based on the following key lemma.

\begin{lemma}\label{le:GR4NI_key_est1}
For Inclusion~\eqref{eq:NI}, suppose that $\zer{\Phi} \neq\emptyset$, $T$ is maximally $3$-cyclically monotone, and $F$ is $L$-Lipschitz continuous.
Let $\sets{(x^k, y^k)}$ be generated by \eqref{eq:GR4NI} with $\tau > 1$.
Then, for any $x^{\star}\in\zer{\Phi}$, we have
\begin{equation}\label{eq:GR4NI_key_est1}
\hspace{-3ex}
\arraycolsep=0.1em
\begin{array}{lcl}
\tau \norms{y^{k+1} - x^{\star}}^2  & + &   (\tau - 1)(\tau - \gamma) \norms{x^{k+1} - x^k}^2  \leq  \tau \norms{y^k - x^{\star}}^2  \vspace{1ex}\\
&& + {~}  \frac{(\tau-1)L^2\eta^2}{\gamma}\norms{x^k - x^{k-1}}^2  -  \frac{(\tau - 1)(1-\tau^2+\tau)}{\tau}\norms{ x^{k+1} - y^k}^2 \vspace{1ex}\\
& & - {~}  \tau(\tau - 1)\norms{x^k - y^k}^2 -  2\eta(\tau - 1)\iprods{Fx^k - Fx^{\star}, x^k - x^{\star}}.
\end{array}
\hspace{-3ex}
\end{equation}
\end{lemma}

\begin{proof}
Since $T$ is $3$-cyclically monotone, for $\xi^{k+1} \in Tx^{k+1}$, $\xi^k \in Tx^k$, and $x^{\star} \in Tx^{\star}$, we have 
\begin{equation}\label{eq:GR4NI_lm1_proof1}
\arraycolsep=0.2em
\begin{array}{lcll}
\iprods{\xi^{k+1}, x^{k+1} - x^{\star}} + \iprods{\xi^{\star}, x^{\star} - x^k} + \iprods{\xi^k, x^k - x^{k+1}} \geq 0.
\end{array}
\end{equation}
From \eqref{eq:GR4NI_reform}, we get $\eta\xi^{k+1} = y^k - x^{k+1} - \eta Fx^k$ and $\eta \xi^k = y^{k-1} - x^k - \eta Fx^{k-1}$. 
Moreover, since $x^{\star} \in \zer{\Phi}$, we obtain $Fx^{\star} + \xi^{\star} = 0$, leading to $\eta\xi^{\star} = - \eta Fx^{\star}$.
Substituting these expressions into \eqref{eq:GR4NI_lm1_proof1}, we can show that
\begin{equation*}
\begin{array}{lcll}
\iprods{y^k - x^{k+1} - \eta Fx^k, x^{k+1} - x^{\star}} & + &  \iprods{y^{k-1} - x^k - \eta Fx^{k-1}, x^k - x^{k+1}} \vspace{1ex}\\
&&  - {~} \iprods{Fx^{\star}, x^{\star} - x^k}  \geq 0.
\end{array}
\end{equation*}
However, since $y^k :=  \frac{\tau -1}{\tau}x^k + \frac{1}{\tau}y^{k-1}$ from the first line of \eqref{eq:GR4NI},  we get $y^{k-1} - x^k =  \tau(y^k - x^k)$.
Substituting this relation into the last inequality, rearranging the result, we obtain
\begin{equation}\label{eq:GR4NI_lm1_proof2}
\begin{array}{ll}
& \iprods{y^k - x^{k+1}, x^{k+1} - x^{\star}}  +    \tau \iprods{x^k  - y^k,  x^{k+1} - x^k}  \vspace{1ex}\\
& + {~}  \eta\iprods{Fx^{k-1} - Fx^k, x^{k+1} - x^k} -  \eta\iprods{Fx^k - Fx^{\star}, x^k - x^{\star}} \geq 0.
\end{array}
\end{equation}
By Young's inequality and the $L$-Lipschitz continuity of $F$, for any $\gamma > 0$, we get the first line of the following:
\begin{equation*}
\arraycolsep=0.2em
\begin{array}{lcl}
2\eta\iprods{Fx^{k-1} - Fx^k, x^{k+1} - x^k} &\leq & \frac{L^2\eta^2}{\gamma}\norms{x^k - x^{k-1}}^2 + \gamma \norms{x^{k+1} - x^k}^2, \vspace{1ex}\\
2\iprods{y^k - x^{k+1}, x^{k+1} - x^{\star} } &= & \norms{y^k - x^{\star}}^2 - \norms{x^{k+1} - x^{\star}}^2 - \norms{x^{k+1} - y^k}^2, \vspace{1ex}\\
2\iprods{x^k - y^k, x^{k+1} - x^k} &= & \norms{x^{k+1} - y^k}^2 -\norms{x^k - y^k}^2 - \norms{x^{k+1} - x^k}^2.
\end{array}
\end{equation*}
Substituting these expressions into  \eqref{eq:GR4NI_lm1_proof2}, and rearranging the results, we obtain
\begin{equation}\label{eq:GR4NI_lm1_proof3}
\arraycolsep=0.2em
\begin{array}{lcl}
\norms{x^{k+1} - x^{\star}}^2 &\leq & \norms{y^k - x^{\star}}^2 + (\tau - 1) \norms{x^{k+1} - y^k}^2 -   \tau \norms{x^k - y^k}^2 \vspace{1ex}\\
&& +  {~}   \frac{L^2\eta^2}{\gamma}\norms{x^k - x^{k-1}}^2   - ( \tau  - \gamma )\norms{x^{k+1} - x^k}^2 \vspace{1ex}\\
&& - {~} 2\eta \iprods{Fx^k - Fx^{\star}, x^k - x^{\star}}.
\end{array}
\end{equation}
Now, using $(\tau - 1)x^{k+1} = \tau y^{k+1} - y^k$ and $\tau(y^{k+1} - y^k) = (\tau - 1)(x^{k+1} - y^k)$ from the first line of \eqref{eq:GR4NI}, we can derive that 
\begin{equation*} 
\arraycolsep=0.2em
\begin{array}{lcl}
(\tau - 1)^2\norms{x^{k+1} - x^{\star}}^2  &= & \tau(\tau - 1)\norms{y^{k+1} - x^{\star}}^2 - (\tau - 1) \norms{y^k - x^{\star}}^2 \vspace{1ex}\\
&& + {~} \tau \norms{y^{k+1} -  y^k}^2 \vspace{1ex}\\
&= & \tau(\tau - 1)\norms{y^{k+1} - x^{\star}}^2 - (\tau - 1) \norms{y^k - x^{\star}}^2 \vspace{1ex}\\
&& + {~} \frac{(\tau - 1)^2}{ \tau }\norms{x^{k+1} -  y^k}^2.
\end{array}
\end{equation*}
Simplifying this expression to get $(\tau - 1)\norms{x^{k+1} - x^{\star}}^2 = \tau \norms{y^{k+1} - x^{\star}}^2 - \norms{y^k - x^{\star}}^2 + \frac{(\tau - 1)}{\tau }\norms{x^{k+1} -  y^k}^2$.
Combining it and \eqref{eq:GR4NI_lm1_proof3}, and rearranging the result, we obtain \eqref{eq:GR4NI_key_est1}.
\Eproof
\end{proof}

\beforesubsubsec
\subsubsection{The Sublinear Best-Iterate Convergence Rate Analysis}\label{subsubsec:GR4NI_convergence}
\aftersubsubsec
Now, we are ready to state the convergence of \eqref{eq:GR4NI}.

\begin{theorem}\label{th:GR4NI_convergence}
For Inclusion~\eqref{eq:NI}, suppose that $\zer{\Phi} \neq\emptyset$, $T$  is maximally $3$-cyclically monotone, and $F$ is  $L$-Lipschitz continuous and satisfies $\iprods{Fx - Fx^{\star}, x - x^{\star}} \geq 0$ for all $x\in\dom{F}$ and some $x^{\star} \in \zer{\Phi}$.
Let $\sets{(x^k, y^k)}$ be generated by \eqref{eq:GR4NI}. 
Then, the following statements hold.
\begin{itemize}
\item[$\mathrm{(a)}$] \textbf{The best-iterate rate of GR.}  If $1 < \tau \leq \frac{1+\sqrt{5}}{2}$ and  $\eta \in \left(0, \frac{\tau}{2L}\right)$, then
\begin{equation}\label{eq:GR4NI_convergence1}
\hspace{-5ex}
\arraycolsep=0.1em
\begin{array}{lcl}
{\displaystyle\frac{1}{K+1}\sum_{k=0}^{K} } \norms{Fx^k + \xi^k}^2  & \leq & {\displaystyle \frac{C_0}{K+1}\sum_{k=0}^K } (\tau - 1)\left[ \tau \norms{x^k - y^k}^2 + \varphi  \norms{x^{k} -  x^{k-1}}^2\right]  \vspace{1ex}\\
& \leq & \dfrac{C_0 \norms{x^0 - x^{\star}}^2}{K+1},
\end{array}
\hspace{-5ex}
\end{equation}
where $\xi^k \in Tx^k$,  $\varphi :=  \frac{\tau^2 - 4L^2\eta^2}{2\tau} > 0$, and $C_0 :=  \frac{ (\tau^2 - 2L^2\eta^2)\tau }{(\tau^2 - 4L^2\eta^2)\eta^2(\tau - 1)} > 0$.

\item[$\mathrm{(b)}$] \textbf{The best-iterate rate of \eqref{eq:GR4NI}.} If we choose $ \frac{1+\sqrt{5}}{2} < \tau < 1 + \sqrt{3}$ and $0 < \eta < \frac{\psi}{2L}$, then 
\begin{equation}\label{eq:GR4NI_convergence2}
\hspace{-5ex}
\arraycolsep=0.1em
\begin{array}{lcl}
{\displaystyle\frac{1}{K+1}\sum_{k=0}^{K} } \norms{Fx^k + \xi^k }^2  & \leq &  {\displaystyle \frac{1}{K+1}\sum_{k = 0}^K } (\tau - 1)\left[ \psi  \norms{x^k - y^k }^2 + \kappa \norms{x^{k} -  y^{k-1}}^2\right] \vspace{1ex}\\
& \leq & \dfrac{\hat{C}_0 \norms{x^0 - x^{\star}}^2}{K+1},
\end{array}
\hspace{-5ex}
\end{equation}
where $\psi :=  \frac{2\tau + 2 - \tau^2}{\tau} > 0$, $\kappa := \frac{\psi^2 - 4L^2\eta^2}{2\psi}$, and $\hat{C}_0 :=  \frac{[\psi^2 - 2L^2\eta^2(2\tau^2 - \psi^2)]\tau}{(\tau - 1)(\psi^2 - 4L^2\eta^2)\eta^2\psi}$.

\end{itemize}
\end{theorem}

\begin{proof}
First, to guarantee that $1 + \tau - \tau^2 \geq 0$ and $\tau > 1$, we need to choose $1 < \tau \leq \frac{\sqrt{5} + 1}{2}$.
If $0 < \eta < \frac{\tau}{2L}$, then by choosing $\gamma := \frac{\tau}{2}$, we have  $\psi := \frac{(\tau - 1)(\tau \gamma - \gamma^2 - L^2\eta^2)}{\gamma} = \frac{(\tau - 1)(\tau^2 - 4L^2\eta^2)}{2\tau } > 0$.
Using this relation and $\iprods{Fx^k - Fx^{\star}, x^k - x^{\star}} \geq 0$, if we define $\Vc_k := \tau \norms{y^k - x^{\star}}^2 + \frac{\tau (\tau - 1)}{2} \norms{x^k - x^{k-1}}^2 \geq 0$, then we can deduce from \eqref{eq:GR4NI_key_est1} that
\begin{equation}\label{eq:GR4NI_th1_proof1}
\arraycolsep=0.2em
\begin{array}{lcl}
\Vc_{k+1} &\leq & \Vc_k  -  \psi \cdot \norms{ x^k - x^{k-1}}^2  - \tau(\tau - 1)\norms{x^k - y^k}^2.
\end{array}
\end{equation}
Next, using $y^k - x^k = \frac{\eta}{ \tau }\breve{w}^k$ and $\breve{w}^k = Fx^{k-1} + \xi^k$, by Young's inequality, we can prove that
\begin{equation}\label{eq:GR4NI_th1_proof2}
\hspace{-3ex}
\arraycolsep=0.1em
\begin{array}{lcl}
\norms{w^k}^2 & = & \norms{Fx^k + \xi^k}^2 \vspace{1ex}\\
& \leq & \left(1 + \frac{\psi\tau}{L^2\eta^2(\tau - 1)} \right) \norms{Fx^k - Fx^{k-1}}^2 + \left( 1+ \frac{L^2\eta^2(\tau - 1)}{\psi\tau}\right)\norms{\breve{w}^k}^2 \vspace{1ex}\\
& \leq &  \left(1 + \frac{\psi\tau}{L^2\eta^2(\tau - 1)} \right) L^2\norms{x^k - x^{k-1}}^2 +  \left( 1+ \frac{L^2\eta^2(\tau - 1)}{ \psi\tau }\right)\frac{\tau^2}{\eta^2}\norms{x^k - y^k}^2 \vspace{1ex}\\
& = & \frac{L^2\eta^2(\tau-1) + \psi\tau }{\psi\eta^2(\tau - 1)} \left[ \psi\cdot\norms{x^k - x^{k-1}}^2 + \tau (\tau - 1)\norms{x^k - y^k}^2\right].
\end{array}
\hspace{-4ex}
\end{equation} 
Combining this estimate and \eqref{eq:GR4NI_th1_proof1}, and noting that $\Vc_k \geq 0$, we can show that
\begin{equation*}
\arraycolsep=0.2em
\begin{array}{lcl}
\sum_{l=0}^k\norms{w^l}^2 & \leq & \frac{L^2\eta^2(\tau - 1) + \psi \tau }{\psi\eta^2(\tau - 1)} \sum_{l=0}^k \left[ \psi \cdot \norms{x^{l} - x^{l-1}}^2 +   \tau (\tau - 1) \norms{x^l - y^l}^2\right] \vspace{1ex}\\
& \leq &    \frac{L^2\eta^2(\tau - 1) + \psi\tau }{\psi\eta^2(\tau - 1)}   \left[ \Vc_0 - \Vc_{k+1} \right] \leq   \frac{L^2\eta^2(\tau - 1) + \psi \tau }{\psi\eta^2(\tau - 1)}  \cdot \Vc_0  \vspace{1ex}\\
& = &  \frac{ (\tau^2 - 2L^2\eta^2)\tau}{(\tau^2 - 4L^2\eta )\eta^2(\tau - 1)} \cdot \norms{x^0 - x^{\star}}^2,
\end{array}
\end{equation*} 
which is exactly \eqref{eq:GR4NI_convergence1},  where we have used $\Vc_0 := \tau \norms{y^0 - x^{\star}}^2 + \frac{\tau(\tau - 1)}{2} \norms{x^0 - x^{-1}}^2 = \tau \norms{x^0 - x^{\star}}^2$ due to $x^{-1} = y^0 = x^0$.

Next, if $1.6180 \approx \frac{1 + \sqrt{5}}{2} < \tau  < 1 + \sqrt{3} \approx 2.7321$, then we have $\tau^2 - \tau - 1 > 0$ and $\psi := \tau - \frac{2(\tau^2 - \tau - 1)}{\tau } > 0$.
In this case, using $\norms{x^{k+1} - y^k}^2 \leq 2\norms{x^{k+1} - x^k}^2 + 2\norms{y^k - x^k}^2$ and $\iprods{Fx^k - Fx^{\star}, x^k - x^{\star}} \geq 0$ into \eqref{eq:GR4NI_key_est1}, rearranging the result, and using $\gamma := \frac{\psi}{2}$, we can derive that
\begin{equation}\label{eq:GR4NI_key_est1_v2}
\hspace{-3ex}
\arraycolsep=0.0em
\begin{array}{lcl}
\tau \norms{y^{k+1} - x^{\star}}^2  & + &   \frac{\psi(\tau - 1)}{2} \norms{x^{k+1} - x^k}^2   \leq  \tau \norms{y^k - x^{\star}}^2 +  \frac{\psi(\tau  - 1)}{2}\norms{x^k-x^{k-1}}^2 \vspace{1ex}\\
&&   - {~}  \psi(\tau - 1)\norms{x^k - y^k}^2 -   \frac{(\tau - 1)(\psi^2 - 4L^2\eta^2)}{2\psi} \norms{x^k - x^{k-1}}^2.
\end{array}
\hspace{-4ex}
\end{equation}
Similar to the proof of \eqref{eq:GR4NI_th1_proof2}, we have $\norms{w^k}^2 \leq \frac{\psi^2\tau^2 - 2L^2\eta^2(2\tau^2 - \psi^2)}{(\psi^2 - 4L^2\eta^2)\eta^2\psi} \big[\frac{\psi^2 - 4L^2\eta^2}{2\psi}\norms{x^k - x^{k-1}}^2 + \psi\norms{x^k - y^k}^2\big]$.
Combining this inequality and \eqref{eq:GR4NI_key_est1_v2}, with same argument as in the proof of \eqref{eq:GR4NI_convergence1}, we obtain \eqref{eq:GR4NI_convergence2}.
\Eproof
\end{proof}

\beforesec
\section{Numerical Experiments}\label{sec:numerical_experiments}
\aftersec
We provide an extensive experiment set to verify the algorithms we studied above for both equations and inclusions under new assumptions and parameters.
All the algorithms are implemented in Python running on a single node of a Linux server (called Longleaf) with the configuration: AMD EPYC 7713 64-Core Processor, 512KB cache, and 4GB RAM.

\beforesubsec
\subsection{\mytb{Quadratic Minimax Problems}}\label{subsec:numexp_quad_minimax}
\aftersubsec
Our first example is the following quadratic minimax problem:
\begin{align}\label{prob:minimax}
\min_{u \in \mathbb{R}^{p_1}}\max_{v\in\mathbb{R}^{p_2}}\Big\{ \mathcal{L}(u,v) = f(u) + \Hc(x, y) - g(v) \Big\},
\end{align}
where $\Hc(x, y) := \frac{1}{2}u^{\top}A u + b^{\top}u + u^{\top}L v - \frac{1}{2}v^{\top}B v - c^{\top}v$ such that  $A \in\mathbb{R}^{p_1\times p_1}$ and $B \in \mathbb{R}^{p_2\times p_2}$ are symmetric matrices, $b \in\mathbb{R}^{p_1}$, $c \in\mathbb{R}^{p_2}$ are given vectors, and $L \in \mathbb{R}^{p_1\times p_2}$ is a given matrix. 
The functions $f$ and $g$ are added to possibly handle constraints or regularizers associated with $u$ and $v$, respectively.

First, we denote $x \coloneqq [u,v] \in \mathbb R^p$ for $p := p_1 + p_2$, which is the concatenation of the primal variable $u \in \mathbb R^{p_1}$ and its dual variable $v \in \mathbb R^{p_2}$. 
Next, we define $\mathbf{F} \coloneqq \left[ [A, L], [-L^\top, B] \right]$ as the KKT matrix in $\mathbb R^{p\times p}$ constructed from the four blocks $A, L, -L^\top$, and $B$, and $\mathbf{f} \coloneqq [b, c] \in \mathbb R^p$. 
The operator $F: \mathbb R^p \to \mathbb R^p$ is then defined as $Fx \coloneqq \mathbf{F}x + \mathbf{f}$. 
When $f$ and $g$ are presented, we denote by $T := [\partial f, \partial g]$ the maximally monotone operator constructed from the subdifferentials of $f$ and $g$. 
Then, the optimality condition of \eqref{prob:minimax} becomes $0 \in Fx + Tx$ covered by \eqref{eq:NI}. 
If $f$ and $g$ are absent from \eqref{prob:minimax}, then its optimality condition reduces to $Fx = 0$ as a special case of \eqref{eq:NE}.

\beforesubsubsec
\subsubsection{The unconstrained case as \eqref{eq:NE}}\label{subsubsec:numexp_unconstr_quad_minimax}
\aftersubsubsec
First, we consider the unconstrained minimax problem as an instance of \eqref{prob:minimax} (or in particular, of \eqref{eq:NE}), where $f=0$ and $g=0$. 
We aim to test different variants covered by \eqref{eq:EG4NE}.

\textit{Data generation.} 
In what follows, all random matrices and vectors are generated randomly from the standard normal distribution. 
We generate the matrix $A = Q D Q^\top$, where $Q$ is an orthonormal matrix obtained from the QR factorization of a random matrix, and $D = \diag{d_{1}, \dots, d_{p_1}}$ is the diagonal matrix formed from $d_{1}, \dots, d_{p_1}$ randomly generated and then clipped by a lower bound $\ul{d}$, i.e. $d_{j} \coloneqq \max\{d_{j}, \ul{d}\}$. 
The matrix $B$ is also generated by the same way. 
The matrix $L$ and vectors $b$ and $c$ are randomly generated. 

\textit{Algorithms.}
We implement $5$ different instances of \eqref{eq:EG4NE} by using different choices of $u^k$.
These variants consist of EG with $u^k := Fx^k$ and $\beta = 1$, EG+ with $u^k := Fx^k$ and $\beta := 0.5$, PEG with $u^k := Fy^{k-1}$ and $\beta = 1$, PEG+ with $u^k := Fy^{k-1}$ and $\beta := 0.5$, and GEG with $u^k = 1.35 Fx^k - 0.25 Fy^{k-1} - 0.1 Fx^{k-1}$ and $\beta = 0.95$.
The reason we choose different $\beta$ in these algorithms is to avoid identical performance between EG and EG+, PEG and PEG+.
For GEG, we performed  an experiment and found at least one possible pair $(\alpha_1, \alpha_2) := (1.35, -0.25)$ that works relatively well.

\textit{Parameters.}
We choose the stepsize $\eta$ in each algorithm based on a grid search and then tuned it manually to obtain the best possible performance.
We run all algorithms up to $5000$ iterations (or $10000$ iterations for PEG and PEG+ due to one evaluation of $F$ per iteration) on 10 problems instances, and report the average of the relative residual norm $\norms{Fx^k}/\norms{Fx^0}$.
The starting point $x^0$ is always chosen as $x^0 = 0.01\cdot \texttt{ones}(p)$.

\textit{Experiment setup.} We perform four different experiments. 
In \textit{Experiment 1} and \textit{Experiment 2}, we choose $\ul{d} = 0.1$ (monotone), and run the five algorithms on 10 problem instances for each case: $p = 1000$ and $p = 2000$, respectively. 
In \textit{Experiment 3} and \textit{Experiment 4}, we choose $\ul{d} = -0.1$ (possibly nonmonotone) and run the five algorithms on 10 problem instances for each case: $p = 1000$ and $p = 2000$, respectively.
Then, we report the mean of the relative operator norm $\norms{Fx^k}/\norms{Fx^0}$ against the number of iterations $k$ and the number of operator evaluations $Fx^k$, respectively over the 10 instances.

\begin{figure}[hpt!]
\centering
\includegraphics[width=\linewidth]{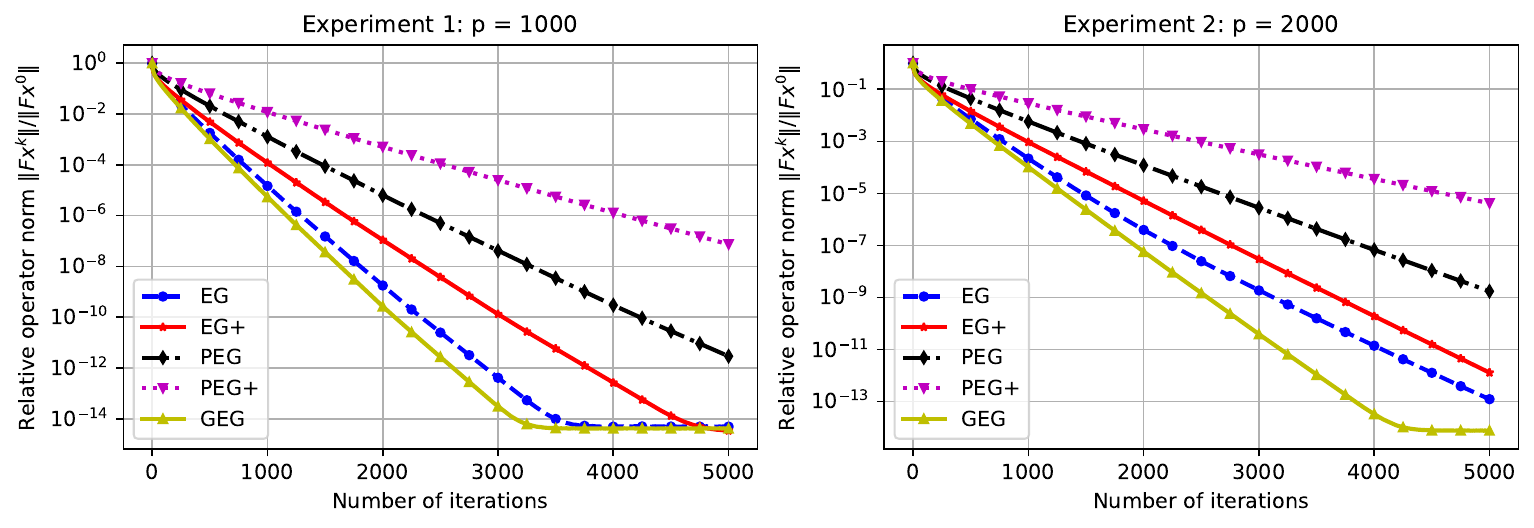}
\includegraphics[width=\linewidth]{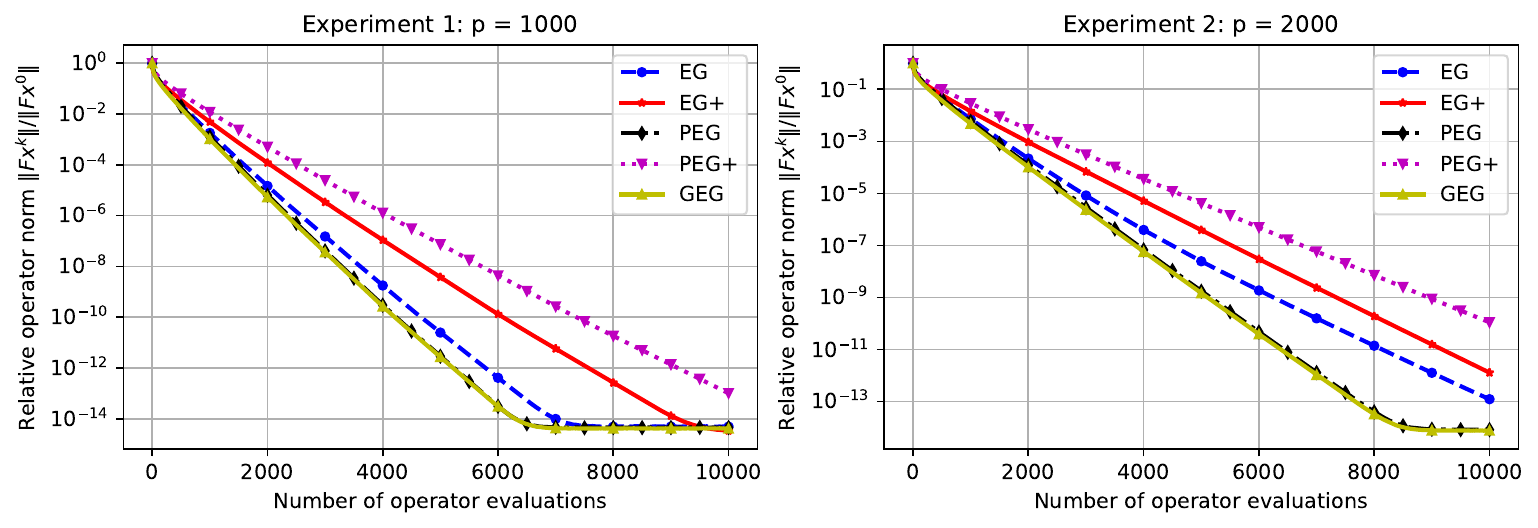}
\caption{
Comparison of the five variants of \eqref{eq:EG4NE} for solving \eqref{prob:minimax} as a special case of \eqref{eq:NE}. 
Two experiments, each is on 10 problem instances of  the monotone setting. 
}
\label{fig:unconstr_minimax1}
\vspace{-2ex}
\end{figure}

\begin{figure}[hpt!]
	\centering
	\includegraphics[width=\linewidth]{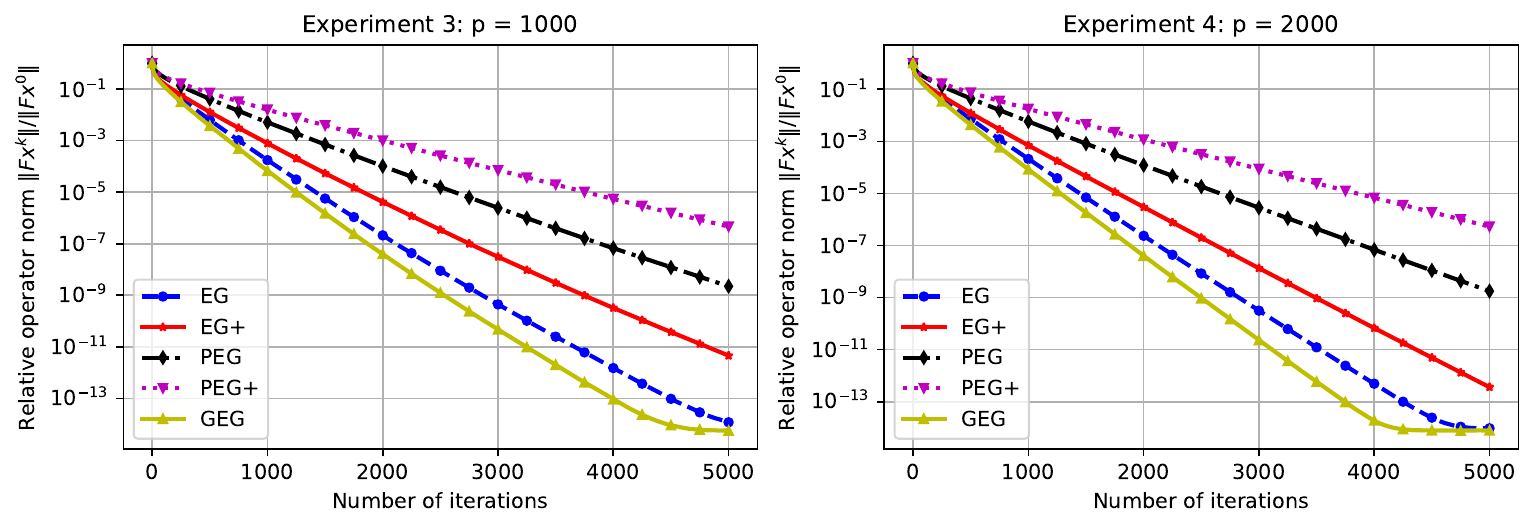}
	\includegraphics[width=\linewidth]{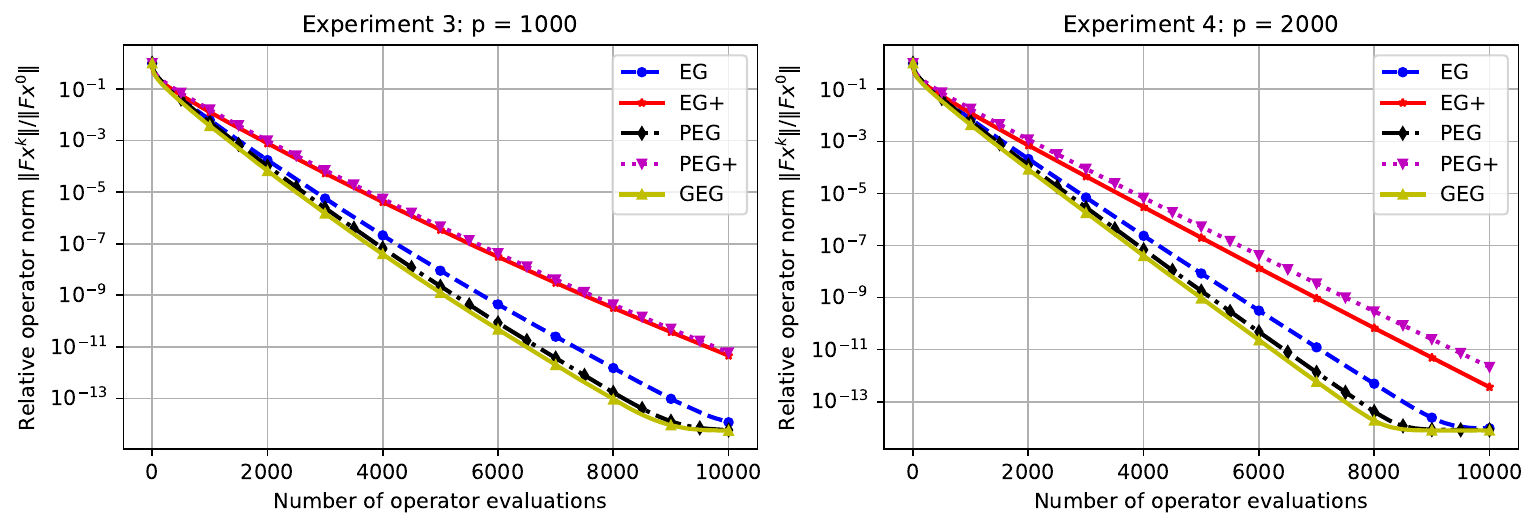}
	\caption{
	Comparison of the five variants of \eqref{eq:EG4NE} for solving \eqref{prob:minimax} as a special case of \eqref{eq:NE}. 
	Two experiments, each is on 10 problem instances of  the possibly nonmonotone setting. 
	}
	\label{fig:unconstr_minimax2}
\vspace{-3ex}	
\end{figure}

\textit{Results.} 
The numerical results after 5000 iterations and after 10000 operator evaluations are reported in Figures \ref{fig:unconstr_minimax1} and \ref{fig:unconstr_minimax2}, respectively. 

From Figures \ref{fig:unconstr_minimax1} and \ref{fig:unconstr_minimax2}, we observe that GEG outperforms the other algorithms across all experiments, particularly in larger instances ($p=2000$). The performances of the remaining algorithms are consistent across all experiments and can be ranked in terms of the number of iterations as follows: EG $>$ EG+ $>$ PEG $>$ PEG+. In terms of operator evaluations, PEG slightly outperforms EG, while the performances of PEG+ and EG+ become more comparable. This can be attributed to PEG/PEG+ saving one operator evaluation per iteration. Interestingly, GEG continues to deliver the best overall performance in terms of the number of operator evaluations and is comparable to PEG in this regard.

\beforesubsubsec
\subsubsection{The constrained case as \eqref{eq:NI}}\label{subsubsec:numexp_constr_quad_minimax}
\aftersubsubsec
Next, we consider the minimax problem \eqref{prob:minimax} with the constraints $u \in \Delta_{p_1}$ and $v \in \Delta_{v_2}$, where $\Delta_{p_1}$ and $\Delta_{p_2}$ are the standard simplexes in $\mathbb R^{p_1}$ and $\mathbb R^{p_2}$, respectively. 
To handle this constraint, in \eqref{prob:minimax}, we use $f(u) = \delta_{\Delta_{p_1}}(u)$ and $g(v) = \delta_{\Delta_{p_2}}(v)$, where $\delta_{\mathcal X}$ is the indicator function of a closed convex set $\mathcal X$. 
The optimality condition of \eqref{prob:minimax} becomes $0 \in Fx + Tx$, where $T \coloneqq [\partial\delta_{\Delta_{p_1}}, \partial\delta_{\Delta_{p_2}}]$ from $\mathbb R^p$ to $2^{\mathbb R^p}$ is a maximally monotone operator.

\textit{Experiment setup.} 
We consider two settings of \eqref{prob:minimax}: \textit{Test 1} with monotone operator $F$ and \textit{Test 2} with [possibly] nonmonotone operator $F$. 
In \textit{Test 1}, we generate data as described in unconstrained case. 
In \textit{Test 2},  data is generated very similarly, except that the diagonal elements $d_j$ of $D$ are sampled from a uniform distribution $\mathbb{U}(-10,10)$.
In each setting, we consider two experiments as before corresponding to $p=1000$ and $p=2000$, respectively.
Then, we report the mean of the relative operator norm $\frac{\norms{\Gc_\eta x^k}}{\norms{\Gc_\eta x^0}}$ over 10 problem instances, where $\Gc_\eta \coloneqq \eta^{-1} \left( x - J_{\eta T}(x-\eta Fx) \right)$ is defined in \eqref{eq:FB_residual}.

\textit{Algorithms and parameters.} 
In \textit{Test 1}, we implement three different variants of \eqref{eq:EG4NI}: EG2 with $u^k = Fx^k$ and $\beta = 1$, EG2+ with $u^k = Fx^k$ and $\beta = 0.5$, GEG2 with $u^k = 1.35Fx^k - 0.45Fy^{k-1} + 0.1Fx^{k-1}$ and $\beta = 0.975$, \eqref{eq:RFBS4NI}, and two variants of \eqref{eq:GR4NI}: GR2 with $\tau = \frac{1+\sqrt{5}}{2}$ and GR2+ with $\tau = \frac{1}{2}\left( \frac{1+\sqrt{5}}{2} + 1 + \sqrt{3} \right) = \frac{3 + 2\sqrt{3} + \sqrt{5}}{4}$.
Next, in \textit{Test 2}, we test 3 variants of \eqref{eq:FBFS4NI}: FBFS2 with $u^k = Fx^k$ and $\beta = 1$, FBFS2+ with $u^k = Fx^k$ and $\beta = 0.25$, and GFBFS2 with $u^k = 1.45 Fx^k - 0.45 Fy^{k-1}$ and $\beta = 1$. 
The stepsize of each algorithm is obtained from a grid search and then fine tuned to obtain the best possible performance. 
The starting points are always chosen as $x^0 \coloneqq 0.01\cdot\texttt{ones}(p)$.

\begin{figure}[hpt!]
\centering
\includegraphics[width=\linewidth]{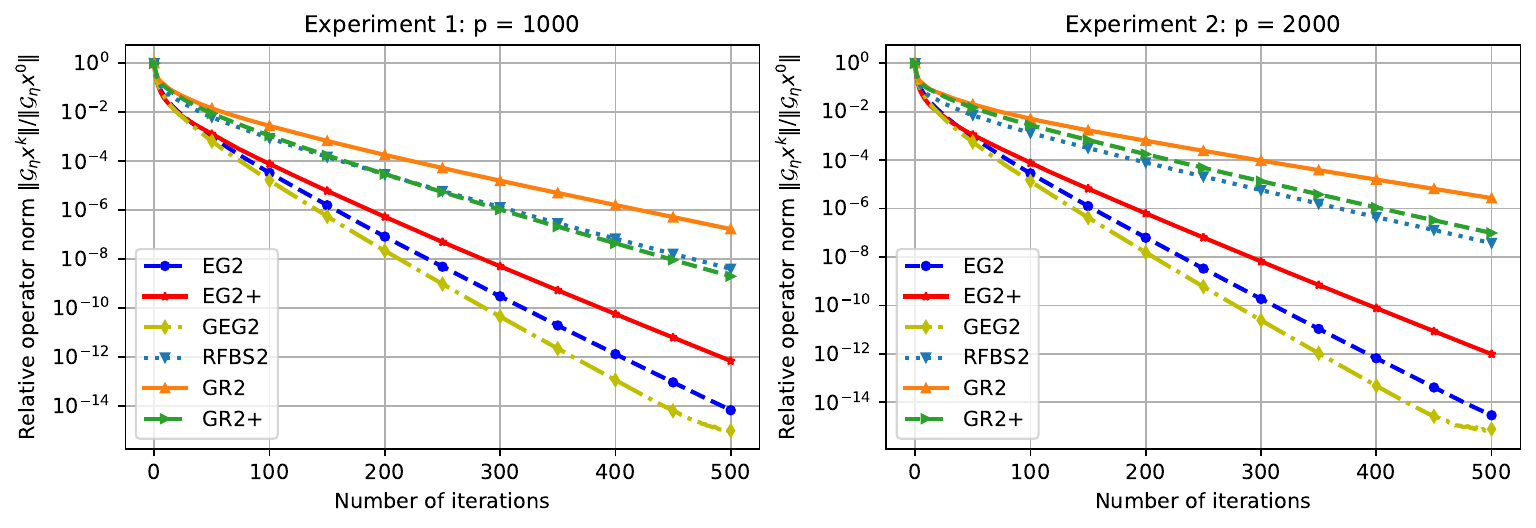}
\includegraphics[width=\linewidth]{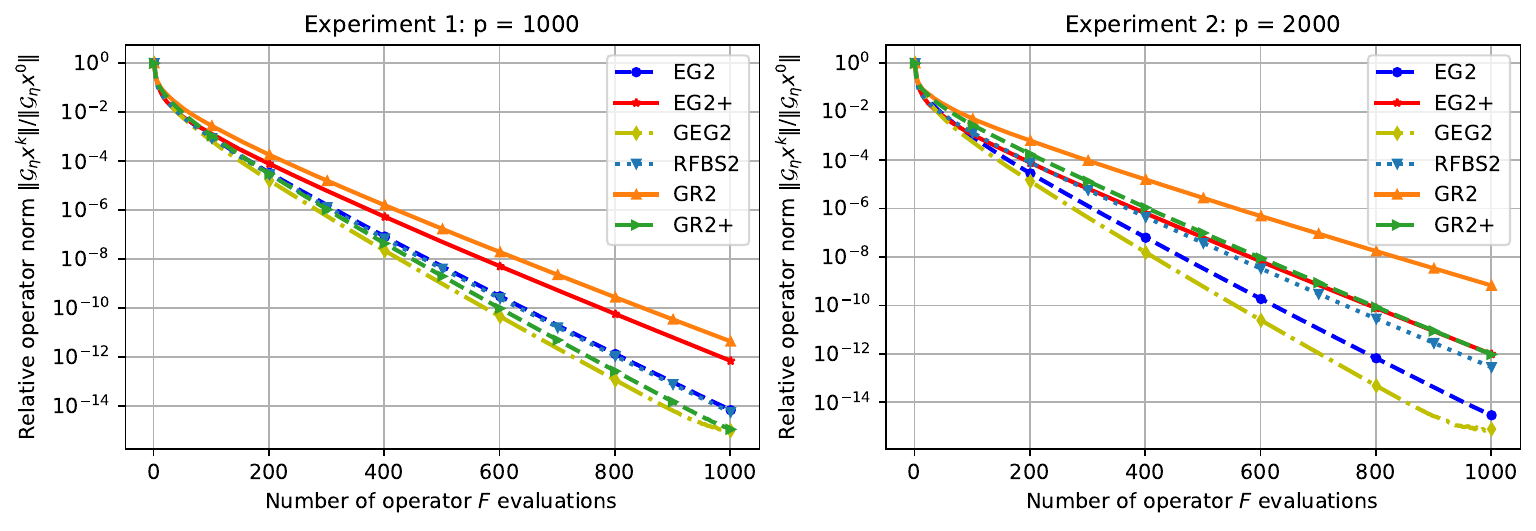}
\caption{Comparison of 6 different algorithms for solving (NI) in \textit{Test 1}. The plot reveals the mean of 10 problem instances.}
\label{fig:constr_minimax_1}
\vspace{-2ex}
\end{figure}

\begin{figure}[hpt!]
\centering
\includegraphics[width=\linewidth]{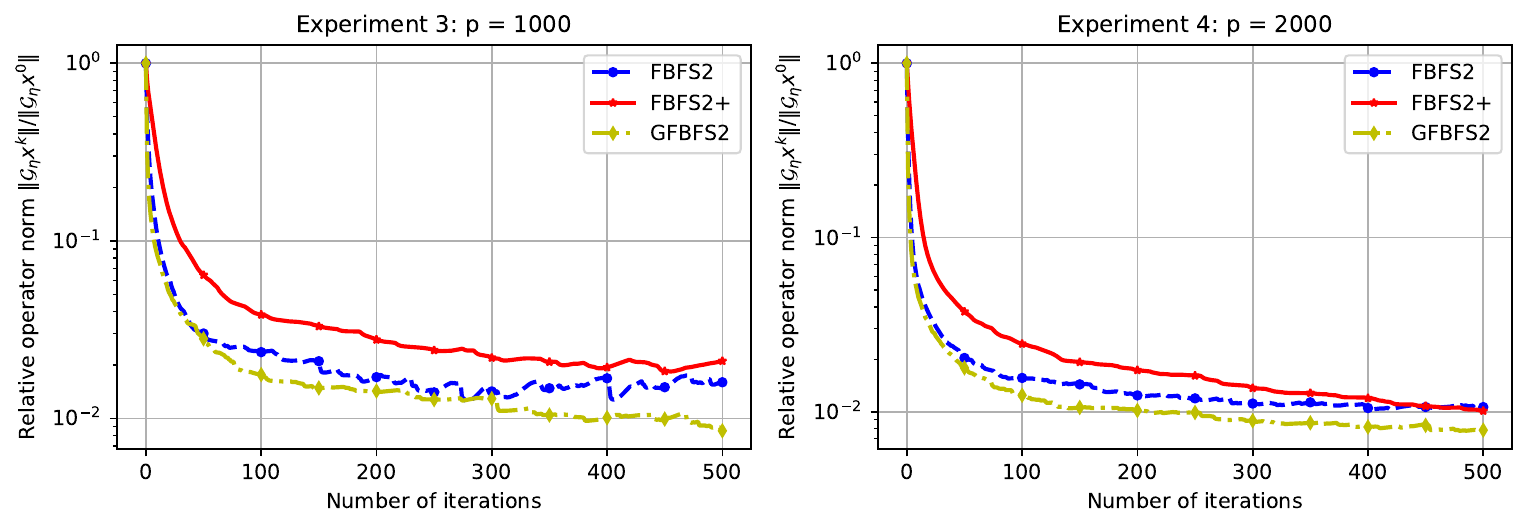}
\caption{Comparison of 3 different algorithms for solving (NI) in \textit{Test 2}. The plot reveals the mean of 10 problem instances.}
\label{fig:constr_minimax_2}
\vspace{-3ex}
\end{figure}

\textit{Results.} 
The numerical results of \textit{Test 1} are reported in Figure \ref{fig:constr_minimax_1}.
As we can see from Figure \ref{fig:constr_minimax_1}, GEG2 performs better than EG2 and outperforms all the remaining algorithms in both experiments. 
We also see that the three variants of \eqref{eq:EG4NI} outperform RFBS2 when they reach the solution accuracy in the range of $10^{-12}$ to $10^{-14}$ after $500$ iterations compared to $10^{-8}$ of RFBS2. 
However, as compensation, RFBS2 saves one operator evaluation and one resolvent evaluation per iteration compared to the three variants of \eqref{eq:EG4NI}. 
Finally, the performance of \eqref{eq:GR4NI} has been improved by extending the value of $\tau$ from $\tau = \frac{1+\sqrt{5}}{2}$ to $\tau = \frac{1}{2} ( \frac{1+\sqrt{5}}{2} + 1 + \sqrt{3} )$.

The numerical results of \textit{Test 2} are reported in Figure \ref{fig:constr_minimax_2}.
This figure shows that GFBFS2 provides a better performance than FBFS2 and FBFS2+. However, their performances are comparable and become more stable in a higher dimensional space.

\beforesubsec
\subsection{\mytb{Bilinear Matrix Games}}\label{subsec:numexp_bilinear_game}
\aftersubsec
Our second example is the following classical bilinear matrix game problem:
\begin{equation}\label{eq:matrix_game}
\min_{u \in \mathcal U} \max_{v \in \mathcal V}\big\{ \Hc(x, y) :=  \iprods{ \mbf{L} u, v } \big\},
\end{equation}
where $\mbf{L} \in \mathbb R^{p_2 \times p_1}$, $\mathcal U \subset \mathbb R^{p_1}$ and $\mathcal V \subset \mathbb R^{p_2}$ are closed convex sets. 

The optimality condition  $ 0 \in Fx + Tx$ of \eqref{eq:matrix_game} is a special case of \eqref{eq:NI}, where we define $x := [u, v]$, $Fx := [\mbf{L}^\top v; -\mbf{L}u]$, $Tx := [\partial \delta_{\mathcal U}, \partial \delta_{\mathcal V}]$, and $\delta_{\mathcal X}$ is the indicator function of $\mathcal X$. 
Clearly, $F$ is skew-symmetric, and thus monotone.
However, it is well-known that \eqref{eq:FBS4NI} for solving this problem may not converge.
Therefore, it is important to apply extragradient-type methods.

\textit{Models and data.}
We follow the suggestions in \cite{Nemirovski2009} to generate data for our test.
We consider a symmetric matrix $\mbf{L}$ of size $q\times q$, where $q:= p_1 = p_2$, with two options:
\begin{itemize}
\item $\mbf{L}$ belongs to the first family with $\mbf{L}_{ij} := \big(\frac{i + j - 1}{2q - 1} \big)^{\alpha}$ for $1 \leq i, j \leq q$.
\item $\mbf{L}$ belongs to the second family with $\mbf{L}_{ij} := \big( \frac{\vert i - j\vert + 1}{2q - 1} \big)^{\alpha}$ for $1 \leq i, j \leq q$.
\end{itemize}
Here, $\alpha > 0$ is a given parameter.

In addition to the above two artificial instances, we also follow \cite{nemirovski2013mini} and consider the Policeman vs. Burglar problem as follows.
There are $q$ houses in a city, the $i$-th house is with the wealth $w_i$. 
Every evening, Burglar chooses a house $i$ to attack, and Policeman chooses his post near a house $j$, $1 \leq i, j \leq q$. 
After Burglary starts, Policeman becomes aware where it happens, and his probability to catch Burglar is  $\exp\{-\theta \cdot \text{dist}(i, j)\}$,  where $\text{dist}(i, j)$ is the distance between houses $i$ and $j$. 
On the other hand, Burglar seeks to maximize his expected profit $w_i(1 - \exp\{-\theta \cdot \text{dist}(i, j)\})$,  the interest of Policeman is completely opposite. 
This leads to the third family of symmetric matrices $\mbf{L}$ with $\mbf{L}_{ij} = w_i(1 - \exp\{-\theta \cdot \text{dist}(i, j)\})$ for $1 \leq i, j \leq q$.
The set $\Uc$ and $\Vc$ are chosen to be the standard simplexes in the respective spaces. 

In the above three problems, we fix $q: = 500$. 
Moreover, in the first two problems, we also take $\alpha = 1$.
In the Burglar and Policeman problem, we choose $\text{dist}(i,j) \coloneqq |i - j|$, $\theta := 0.005$, and the vector of wealth $w \in \mathbb R^q$ is generated randomly from a standard normal distribution and then take the absolute value so that it is nonnegative. 

\textit{Experiment with 5 variants of \eqref{eq:EG4NE}.} 
We first reformulate the optimality $0 \in Fx + Tx$ of \eqref{eq:matrix_game} into \eqref{eq:NE} by using Tseng's FBFS operator as $\hat{F}x := x - J_{\lambda T}(x - \lambda Fx) - \lambda(Fx - F(J_{\lambda T}(x - \lambda Fx) ))$ for any $\lambda > 0$.
It is clear that solving $0 \in Fx^{\star} + Fx^{\star}$ is equivalent to solving $\hat{F}x^{\star} = 0$.
We implement our \eqref{eq:EG4NE} with different choices of $u^k$ as in Subsection \ref{subsec:numexp_quad_minimax} to solve $\hat{F}x^{\star} = 0$ with $\lambda := 0.5$, except that now GEG uses $u^k = 1.1Fx^{k} - 0.1Fy^{k-1}$ and $\beta = 0.1$. 
The starting point for all experiments is always chosen as $x^0 := 0.5\cdot\texttt{ones}(p)$.
For each algorithm, the stepsize $\eta$ is chosen by a grid search so that it achieves the best possible performance. 
The numerical results  are reported in Figure \ref{fig:matrix_game_1}. 

\begin{figure}
\centering
\includegraphics[width=1\linewidth]{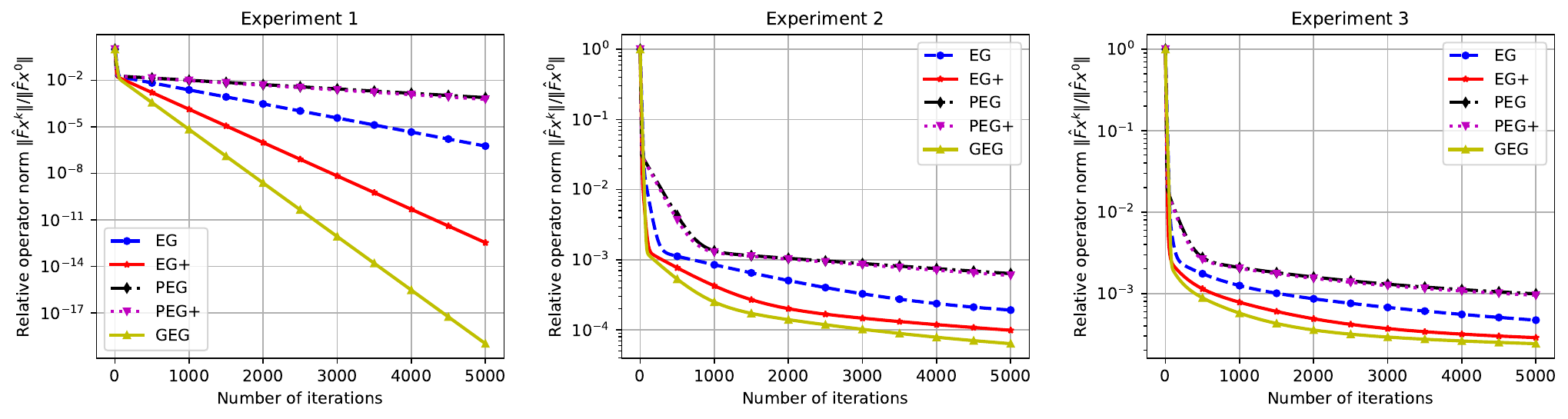}
\includegraphics[width=1\linewidth]{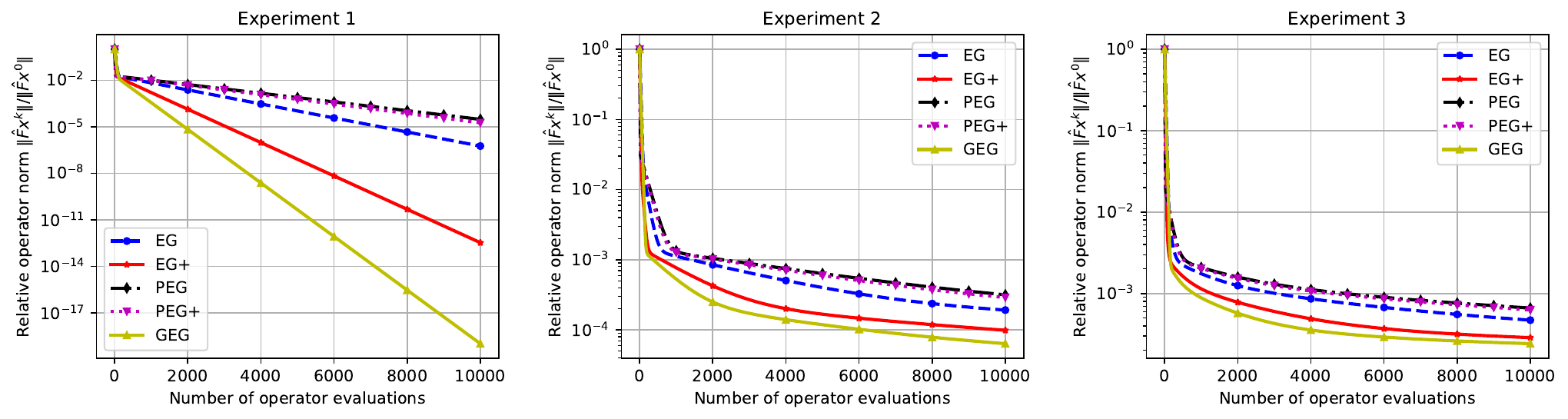}
\caption{The performance of $5$ variants of \eqref{eq:EG4NE} for solving  $\hat{F}x = 0$ resulting from \eqref{eq:matrix_game} using three different problems.}
\label{fig:matrix_game_1}
\vspace{-4ex}
\end{figure}

As we can observe from Figure \ref{fig:matrix_game_1}, our GEG outperforms all the remaining algorithms on all $3$ problems in both two criteria, namely the number of iterations and the number of operator evaluations. 
Among the remaining algorithms, while EG+ provides better performances than EG, both still outperform PEG and PEG+, which have comparable performances in all experiments. 
The differences between the performances of PEG/PEG+ and the other algorithms are diminished  as expected when we report the number of operator evaluations, since PEG and PEG+, each saves one operator evaluation per iteration compared to EG, EG+, or our new GEG.

\textit{Experiment with different variants of \eqref{eq:EG4NI}, \eqref{eq:RFBS4NI}, and \eqref{eq:GR4NI}}
Now, we will test different variants of \eqref{eq:EG4NI}, \eqref{eq:RFBS4NI}, and \eqref{eq:GR4NI} described in Subsection \ref{subsec:numexp_quad_minimax} for solving \eqref{eq:matrix_game} as a special case of \eqref{eq:NI}. 
Again, for each algorithm, the stepsize $\eta$ is obtained from a grid search and then tuned (if necessary) so that it produces the best possible performance. 
Our numerical results of this test are depicted in Figure \ref{fig:matrix_game_2}.

\begin{figure}
	\centering
	\includegraphics[width=1\linewidth]{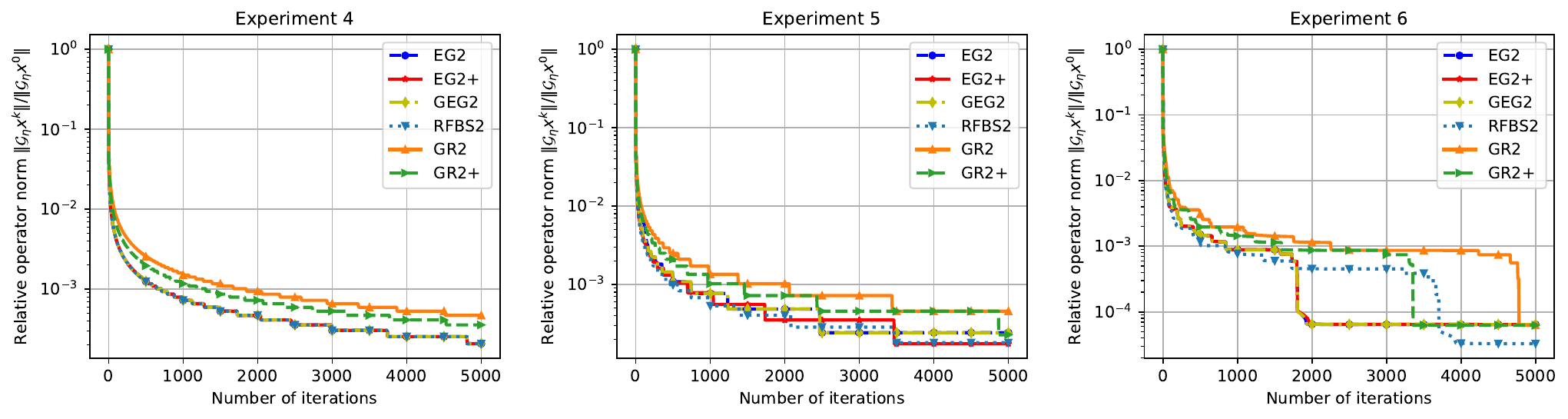}
	\caption{The performance of $6$ different algorithms for solving  $0 \in Fx + Tx$ resulting from \eqref{eq:matrix_game} using three different problems.}
	\label{fig:matrix_game_2}
	\vspace{-3ex}
\end{figure}

Figure \ref{fig:matrix_game_2} shows that the three variants of \eqref{eq:EG4NI} and \eqref{eq:RFBS4NI} have comparable performance in all three problems, and they also have a better performance than GR2 and GR2+. 
Unlike the previous tests, in this particular example, we do not see a clear improvement of GEG2 over EG2, but it is better than EG2+.
We can also verify the effectiveness of extending the range of $\tau$ in \eqref{eq:GR4NI} from this set of experiments when GR2+ with a larger value of $\tau$ produces a slightly better performance than GR2.

\beforesubsec
\subsection{\mytb{Regularized Logistic Regression with Ambiguous Features}}\label{subsec:numexp_logistic}
\aftersubsec
Finally, we consider a standard regularized logistic regression model associated with a given dataset $\{(\hat X_i, y_i)\}_{i=1}^N$, where $\hat X_i$ is an i.i.d. sample of a feature vector and $y_i \in \{0,1\}$ is the corresponding label of $\hat X_i$. 
Unfortunately, $\hat X_i$ is ambiguous, i.e. it belongs to one of $m$ possible examples $\{X_{ij}\}_{j=1}^m$. 
Since we do not know which $\hat X_i$ to evaluate a loss, we consider the worst-case loss $f_i(w) \coloneqq \max_{1\le j \le m} \ell(\iprod{X_{ij}, w}, y_i)$ computed from $m$ examples, where $\ell(t, s) \coloneqq \log(1+\exp(t)) - st$ is the standard logistic loss.

Using the fact that $\max_{1 \le j \le m} \ell_j(\cdot) = \max_{v \in \Delta_m} \sum_{j=1}^m v_j \ell_j(\cdot)$, where $\Delta_m$ is the standard simplex in $\mathbb R^m$, we can model this regularized logistic regression into the following minimax problem:
\begin{align}\label{prob:minimax_logit}
\min_{w \in \mathbb R^d} \max_{v \in \mathbb R^m} \Big\{ \mathcal L(w,v) \coloneqq \frac{1}{N} \sum_{i=1}^N \sum_{j=1}^m v_j \ell(\iprod{X_{ij}, w}, y_i) + \gamma R(w) - \delta_{\Delta_m}(v)  \Big\},
\end{align}
where $R(w) \coloneqq \norm{w}_1$ is an $\ell_1$-norm regularizer used to introduce sparsity to $w$, $\gamma > 0$ is a regularization parameter, and $\delta_{\Delta_m}$ is the indicator of $\Delta_m$ that handles the constraint $v \in \Delta_m$.

Let us define $x \coloneqq [w;v]$, $T(x) \coloneqq \left[ \gamma \partial R(w); \partial \delta_{\Delta_m}(v) \right]$, and
\begin{equation*}
F_i(x) \coloneqq \Big[ \sum_{j=1}^m v_j\ell'(\iprod{X_{ij}, w}, y_i) X_{ij}; - \ell(\iprod{X_{i1}, w}, y_i); \cdots; - \ell(\iprod{X_{im}, w}, y_i) \Big], 
\end{equation*}
where $\ell'(t, s) = \frac{\exp(t)}{1+\exp(t)} - s$. Then, the optimality condition of \eqref{prob:minimax_logit} can be written as $0 \in Fx + Tx$, where $F \coloneqq \frac{1}{N}\sum_{i=1}^N F_i$.

\textit{Data generation.} 
We use the following two real datasets for the experiments: \texttt{w7a} ($300$ features and $24,692$ data points) and \texttt{duke breast-cancer} ($7,129$ features and $44$ data points) downloaded from \texttt{LIBSVM} \cite{CC01a}. 
We first normalize the feature vector $\hat X_i$ such that each sample has unit norm, and add a column of all ones to address the bias term. 
To generate ambiguous features, we take the nominal feature vector $\hat X_i$ and add a random noise generated from a standard normal distribution. 
In our test, we choose $\gamma := 5\cdot10^{-4}$ and $m\coloneqq 5$. 
The starting point for every experiment is chosen as $x^0 := 0.5\cdot \texttt{ones}(p)$.  

\textit{Experiment with 5 variants of \eqref{eq:EG4NE}.} 
We first reformulate the optimality $0 \in Fx + Tx$ of \eqref{prob:minimax_logit} into \eqref{eq:NE} by using Tseng's FBFS operator $\hat F$ as stated in Subsection \ref{subsec:numexp_bilinear_game}.
We implement our \eqref{eq:EG4NE} with different choices of $u^k$ as in Subsection \ref{subsec:numexp_quad_minimax} to solve $\hat{F}x^{\star} = 0$, except that GEG now uses the direction $u^k = 0.7 Fx^k + 0.3 Fy^{k-1}$ and $\beta = 1$. 
As before, for each algorithm, the stepsize $\eta$ is tuned manually to obtain the best possible performance. 
Then, the relative norm $\frac{\norms{\hat F x^k}}{\norms{\hat F x^0}}$ against the number of iterations $k$ of $5$ variants of \eqref{eq:EG4NE} are reported in Figure \ref{fig:logistic_reg_NE}.

\begin{figure}[h]
	\centering
	\includegraphics[width=\linewidth]{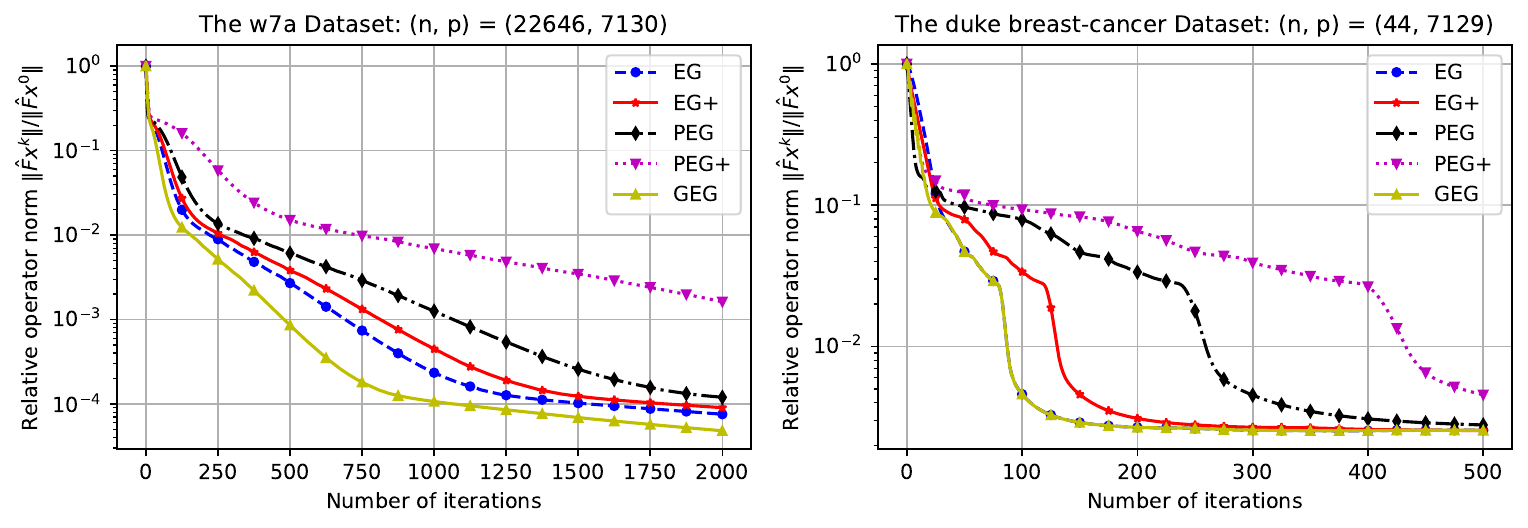}
	\includegraphics[width=\linewidth]{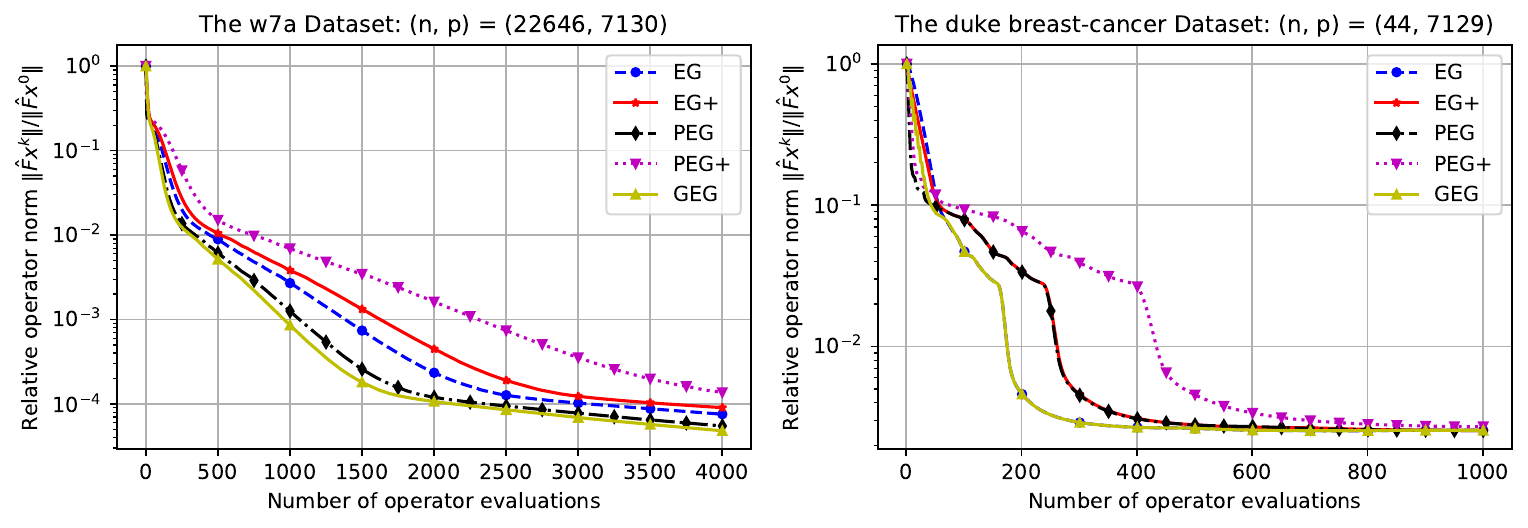}
	\caption{The performance of $5$ variants of \eqref{eq:EG4NE} for solving  $\hat{F}x = 0$ resulting from \eqref{prob:minimax_logit} using the two datasets: \texttt{w7a} and \texttt{duke breast-cancer} from \texttt{LibSVM}.}
	\label{fig:logistic_reg_NE}
\vspace{-3ex}	
\end{figure}

From Figure \ref{fig:logistic_reg_NE}, we can see that for the \texttt{w7a} dataset, our GEG outperforms all the remaining variants in both the number of iterations and the number of operator evaluations. 
If we consider the number of iterations, then the remaining variants can be ranked as EG $>$ EG+ $>$ PEG $>$ PEG+ based on their performance, which agrees with the observation in Subsection \ref{subsec:numexp_quad_minimax}. 
For the \texttt{duke breast-cancer} dataset, we can see that the ranking of the first four classical variants of \eqref{eq:EG4NE} based on their performances are still consistent with those in the \texttt{w7a}. 
The GEG variant still produces the best performance, though there is no significant difference from EG.
If we consider the number of operator evaluations, then PEG becomes the second-best algorithm, which outperforms EG in the experiment with \texttt{w7a} dataset and is comparable to EG+ in the experiment with the \texttt{duke breast-cancer}. 
At the same time, PEG+ also provides a better performance, though still the humblest, thanks to the saving of one operator evaluation per iteration. 
It is worth noting that our GEG still provides the best performance when the number of operator evaluations is taken into account.

\textit{Experiment with different variants of \eqref{eq:EG4NI}, \eqref{eq:RFBS4NI}, and \eqref{eq:GR4NI}}
Next, we test different variants of \eqref{eq:EG4NI}, \eqref{eq:RFBS4NI}, and \eqref{eq:GR4NI} described in Subsection \ref{subsec:numexp_quad_minimax} for solving \eqref{prob:minimax_logit} as a special case of \eqref{eq:NI}, but now GEG2 uses the direction $u^k = 0.7 Fx^k + 0.3 Fy^{k-1}$ and $\beta = 1$.
Now, for each algorithm, the stepsize $\eta$ is tuned manually so that it provides the best possible performance. 
The corresponding relative norm $\frac{\norms{ \Gc_{\eta} x^k}}{\norms{ \Gc_{\eta}x^0}}$ against the number of iterations is presented in Figure \ref{fig:logistic_reg_NI}.

\begin{figure}[h]
\vspace{-3ex}
	\centering
	\includegraphics[width=\linewidth]{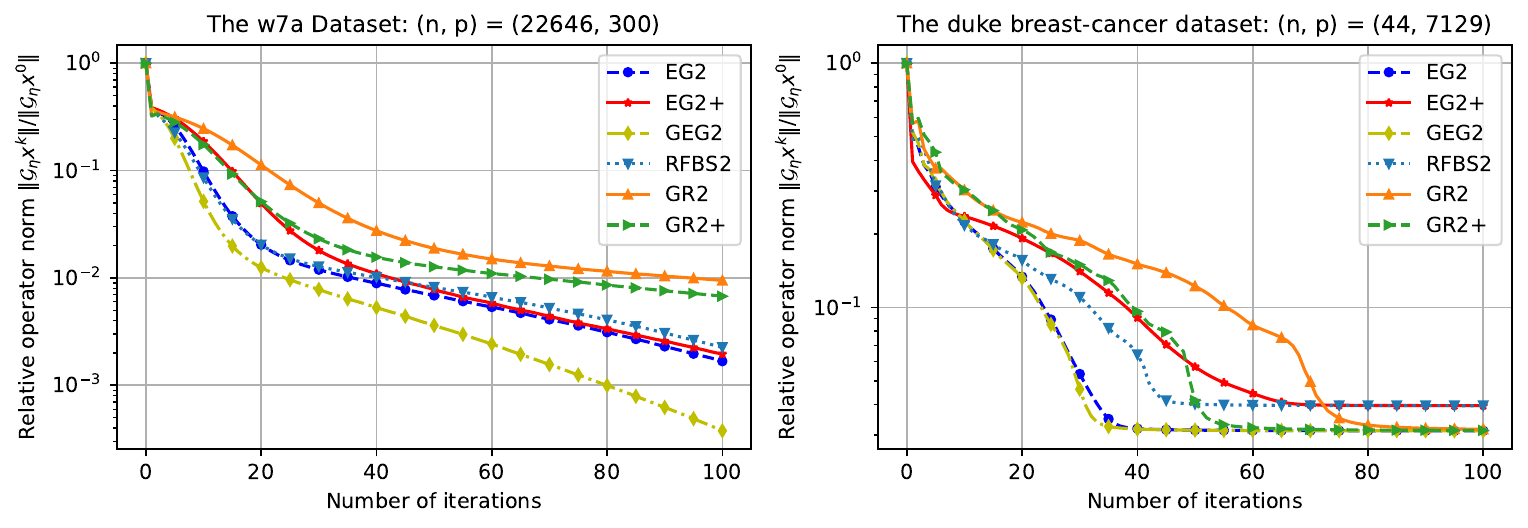}
	\caption{The performance of $5$ variants of different variants of \eqref{eq:EG4NI}, \eqref{eq:RFBS4NI}, and \eqref{eq:GR4NI} for solving  $0 \in Fx + Tx$ resulting from \eqref{prob:minimax_logit} using the two datasets: \texttt{w7a} and \texttt{duke breast-cancer} from \texttt{LibSVM}.}
	\label{fig:logistic_reg_NI}
\vspace{-3ex}	
\end{figure}

We can see from Figure \ref{fig:logistic_reg_NI} that for the \texttt{w7a} dataset, GEG2 highly outperforms all its competitors and reaches the accuracy of $10^{-4}$ in only 100 iterations, which is pretty fast. 
EG2, EG2+, and RFBS2 are the second-fastest algorithms, and their performance is comparable after the first 40 iterations. 
Again, we still see an improvement of GR2+ compared to GR2, which illustrates the effectiveness of extending the range of $\tau$. 
Finally, for the \texttt{duke breast-cancer} dataset, the difference between the performances of the algorithms becomes less clear, but GEG2 still shows the best performance overall.

{\vspace{1ex}
\noindent\textbf{Data availability.}
This paper uses both synthetic and publicly available data.
The procedure of generating synthetic data is clearly described in the paper.
The publicly available data is from \href{https://www.csie.ntu.edu.tw/~cjlin/libsvm}{LIBSVM} \cite{CC01a}.
}

\vspace{1ex}
\noindent\textbf{Acknowledgements.}
This work is  partially supported by the National Science Foundation (NSF), grant no. NSF-RTG DMS-2134107 and the Office of Naval Research (ONR), grant No. N00014-20-1-2088 (2020-2023) and grant No. N00014-23-1-2588 (2023-2026).

\appendix
\beforesec
\section{Technical Lemmas}\label{apdx:sec:proofs}
\aftersec
The following technical lemmas will be used to prove Theorem~\ref{th:EG4NE_convergence1} and Theorem~\ref{th:EG4NE_convergence2}.

\begin{lemma}\label{le:EG4NE_para_conditions}
Suppose that $L\rho \leq \Delta$ for $\Delta$ defined by \eqref{eq:EG4NE_kappa_i_quantities}.
Then, $\underline{\eta}$ and $\bar{\eta}$ defined by \eqref{eq:EG4NE_eta_bounds} are well-defined and $\underline{\eta} \geq 0$ and $\bar{\eta} \geq 0$.
\begin{compactitem}
\item[$($a$)$] If we choose $0 < \beta \leq  \frac{r}{r+\kappa_2}$,  then $\underline{\eta}_2 \le \underline{\eta}_1 \le \bar{\eta}_1 \le \bar{\eta}_2$, leading to $\underline{\eta} \leq \bar{\eta}$.
\item[$($b$)$] If we choose $\frac{r}{r+\kappa_2} \leq \beta < 1$, then $\underline{\eta}_1 \le \underline{\eta}_2 \le \bar{\eta}_2 \le \bar{\eta}_1$, leading to $\underline{\eta} \leq \bar{\eta}$.
\item[$($c$)$] If $\kappa_2 = 0$, then we can choose $\beta = 1$, leading to  $\underline{\eta} = \underline{\eta}_1 \leq \bar{\eta} = \bar{\eta}_1$.
\end{compactitem}
Moreover, in all cases above, $C_1$ and $C_2$ defined by \eqref{eq:EG4NE_C_constants} are nonnegative.
\end{lemma}

\begin{proof}
First, given $C_1$ and $C_2$ in \eqref{eq:EG4NE_C_constants}, to have $C_1 \geq 0$, we need to choose $\eta$ such that $(1+r)L\eta^2 - \beta \eta + 4\mu\rho < 0$.
This implies $\underline{\eta}_1 := \frac{\beta - \sqrt{\beta^2 - 16(1+r)\mu L\rho}}{2(1+r)L} \leq \eta \leq \bar{\eta}_1 := \frac{\beta + \sqrt{\beta^2 - 16(1+r)\mu L\rho }}{2(1+r)L}$, provided that $L\rho \leq \frac{\beta^2}{16(1+r)\mu} = \frac{\beta^2(r + 2\kappa_2)}{16(1+r)r}$.
This choice of $\eta$ leads to the first line of \eqref{eq:EG4NE_eta_bounds}.

Next, to guarantee $C_2 \geq 0$, we need to choose $\eta$ such that $\kappa_2\alpha  L \eta^2 - (1-\beta)\eta + 2(1-\mu) \rho \leq 0$.
This implies $\underline{\eta}_2 := \frac{1-\beta  -  \sqrt{(1-\beta)^2 - 8\alpha(1-\mu)\kappa_2L\rho}}{2\alpha \kappa_2L} \leq \eta \leq \bar{\eta}_2 := \frac{1-\beta  + \sqrt{ (1-\beta)^2 - 8\alpha(1-\mu)\kappa_2L\rho}}{2\alpha \kappa_2L}$, provided that $L\rho \leq \frac{(1-\beta)^2}{8\alpha (1-\mu)\kappa_2} = \frac{(1-\beta)^2(r + 2\kappa_2)r}{16(1+r)\kappa_2^2}$.
This $\eta$ leads to the second line of \eqref{eq:EG4NE_eta_bounds}.

Now, under the choice of $\alpha$ and $\mu$ above, we can show that $\underline{\eta} \le \bar{\eta}$. 
Indeed, this is obvious when $\beta = 1$, proving (c).
 
For $\beta \in (0,1)$, we define $\varphi_1(\eta) := (1+r)L\eta^2 - \beta \eta + 4\mu\rho$ and $\varphi_2(\eta) := \kappa_2\alpha  L \eta^2 - (1-\beta)\eta + 2(1-\mu) \rho$ as two quadratic functions of $\eta$.

\vspace{0.5ex}
\noindent\textbf{Case 1:} 
If we choose $0 < \beta \le \frac{r}{r+\kappa_2}$, then we can evaluate $\varphi_2(\underline{\eta}_1)$ as 
\begin{equation*}
\begin{array}{lcl}
\varphi_2(\underline{\eta}_1) &= &  \kappa_2\alpha L \Big( \frac{\beta - \sqrt{\beta^2 - 16(1+r)\mu L\rho}}{2(1+r)L} \Big)^2 - (1-\beta)\frac{\beta - \sqrt{\beta^2 - 16(1+r)\mu L\rho}}{2(1+r)L} + 2(1-\mu)\rho  \vspace{1ex}\\
&= &  \frac{[(1-\beta)r-\kappa_2\beta] [\sqrt{\beta^2(r+2\kappa_2)^2 - 16r(1+r)(r+2\kappa_2)L\rho} - \beta(r+2\kappa_2)]}{2r(1+r)(r+2\kappa_2)L} \vspace{1ex}\\
&\leq & 0,
\end{array}
\end{equation*}
provided that $L\rho \le \Delta \le \frac{\beta^2(r+2\kappa_2)}{16r(1+r)}$.
This shows that $\underline{\eta}_2 \leq \underline{\eta}_1$.

Alternatively, we can evaluate $\varphi_2(\bar{\eta}_1)$ as
\begin{equation*}
\begin{array}{lcl}
\varphi_2(\bar{\eta}_1) &= & \kappa_2\alpha L \Big( \frac{\beta + \sqrt{\beta^2 - 16(1+r)\mu L\rho}}{2(1+r)L} \Big)^2 - (1-\beta)\frac{\beta + \sqrt{\beta^2 - 16(1+r)\mu L\rho}}{2(1+r)L} + 2(1-\mu)\rho \vspace{1ex}\\
&= & \frac{[\kappa_2\beta-(1-\beta)r] [\sqrt{\beta^2(r+2\kappa_2)^2 - 16r(1+r)(r+2\kappa_2)L\rho} + \beta(r+2\kappa_2)]}{2r(1+r)(r+2\kappa_2)L} \vspace{1ex}\\
&\leq & 0,
\end{array}
\end{equation*}
provided that $L\rho \le \Delta \le \frac{\beta^2(r+2\kappa_2)}{16r(1+r)}$. 
This shows that $\bar{\eta}_1 \leq \bar{\eta}_2$.

Combining both cases, we have $\underline{\eta}_2 \leq \underline{\eta}_1 \leq \bar{\eta}_1 \leq \bar{\eta}_2$, and thus $\underline{\eta} \le \bar{\eta}$.
This proves (a).

\vspace{0.5ex}
\noindent\textbf{Case 2.} If we choose $\frac{r}{r+\kappa_2} \le \beta < 1$, then we can evaluate $\varphi_1(\underline{\eta}_2)$ as 
\begin{equation*}
\begin{array}{lcl}
\varphi_1(\underline{\eta}_2) &= & (1+r)L \Big( \frac{1-\beta  -  \sqrt{(1-\beta)^2 - 8\alpha(1-\mu)\kappa_2L\rho}}{2\alpha \kappa_2L} \Big)^2 - \beta \frac{1-\beta  -  \sqrt{(1-\beta)^2 - 8\alpha(1-\mu)\kappa_2L\rho}}{2\alpha \kappa_2L} + 4\mu\rho \vspace{1ex}\\
& = & \frac{[\kappa_2\beta - r(1-\beta)][\sqrt{r^2(1-\beta)^2(r+2\kappa_2)^2 - 16r(1+r)(r+2\kappa_2)\kappa_2^2L\rho} - r(1-\beta)(r+2\kappa_2)]}{2\kappa_2^2 (1+r)(r+2\kappa_2) L} \vspace{1ex}\\
&\leq & 0,
\end{array}
\end{equation*}
provided that $L\rho \le \Delta \le \frac{r(1-\beta)^2(r+2\kappa_2)}{16(1+r)\kappa_2^2}$. 
This shows that $\underline{\eta}_1 \leq \underline{\eta}_2$.

Similarly, we can evaluate $\varphi_1(\bar{\eta}_2)$ as 
\begin{equation*}
\begin{array}{lcl}
\varphi_1(\bar{\eta}_2) &= & (1+r)L \Big( \frac{1-\beta  +  \sqrt{(1-\beta)^2 - 8\alpha(1-\mu)\kappa_2L\rho}}{2\alpha \kappa_2L} \Big)^2 - \beta \frac{1-\beta  +  \sqrt{(1-\beta)^2 - 8\alpha(1-\mu)\kappa_2L\rho}}{2\alpha \kappa_2L} + 4\mu\rho \vspace{1ex}\\
&= & \frac{[r(1-\beta) - \kappa_2\beta][\sqrt{r^2(1-\beta)^2(r+2\kappa_2)^2 - 16r(1+r)(r+2\kappa_2)\kappa_2^2L\rho} + r(1-\beta)(r+2\kappa_2)]}{2\kappa_2^2 (1+r)(r+2\kappa_2) L} \vspace{1ex}\\
&\leq & 0,
\end{array}
\end{equation*}
provided that $L\rho \le \Delta \le \frac{r(1-\beta)^2(r+2\kappa_2)}{16(1+r)\kappa_2^2}$. 
This shows that $\bar{\eta}_2 \leq \bar{\eta}_1$.

Combining both cases, we can show that $\underline{\eta}_1 \le \underline{\eta}_2 \le \bar{\eta}_2 \le \bar{\eta}_1$, and thus $\underline{\eta} \le \bar{\eta}$.
This proves (b).
\Eproof
\end{proof}

\begin{lemma}\label{le:EG4NE_nonemptyness2}
Let $\underline{\eta}$ and $\bar{\eta}$ be defined by  \eqref{eq:EG4NE_eta_bounds}, and $\underline{\hat{\eta}}$ and $ \bar{\hat{\eta}}$ be defined by \eqref{eq:EG4NE_last_iterate_constant}.
Then, we have $[ \underline{\eta}, \bar{\eta}] \cap [\underline{\hat{\eta}}, \bar{\hat{\eta}}] \neq \emptyset$.
\end{lemma}

\begin{proof}
When $\beta = 1$, $\underline{\eta}$ and $\bar{\eta}$ in \eqref{eq:EG4NE_eta_bounds} become $\underline{\eta} = \frac{1-\sqrt{1 - 16(1+r)L\rho}}{2(1+r)L}$ and $\bar{\eta} = \frac{1+\sqrt{1 - 16(1+r)L\rho}}{2(1+r)L}$, respectively. 
Then, we consider the two following functions 
\begin{equation*}
\arraycolsep=0.2em
\begin{array}{l}
\phi_1(t) := \frac{1-\sqrt{1-16\rho t}}{2t} \quad  \text{and} \quad \phi_2(t) := \frac{1+\sqrt{1-16\rho t}}{2t}, \quad \text{for} \quad t \in \big( 0, \frac{1}{16\rho} \big).
\end{array}
\end{equation*}
On the one hand, we can see that $\phi'_1(t) = \frac{1-8\rho t - \sqrt{(1-8\rho t)^2 - 64\rho^2 t^2}}{2t^2\sqrt{1 - 16\rho t}} \ge 0$ for all $t \in \big( 0, \frac{1}{16\rho} \big)$, which implies that $\phi_1(t)$ is an increasing function on $\big( 0, \frac{1}{16\rho} \big)$.
On the other hand, it is obvious that $\phi_2(t)$ is a decreasing function on $\big( 0, \frac{1}{16\rho} \big)$ as it is the product of two positive decreasing functions. 
Thus if $1+r \le m$, by substituting $t := mL$ and $t := (1+r)L$ into $\phi_1(t)$ and $\phi_2(t)$, we have $\underline{\eta} \le \underline{\hat \eta} \le \bar{\hat{\eta}} \le \bar{\eta}$. 
Similarly, if $1+r \ge m$, we have $\underline{\hat \eta} \le \underline{\eta} \le \bar{\eta} \le \bar{\hat \eta}$. 
Combining both cases, we obtain $[ \underline{\eta}, \bar{\eta}] \cap [\underline{\hat{\eta}}, \bar{\hat{\eta}}] \neq \emptyset$.
\Eproof
\end{proof}

\bibliographystyle{plain}

\end{document}